\documentclass[a4paper,12pt]{amsart}
\usepackage[all]{xy}

\newtheorem{theo}{Theorem}[section]
\newtheorem{ex}[theo]{Example}
\newtheorem{fact}[theo]{Fact}
\newtheorem{facts}[theo]{Facts}
\newtheorem{exs}[theo]{Examples}
\newtheorem{prop}[theo]{Proposition}
\newtheorem{lem}[theo]{Lemma}
\newtheorem{cor}[theo]{Corollary}

\newtheorem{cons}[theo]{Construction}
\newtheorem{defi}[theo]{Definition}

\newtheorem{rema}[theo]{Remark}
\newtheorem{remas}[theo]{Remarks}

\renewcommand{\theprop} {\arabic{section}.\arabic{prop}}
\renewcommand{\thedefi} {\arabic{section}.\arabic{defi}}

\renewcommand{\therema} {\arabic{section}.\arabic{rema}}
\renewcommand{\theex} {\arabic{section}.\arabic{ex}}

\renewcommand{\theexs} {\arabic{section}.\arabic{exs}}
\renewcommand{\theconj} {\arabic{section}.\arabic{conj}}
\renewcommand{\thetheo} {\arabic{section}.\arabic{theo}}
\renewcommand{\thecor} {\arabic{section}.\arabic{cor}}
\renewcommand{\thelem} {\arabic{section}.\arabic{lem}}
\renewcommand{\thesection}{\arabic{section}}
\renewcommand{\thesubsection}{\arabic{section}.\arabic{subsection}}

\def \Romannumeral #1 {\expandafter\uppercase\expandafter {\romannumeral #1} }

\def \calo {{\mathcal O}}
\def \calu {{\mathcal U}}
\def \calv {{\mathcal V}}
\def \calf {{\mathcal F}}

\def \Z {{\Bbb Z}}
\def \Q {{\Bbb Q}}
\def \F {{\Bbb F}}

\def \C {{\Bbb C}}

\def \cok {{\rm{coker\,}}}
\def \im {{\rm {Im\,}}}
\def \G {{\bf G}_m}
\def \H {{\mathcal H}}
\def \A {{\bf A}}
\def \Het {H_{\mbox{\scriptsize\'et}}}
\def \Bdr {B_{\rm dR}}
\def \Bht {B_{\rm HT}}
\def \Ainf {A_{\rm inf}}
\def \Adr {A_{\rm dR}}
\def \Acris {A_{\rm cris}}

\def \limit{\mathop{\lim}}
\def \limproj{\limit\limits_{\leftarrow}}
\def \limind{\limit\limits_{\to}}

\def\smallsquare{\vbox{\hrule\hbox{\vrule height 1 ex\kern 1 ex\vrule}\hrule}}
\def\enddem{\hfill \smallsquare\vskip 3mm}

\usepackage{amscd}
\usepackage{amssymb}
\usepackage{amsmath}
\usepackage{pb-diagram}

\everymath{\displaystyle}
\def \Xan {X^{\rm an}}

\def \hyp {{\Bbb H}}

\newcommand{\Ker}{\mathrm{Ker}}
\newcommand{\Fil}{\mathrm{Fil}}

\newcommand{\Pic}{\mathrm{Pic}}
\newcommand{\Ext}{\mathrm{Ext}}
\newcommand{\Hom}{\mathrm{Hom}}

\newcommand{\Tor}{\mathrm{Tor}}

\newcommand{\Gal}{\mathrm{Gal}}
\newcommand{\id}{\mathrm{id}}

\newcommand{\Spec}{\mathrm{Spec\,}}

\newcommand{\OK}{\mathcal{O}_K}
\newcommand{\OKB}{{\mathcal{O}_{\overline{K}}}}
\newcommand{\OCK}{{\mathcal{O}_{\mathbb{C}_K}}}
\newcommand{\CO}{\mathcal{O}}
\newcommand{\Zp}{\mathbb{Z}_p}
\newcommand{\Fp}{\mathbb{F}_p}

\newcommand{\OL}{\mathcal{O}_L}

\DeclareFontFamily{U}{wncy}{}
\DeclareFontShape{U}{wncy}{m}{n}{%
   <5>wncyr5%
   <6>wncyr6%
   <7>wncyr7%
   <8>wncyr8%
   <9>wncyr9%
   <10>wncyr10%
   <11>wncyr10%
   <12>wncyr6%
   <14>wncyr7%
   <17>wncyr8%
   <20>wncyr10%
   <25>wncyr10}{}
\DeclareMathAlphabet{\cyrille}{U}{wncy}{m}{n}

\title{The $p$-adic Hodge decomposition according to Beilinson}
\author{Tam\'as Szamuely and Gergely Z\'abr\'adi}

\date{}
\begin{document}
\maketitle

\tableofcontents

\section{Introduction}

The Hodge decomposition of the cohomology of smooth projective complex varieties is a fundamental tool in the study of their geometry. Over an arbitrary base field of characteristic zero, \'etale cohomology with $\Q_p$-coefficients is a good substitute for singular cohomology with complex coefficients, but in general no analogue of the Hodge decomposition is known. However, owing to a fundamental insight of Tate \cite{tate}, we know that  over a $p$-adic base field a version of Hodge decomposition can indeed be constructed. Moreover, the cohomology groups involved carry an action of the Galois group of the base field, whose interaction with the Hodge decomposition can be analyzed by methods inspired by the study of the monodromy action in the complex case. This has deep consequences for the study of varieties of arithmetic interest, and can even be used to prove some purely geometric statements.

The first proof of the $p$-adic Hodge decomposition is due to Faltings \cite{faltings1}; several other proofs have been given since. One of the most recent is a wonderful proof by Beilinson \cite{Bei} which is the closest to geometry. It can be hoped that its groundbreaking new ideas will lead to important applications; some of them already appear in the recent construction of $p$-adic realizations of mixed motives by D\'eglise and Niziol \cite{degni}. Moreover, one of the key tools in Beilinson's approach is Illusie's theory \cite{I2} of the derived de Rham complex which has also reappeared during the recent development of derived algebraic geometry. Beilinson's work may thus also be viewed as a first bridge between this emerging field and $p$-adic Hodge theory.

In the present text we give a detailed presentation of Beilinson's approach, complemented by some further advances due to Bhatt \cite{Bh2}. Let us start by reviewing the complex situation which will serve as a guide to $p$-adic analogues.

\subsection{The Hodge decomposition over $\C$}

We begin by recalling some basic facts from complex Hodge theory; standard references are \cite{voisin} and \cite{demailly}.
Let $X$ be a smooth projective variety over $\C$ (or more generally a K\"ahler manifold). The Hodge decomposition is a direct sum decomposition for all $n\geq 0$
$$
H^n(\Xan,\C) =\bigoplus_{p+q=n} H^{p,q}
$$
where on the left hand side we have singular cohomology of the complex analytic manifold $\Xan$ and
$$
H^{p,q}\cong H^q(\Xan, \Omega_{\Xan}^p)
$$
with $\Omega_{\Xan}^p$ denoting the sheaf of holomorphic $p$-forms.

Furthermore, complex conjugation acts on $H^n(\Xan,\C)={H^n(\Xan,\Q)\otimes\C}$ via its action on $\C$, and we have
$$
H^{p,q}=\overline{ H^{q,p}}.
$$
These results are proven via identifying $H^{p,q}$ with Dolbeault cohomology groups and using the (deep) theory of harmonic forms on a K\"ahler manifold. However, part of the theory can be understood purely algebraically.

First, observe that $H^n(\Xan,\C)$ is also the cohomology of the constant sheaf $\C$ for the complex topology of $\Xan$. Consider the de Rham complex
$$
\Omega_{\Xan}^\bullet:=\calo_{\Xan}\stackrel d\to \Omega_{\Xan}^1\stackrel d\to \Omega_{\Xan}^2\to\dots
$$
Here the first $d$ is the usual derivation and the higher $d$'s are the unique ones satisfying
$$
d(\omega_1\wedge \omega_2)=d\omega_1\wedge \omega_2+(-1)^p\omega_1\wedge d\omega_2
$$
for $\omega_1\in \Omega^p_{\Xan}$ and $\omega_2\in \Omega^{p'}_{\Xan}$.

The holomorphic Poincar\'e lemma implies that $\Omega^\bullet_{\Xan}$ has trivial cohomologies over contractible open subsets except on the left where the kernel is $\C$. In other words, the augmented complex of analytic sheaves
$$
0\to\C\to\calo_{\Xan}\stackrel d\to \Omega_{\Xan}^1\stackrel d\to \Omega_{\Xan}^2\to\dots
$$
is exact. Thus we have an isomorphism of (hyper)cohomology groups
\begin{equation}\label{aniso}
H^n(\Xan,\C)\cong \hyp^n(\Xan, \Omega_{\Xan}^\bullet)=: H^n_{\rm dR}(\Xan).
\end{equation}

Now $\Omega_{\Xan}^\bullet$ has a descending filtration by subcomplexes
$$
\Omega_{\Xan}^{\geq p}:= 0\to\dots\to 0 \to\Omega_{\Xan}^p\stackrel d\to \Omega_{\Xan}^{p+1}\to\dots
$$
The $p$-th graded quotient is isomorphic to $ \Omega^p_{\Xan}$
(shifted by $p$), whence a spectral sequence (the Hodge to de Rham spectral sequence)
$$
E_1^{p,q}=H^q(\Xan,\Omega_{\Xan}^p)\Rightarrow H^{p+q}_{\rm dR}(\Xan)
$$
inducing a descending filtration
$$
H^n_{\rm dR}(\Xan)=F^0\supset F^1\supset\cdots\supset F^n\supset F^{n+1}=0
$$
on $H^n_{\rm dR}(\Xan)$, the Hodge filtration.\medskip

The first fundamental fact is that the Hodge to de Rham spectral sequence {\em degenerates at $E_1$}, giving rise to isomorphisms
$$
F^p/F^{p+1}\cong E_1^{p,q}=H^q(\Xan,\Omega_{\Xan}^p).
$$
Via the isomorphism (\ref{aniso}) the conjugation action on $\C$ induces an action on $H^n_{\rm dR}(\Xan)$. Setting
$$
\H^{p,q}:=F^p\cap\overline{F^q}
$$
we have obviously
$
\H^{p,q}=\overline{ \H^{q,p}}.
$

The second nontrivial fact is that the natural map
$$
\H^{p,q}\to F^p/F^{p+1}
$$
is an isomorphism and hence $\H^{p.q}\cong H^q(\Xan,\Omega_{\Xan}^p)$. In other words, $\H^{p,q}$ is a complement of $F^{p+1}$ in $F^p$, so {\em complex conjugation splits the Hodge filtration.}

However, the only known proof of this uses the Hodge decomposition we started with. Namely, one proves
$$
F^i=\bigoplus_{p+q=n, p\geq i}H^{p,q}
$$
whence of course we also get
$
H^{p,q}=\H^{p,q}.
$

\subsection{Algebraization}

On a complex algebraic variety $X$ one may also consider sheaves of algebraic differential forms $\Omega^p_X$ and the algebraic de Rham complex
$$
\Omega_{X}^\bullet:=\calo_{X}\stackrel d\to \Omega_{X}^1\stackrel d\to \Omega_{X}^2\to\dots
$$
which is a complex of coherent sheaves on $X$; for $X$ smooth they are moreover locally free. There is a Hodge to de Rham spectral sequence
$$
E_1^{p,q}=H^q(X,\Omega_{X}^p)\Rightarrow H^{p+q}_{\rm dR}(X).
$$
defined in the same way. Here we are using cohomology of coherent sheaves in the Zariski topology.

There are natural maps
$$
H^q(X,\Omega_{X}^p)\to H^q(\Xan,\Omega_{\Xan}^p), \qquad H^{p+q}_{\rm dR}(X)\to H^{p+q}_{\rm dR}(\Xan)
$$
compatible with the maps in the spectral sequence. By the GAGA theorem of Serre, for $X$ projective the first maps are isomorphisms, and hence so are the second ones (in fact, the maps on de Rham cohomology are isomorphisms for general smooth $X$ by a result of Grothendieck \cite{grothdr}). Thus degeneration for the analytic spectral sequence is equivalent to that of the algebraic spectral sequence. Indeed, there is a purely algebraic proof of the degeneration of the algebraic Hodge to de Rham spectral sequence due to Deligne and Illusie \cite{deligneillusie}.

However, there is no algebraic Poincar\'e lemma (so the algebraic de Rham complex is not a resolution of the constant sheaf $\C$), and anyway the Zariski cohomology of the constant sheaf $\C$ is trivial. However, comparison with the analytic results imply that the singular cohomology $H^n(\Xan,\C)$ has a Hodge decomposition involving {\em algebraic differential forms.}

The singular cohomology of $\Xan$ can also be defined algebraically for certain coefficients by means of \'etale cohomology. Indeed for $m>1$ we have a comparison isomorphism
$$
H^n(\Xan,\Z/m\Z)\cong \Het^n(X,\Z/m\Z)
$$
due to M. Artin, whence for a prime $p$
$$
H^n(\Xan,\Q_p)\cong \Het^n(X,\Q_p)
$$
where $\Het^n(X,\Q_p):=\limproj \Het^n(X,\Z/p^r\Z)\otimes_{\Z_p}\Q_p$.

But in general it does not compare with $H^{n}_{\rm dR}(X)$. The situation is better, however, over $p$-adic base fields.
\medskip

\subsection{The case of a $p$-adic base field}

Recall that $\C_p$ is the completion of an algebraic closure $\overline\Q_p$  of $\Q_p$. The Galois group $G:=\Gal(\overline \Q_p|\Q_p)$ acts on $\C_p$ by continuity. Similarly, if $K$ is a finite extension of $\Q_p$, by completing an algebraic closure of $K$ we obtain a complete valued field $\C_K$ with an action of $G_K:=\Gal(\overline \Q_p|K)$. Of course, as a field it is the same as $\C_p$ but it carries the action of a subgroup of $G$.

The $G_K$-action on $\overline \Q_p^\times$ induces a $G_K$-action on
$$
\Z_p(1):=\limproj \mu_{p^r}
$$
and hence on the tensor powers
$$
\Z_p(i):=\Z_p(1)^{\otimes i}.
$$
This can be extended to negative $i$ by setting $\Z_p(-i)$ to be the $\Z_p$-linear dual of $\Z_p(i)$ with its natural $G_K$-action. Finally, we have $G_K$-modules
$$
\C_K(i):=\C_K\otimes_{\Z_p}\Z_p(i)
$$
with $G_K$ acting via $\sigma(\lambda\otimes\omega)=\sigma(\lambda)\otimes\sigma(\omega)$. A famous theorem of Tate \cite{tate}, which was the starting point of $p$-adic Hodge theory, states that
\begin{equation}\label{tatetheo}
\C_K(i)^{G_K}=\begin{cases}
K & i=0 \\
0 & i\neq 0.
\end{cases}
\end{equation}

Now assume $X$ is a smooth projective $K$-variety. The Hodge--Tate decomposition, conjectured by Tate and first proven by Faltings \cite{faltings1}, is the following analogue of the Hodge decomposition over $\C$.

\begin{theo}\label{ht}
There is a canonical isomorphism
$$
\Het^n(X_{\overline K}, \Q_p)\otimes_{\Q_p}\C_K\cong \bigoplus_q H^{q}(X, \Omega_X^{n-q})\otimes_K \C_K(q-n)
$$
of $G_K$-modules.\end{theo}

Here $G_K$ acts on the left by the tensor product of its actions on $\Het^n(X_{\overline K}, \Q_p)$ and on $\C_K$ and on the right via its actions on the $\C_K(q-n)$ (so the $H^{q}(X, \Omega_X^{n-q})$ are equipped with the trivial $G_K$-action).

\begin{rema}\rm
The Hodge--Tate decomposition holds more generally for smooth varieties having a smooth projective normal crossing compactification, provided that one uses the de Rham complex with logarithmic poles along the divisor at infinity (see Subsection \ref{geomside} for definitions). It is in this generality that the theorem will be proven in the present text. The existence of the smooth projective normal crossing compactification is guaranteed by Hironaka's theorem for smooth quasi-projective $X$. It is also possible to extend the theorem to a statement about arbitrary varieties using hypercoverings.
\end{rema}

\begin{ex}\rm In the case $n=1$ we get
$$
\Het^1(X_{\overline K}, \Q_p)\otimes \C_K\cong (H^0(X, \Omega^1_X)\otimes \C_K(-1))\oplus (H^1(X, \calo_X)\otimes\C_K)
$$
or else
$$
\Het^1(X_{\overline K}, \Q_p(1))\otimes \C_K\cong (H^0(X, \Omega^1_X)\otimes \C_K)\oplus (H^1(X, \calo_X)\otimes\C_K(1)).
$$
Here
$$
\Het^1(X_{\overline K}, \Q_p(1))\cong T_p(\Pic X_{\overline K})\otimes_{\Z_p}\Q_p.
$$
In the case of an abelian variety this was first proven by Tate \cite{tate} in the good reduction case and by Raynaud in general, and then by Fontaine \cite{F}  by a different method. For abelian varieties this implies the Hodge-Tate decomposition for all $H^n$, as the (\'etale, Hodge or coherent) cohomology algebra of an abelian variety is the exterior algebra on $H^1$.\end{ex}

Theorem \ref{ht} can be reformulated as follows. Introduce the $\C_K$-algebra
$$
B_{\rm HT}:=\bigoplus_{i\in\Z} \C_K(i)
$$
where multiplication is given by the natural maps $\C_K(i)\otimes\C_K(j)\to \C_K(i+j)$. It carries a natural $G_K$-action. Also, define the $K$-algebra
$$
H^n_{\rm Hdg}(X):=\bigoplus_{q=0}^n H^{q}(X, \Omega_X^{n-q}).
$$
Both are graded algebras, so there is a grading on the tensor product ${H^n_{\rm Hdg}(X)\otimes_KB_{\rm HT}}$ given by the sum of grades. Thus tensoring the Hodge--Tate decomposition of Theorem \ref{ht} by $B_{HT}$ yields a $G_K$-equivariant isomorphism of graded $\C_K$-algebras
$$
\Het^n(X_{\overline K},\Q_p)\otimes_{\Q_p}B_{\rm HT}\cong H^n_{\rm Hdg}(X)\otimes_KB_{\rm HT}.
$$
Moreover, Tate's theorem (\ref{tatetheo}) implies
$$
(\Het^n(X_{\overline K},\Q_p)\otimes_{\Q_p}B_{\rm HT})^{G_K}\cong H^n_{\rm Hdg}(X).
$$
So we indeed recover the Hodge cohomology of $X$ from the \'etale cohomology using the Galois action. But what about the de Rham cohomology?

In his groundbreaking paper \cite{Fannals}, Fontaine defined a complete discrete valued field $\Bdr$ containing $\overline K$ that is equipped with a $G_K$-action and has a {\em non-split} decreasing $G_K$-equivariant valuation filtration $\Fil^i$ such that there are $G_K$-equivariant isomorphisms
$$
\Fil^i/\Fil^{i+1}\cong \C_K(i)
$$
for all $i\in\Z$. So the associated graded ring of $\Bdr$ with respect to $\Fil^i$ is $\Bht$, and by Tate's theorem we have
$
\Bdr^{G_K}=K.
$

We then have the following stronger statement, from which Theorem \ref{ht} results after passing to associated graded rings.

\begin{theo}\label{cdr} For all $n\geq 0$ there is a
$G_K$-equivariant isomorphism of filtered $\overline K$-algebras
$$
\Het^n(X_{\overline K},\Q_p)\otimes_{\Q_p}\Bdr\cong H^n_{\rm dR}(X)\otimes_K\Bdr.
$$
Here the filtration on the right hand side is the tensor product of the Hodge filtration $F^i$ and the filtration $\Fil^j$ on $\Bdr$.
\end{theo}

The equality $\Bdr^{G_K}=K$ implies that we indeed recover de Rham cohomology from \'etale cohomology:

\begin{cor}
For all $n\geq 0$ there is an isomorphism of filtered $K$-algebras
$$
(\Het^n(X_{\overline K},\Q_p)\otimes_{\Q_p}\Bdr)^{G_K}\cong H^n_{\rm dR}(X).
$$
\end{cor}
This was Fontaine's $C_{\rm dR}$ conjecture, again first proven by Faltings in his paper \cite{faltings2}; see also Illusie's Bourbaki report \cite{illbourbaki}. Fontaine has also made finer conjectures for smooth proper varieties with good reduction (the $C_{\rm cris}$ conjecture) and with semistable reduction (the $C_{\rm st}$ conjecture), involving other period rings $B_{\rm cris}$ and $B_{\rm st}$. Both conjectures imply the $C_{\rm dR}$ conjecture but in addition the groups in the comparison theorems carry extra structure. In the semi-stable case these are a semi-linear Frobenius and a monodromy operator, which together allow one to recover \'etale cohomnology from de Rham cohomology, not just the other way round. The $C_{\rm st}$ conjecture together with de Jong's alteration theorem also implies a $p$-adic analogue of Grothendieck's local monodromy theorem.

All these conjectures are now theorems thanks to work of Fontaine--Messing, Faltings, Hyodo, Kato, Tsuji, Niziol, Scholze and others; for the situation in 2009 complete with references and an in-depth account of Faltings's method, see Olsson's report \cite{olssonseattle}. We are concerned here with a recent beautiful approach by Beilinson \cite{Bei} that closely resembles the complex setting. We shall only discuss the de Rham comparison theorem, but Beilinson's method also yields proofs of the $C_{\rm cris}$ and $C_{\rm st}$ conjectures, as shown in work by Beilinson himself \cite{Bei2} as well as Bhatt \cite{Bh2}.

\subsection{Beilinson's method}

The first innovative ingredient in Beilinson's approach is a new construction of Fontaine's period ring that immediately shows its relation to de Rham theory. Namely, Beilinson considers
$$
A_{{\rm dR},K}:=L\widehat\Omega^\bullet_{\OKB/\OK}
$$
where on the right hand side we have the Hodge-completed derived de Rham algebra of Illusie \cite{I2}. The de Rham algebra $L\Omega^\bullet_{\OKB/\OK}$ itself is represented by a complex of $\OK$-modules equipped with a multiplicative structure and a descending filtration, i.e. a filtered differential graded algebra. It is constructed by choosing a free resolution of the $\OK$-algebra $\OKB$ and considering the de Rham complexes associated with each term in the resolution. Note that since $\OK$-algebras do not form an abelian category, the usual methods of homological algebra for constructing resolutions do not apply, and one has to use simplicial methods instead. The filtration is then induced by the Hodge filtration on the de Rham complexes.

Next, Beilinson considers the derived $p$-adic completion
$$
A_{{\rm dR},K}\widehat\otimes\Z_p:=R\limproj (A_{{\rm dR},K}\otimes^L\Z/p^r\Z).
$$
It turns out that the homology of this object is concentrated in degree 0, so it is a genuine filtered $\OK$-algebra.  Moreover, after tensoring with $\Q$ one obtains a complete discrete valuation ring that does not depend on $K$ any more and can be identified with Fontaine's ring $\Bdr^+$ which is the valuation ring of  $\Bdr$. The key point in this identification is Fontaine's calculation of the module of differentials $\Omega^1_{\OKB/\OK}$ in \cite{F}: it yields in particular a $G_K$-equivariant isomorphism $T_p(\Omega^1_{\OKB/\OK})\otimes\Q\cong\C_K(1).$

Beilinson's second main idea is to introduce a sheafification ${\mathcal A}^\natural_{\rm dR}$ of $\Adr:=A_{{\rm dR},K}$ for a certain Grothendieck topology that is fine enough to hope for an analogue of the Poincar\'e lemma. This is Voevodsky's $h$-topology \cite{sv} in which coverings are generated by \'etale surjective maps and {proper} surjective maps. The consideration of proper surjections is justified by an ingenious use of a theorem of Bhatt \cite{bhatt}. According to Bhatt's theorem, on a smooth variety every higher Zariski cohomology class of a coherent sheaf becomes $p$-divisible after passing to a suitable proper surjective covering; in particular, it vanishes after tensoring with $\Z/p^r\Z$. As a result, if one sheafifies the construction of the complexes $\Adr\otimes^L\Z/p^r\Z$ for the $h$-topology, they will have no higher cohomology over `small open sets'. This is Beilinson's $p$-adic version of the Poincar\'e lemma: the natural maps
$$
\Adr\otimes^L\Z/p^r\Z\to {\mathcal A}^\natural_{\rm dR}\otimes^L\Z/p^r\Z
$$
are filtered quasi-isomorphisms, where on the left hand side we have a constant $h$-sheaf. As a result, for a smooth $K$-variety $X$ we have  filtered isomorphisms
$$
H^n_{\mbox{\rm\scriptsize\'et}}(X_{\overline K},\Z_p)\otimes_{\Z_p}B_{\rm dR}^+\stackrel\sim\to H^n_h(X_{\overline K}, {\mathcal A}_{\rm dR}^\natural)\widehat\otimes\Q_p
$$
for all $n\geq 0$ that may be viewed as $p$-adic analogues of (\ref{aniso}). We call these the arithmetic side of the comparison isomorphism.

On the geometric side, one has to relate the right hand side of the above isomorphism to de Rham cohomology. This is accomplished by showing that  ${\mathcal A}^\natural_{\rm dR}\otimes\Q$ is none but the $h$-sheafification of the Hodge-completed (but non-derived) logarithmic de Rham complex.

There are several technical issues to be settled in order to make these ideas precise. First, the de Rham complexes under consideration only behave well for smooth schemes $U$ having a smooth normal crossing compactification $\overline U$. For these one has to work with logarithmic de Rham complexes and, in the arithmetic situation, log de Rham complexes of log schemes. Afterwards, $h$-sheafification causes a problem as the Zariski presheaves we want to sheafify are only defined for pairs $(U, \overline U)$ as above. Beilinson overcomes this difficulty by refining a general sheaf-theoretic result of Verdier that becomes applicable in our situation thanks to de Jong's alteration theorems. Finally, there is a complication of homological nature caused by the fact that we want to sheafify filtered objects in a derived category. Beilinson handles it by using the theory of $E_\infty$-algebras; here we follow the more pedestrian approach of Illusie \cite{illsurvey} that uses canonical Godement resolutions.

Once the comparison map between $\Het^n(X_{\overline K},\Q_p)\otimes_{\Q_p}\Bdr$ and $H^n_{\rm dR}(X)\otimes_K\Bdr$ has been constructed, the key computation is to verify that it is an isomorphism in the case $X={\Bbb G}_m$. Afterwards, the general case follows by formal cohomological arguments already present in the work of Faltings and Fontaine--Messing.

\subsection{Overview of the present text}

In the first chapter we give a reasonably complete introduction to Illusie's theory \cite{I1} of the cotangent complex and the derived de Rham algebra. The construction of these objects relies on simplicial methods which are usually not part of the toolkit of algebraic geometers and number theorists (such as yours truly). We have therefore summarized the results we need in an appendix.

Next, we present Fontaine's computation of the module of differentials for the $p$-adic ring extension $\OKB|\OK$ with simplifications due to Beilinson. This then serves for the computation of the $p$-completed derived de Rham algebra of the above ring extension, for which we use techniques from Bhatt's paper \cite{Bh2}. We emphasize throughout the role played by deformation problems in these constructions, culminating in a description of the $p$-completed derived de Rham algebra of $\OKB|\OK$ as a solution of a certain universal deformation problem. This ties in with Fontaine's approach in \cite{F2} to period rings via deformation problems, with the notable difference that he constructs universal deformation rings `by hand', whereas here, to use a somewhat dangerous formulation, we derive them from derived de Rham theory. This approach also makes it possible to prove directly that $\Bdr^+$ as constructed via Beilinson's method is a complete discrete valuation ring with the required properties, whereas he himself proceeds by comparison with Fontaine's constructions. A subtle point deserves to be mentioned here: as already noticed by Illusie in his thesis \cite{I2}, the $p$-completed derived de Rham algebras under consideration come equipped with a divided power structure. This structure enters calculations in a crucial way but then gets killed when one inverts the prime $p$ to obtain the ring $\Bdr^+$. This indicates that the $p$-completed derived de Rham algebra is `really' related to the crystalline theory, as confirmed by Bhatt's construction of the period ring $A_{\rm cris}$ that we also briefly review in the text.

The following chapter presents Beilinson's construction of the comparison map. We have separated the geometric side of the construction from the arithmetic side, as already outlined in the survey above. The geometric side does not use the derived de Rham complex or logarithmic geometry but the $h$-sheafification process already enters the game, in a somewhat simpler setting than in the arithmetic situation. It should be pointed out here that the comparison with classical de Rham cohomology uses complex Hodge theory. Presented this way, the arithmetic side of the construction becomes a logarithmic variant of the geometric one over a $p$-adic integral base, relying heavily on the Olsson--Gabber theory \cite{olsson} of the logarithmic cotangent complex. We only give a brief summary of the results of \cite{olsson}, but we hope that the reader will take on faith that the exact analogues of non-logarithmic results hold in this setting.

Of course, the arithmetic side of the comparison map has another non-trivial input besides those just mentioned: Beilinson's $p$-adic Poincar\'e lemma. We give its proof in the last chapter. However,  the key geometric result inspired by Bhatt \cite{bhatt} is only presented in a special case (due to Bhatt himself) where the argument is more transparent. The last chapter also contains a verification that the comparison map is an isomorphism.

This text grew out of a study seminar organized by the authors at the R\'enyi Institute during the academic year 2014/15, and was the basis of seminars at Universit\"at Duisburg-Essen and Oxford University in 2016. We thank all participants for their contribution. We are also indebted to Bhargav Bhatt, Luc Illusie and Marc Levine for enlightening discussions and to Alexander Beilinson for his kind comments on a preliminary version. We are grateful to the editors of the 2015 AMS Summer Institute proceedings for their kind interest in our text and to the referee whose suggestions have considerably improved it. The first author was partially supported by NKFI grant No. K112735, the second author by NKFI grant No. K100291 and by a Bolyai Scholarship of the Hungarian Academy of Sciences.

\section{The cotangent complex and the derived de Rham algebra}

\subsection{The cotangent complex of a ring homomorphism}

In this section and the next we give a quick introduction to Illusie's cotangent complex in the affine case. To begin with, we summarize basic properties of differential forms for the sake of reference.

\begin{facts}\label{diff}\rm Let $A \to B$ be a homomorphism of rings, and $M$ a $B$-module.
An {\em $A$-derivation} of $B$ in $M$ is an $A$-linear map  $D: B \to M$ satisfying the Leibniz rule $D(b_1b_2)={b_1D(b_2)+ b_1 D(b_2)}$ for all $b_1,b_2 \in B$. We denote the set of $A$-derivations $B\to M$ by ${\rm Der_A}(B,M)$; it carries a natural $B$-module structure with scalar multiplication given by $(bD)(x)=b\cdot D(x)$ for all $b\in B$.

The functor $M\mapsto {\rm Der}_A(B,M)$ on the category of $B$-modules is representable by a $B$-module $\Omega^1_{B/A}$, the {\em module of relative differentials}. A presentation of $\Omega^1_{B/A}$ is given by generators $db$ for each $b\in B$ subject to the relations $d(a_1b_1+a_2b_2)-a_1db_1-a_2db_2$ and $d(b_1b_2)-b_1db_2-b_2db_1$ for $a_i\in A$ and $b_i\in B$. It satisfies the following basic properties:
\begin{enumerate}
\item (Base change) For an $A$-algebra $A'$ one has ${\Omega^1_{B \otimes_A A'/A'}\cong \Omega^1_{B/A} \otimes_A A'}$.
\item (Localization) Given
 a multiplicative subset $S$ of $B$, one has
 $${\Omega^1_{B_S/A}\cong \Omega^1_{B/A} \otimes_B B_S}.$$
\item (First exact sequence) A sequence of ring homomorphisms $A\to B\to C$ gives rise to an exact sequence of $C$-modules
\begin{equation*}\label{firstexact}
C\otimes_B\Omega^1_{B/A}\to \Omega^1_{C/A}\to\Omega^1_{C/B}\to 0.
\end{equation*}
\item (Second exact sequence) A surjective morphism $B\to C$ of $A$-algebras with kernel $I$ gives rise to an exact sequence
\begin{equation*}\label{secondexact}
I/I^2\stackrel\delta\to C\otimes_B\Omega^1_{B/A}\to\Omega^1_{C/A}\to 0
\end{equation*}
of $C$-modules, where the map $\delta$ sends a class $x$ mod $I^2$ to $1\otimes dx$. (Note that the $B$-module structure on $I/I^2$ induces a $C$-module structure.)
\end{enumerate}
For all these facts, see e.g. \cite{M}, \S 25.
\end{facts}

Exact sequence (3) above can be extended by 0 on the left under a smoothness assumption on the map $B\to C$. However, in general exactness on the left fails. One of the main motivations for introducing the cotangent complex $L_{B/A}$ is to remedy this defect. To construct $L_{B/A}$, we use the simplicial techniques from Subsection \ref{simpmeth} of the Appendix.

\begin{defi}\rm
Let $A$ be a ring. We call an augmented simplicial object $Q_\bullet\to B$ in the category of $A$-algebras a {\em simplicial resolution} if it induces a simplicial resolution on underlying $A$-modules in the sense of Definition \ref{simpres}.
\end{defi}

Note that the category of $A$-algebras is {\em not} an abelian category, and therefore Definition \ref{simpres} does not apply directly.

\begin{cons}\label{standardres}\rm
We define the {\em standard simplicial resolution} $P_\bullet=P_\bullet(B)$ of the $A$-algebra $B$ as follows. Set $P_0:=A[B]$, the  free $A$-algebra on generators $x_b$ indexed by the elements of $B$; then define inductively $$P_{i+1}:=A[P_i]$$ for $i\geq 0$.

We turn the sequence of the $P_i$ into a simplicial $A$-algebra as follows. Note first that given an $A$-algebra $B$, its identity map induces an $A$-algebra homomorphism $\kappa_B:\,A[B]\to B$, and also a map of sets $\tau_B:\,B\to A[B]$ in the other direction. Whence for $0\leq j\leq i$ face maps
$$
\partial^j_i:\,P_i=\underbrace{A[A[\dots [B]]\dots ]}_i \to P_{i-1}=\underbrace{A[A[\dots [B]]\dots ]}_{i-1}
$$
induced by applying $\kappa_{A[P_j]}$, and degeneracy maps
$$
\sigma^j_i:\,P_{i-1}=\underbrace{A[A[\dots [B]]\dots ]}_{i-1} \to P_{i}=\underbrace{A[A[\dots [B]]\dots ]}_{i}
$$
induced by applying $\tau_{A[P_j]}$. Direct computation shows that this defines a simplicial resolution of the $A$-algebra $B$; this fact may also be deduced from the general categorical result of (\cite{weibel}, Proposition 8.6.8).

For later use, note that in a similar fashion we obtain a standard simplicial resolution for an $A$-module $M$, by iterating the functor associating with $M$ the free $A$-module with basis the underlying set of $M$. Finally, the construction may be carried out for simplicial algebras $B_\bullet$ (or modules $M_\bullet$) over a simplicial ring $A_\bullet$: it yields a bisimplicial object whose associated double complex gives a free resolution in each column.
\end{cons}

The standard resolution has the following important property.

\begin{lem}\label{standardtriv} Assume $B=A[X]$ is a free algebra on the generating set $X$.
Then the standard simplicial resolution $P_\bullet(A[X])\to A[X]_\bullet$ defined above is a homotopy equivalence.
\end{lem}

Here $A[X]_\bullet$ denotes the constant simplicial object associated with $A[X]$, as in Definition \ref{constsimp} of the Appendix.

\begin{proof} Define $f_\bullet:\, A[X]_\bullet\to P_\bullet(A[X])$ and $g_\bullet:\, P_\bullet(A[X])\to A[X]_\bullet$ by iterating the operations $\tau_B$ and $\kappa_B$:
\begin{eqnarray*}
f_n:= &\underbrace{\tau\circ\dots\circ\tau}_{n+1},\\
g_n:= &\underbrace{\kappa\circ\dots\circ\kappa}_{n+1}.
\end{eqnarray*}
These indeed define morphisms of simplicial objects, and by construction we have $g_\bullet\circ f_\bullet={\id}_{A[X]_\bullet}$. We define a simplicial homotopy between $f_\bullet\circ g_\bullet$ and ${\id}_{P_\bullet(A[X])}$ as follows. For $\alpha_i\colon [n]\to [1]$ ($i=-1,0,\dots,n$) with $\alpha_i^{-1}(0)=\{0,\dots,i\}$ we put
\begin{equation*}
\xymatrixcolsep{8pc}\xymatrix{H_{\alpha_i}\colon P_n \ar[r]^{\underbrace{\tau\circ\dots\circ\tau}_{n-i}} & P_i\ar[r]^{\underbrace{\kappa\circ\dots\circ\kappa}_{n-i}} &  P_n.
}
\end{equation*}
Taking the sum of these maps over all $\alpha_i$ defines a simplicial homotopy $$H:\, {P_\bullet(A[X])\times \Delta[1]_\bullet}\to P_\bullet(A[X])$$ between $f_\bullet\circ g_\bullet$ and ${\id}_{P_\bullet(A[X])}$.
\end{proof}

Now we come to the fundamental definition of Illusie \cite{I1}.

\begin{defi}\rm
Consider an $A$-algebra $B$, and take the standard resolution $P_\bullet\to B$. The \emph{cotangent complex} $L_{B/A}$ of the $A$-algebra $B$ is defined as the complex of $B$-modules $$L_{B/A}:=C(B_\bullet\otimes_{P_\bullet}\Omega^1_{P_\bullet/A}).$$
\end{defi}

Like in the previous lemma, here $B_\bullet$ stands for the constant simplicial ring associated with $B$ (see Definition \ref{constsimp}). It is a simplicial $P_\bullet$-algebra via the augmentation map $P_\bullet\to B_\bullet$ (see Definition \ref{augment}). The simplicial $A$-module $\Omega^1_{P_\bullet/A}$ is obtained by applying the functor $B\to \Omega^1_{B/A}$ to the terms of the resolution $P_\bullet$, and $C$ denotes the associated chain complex.

The cotangent complex is related to the module of differentials as follows.

\begin{prop}\label{h0cotang}
We have a natural isomorphism of $B$-modules $$H_0(L_{B/A})\cong \Omega^1_{B/A}.$$
\end{prop}

\begin{proof} Since $\epsilon_\bullet:\, P_\bullet\to B_\bullet$ is an augmentation for the simplicial object $P_\bullet$, we have $\epsilon_0 d_0=\epsilon_0 d_1$. Therefore the composed map in the associated chain complex
$$B\otimes_{P_1}\Omega^1_{P_1/A}\to B\otimes_{P_0}\Omega^1_{P_0/A}\to \Omega^1_{B/A}$$
is the zero map. Thus we have a morphism of complexes $L_{B/A}\to \Omega^1_{B/A}$, with $\Omega^1_{B/A}$ considered as a complex concentrated in degree 0. The induced map $H_0 (L_{B/A})\to \Omega^1_{B/A}$ is surjective because so is the ring homomorphism $P_0\to B$. By Fact \ref{diff} (4) we have an exact sequence $$I/I^2\to B\otimes_{P_0}\Omega^1_{P_0/A}\to \Omega^1_{B/A},$$  where $I$ is the kernel of the augmentation map $\epsilon_0\colon P_0\to B$. But since $P_\bullet\to B$ is a resolution, here $I$ is also the image of the map $d_0-d_1\colon P_1\to P_0$, and therefore the image of $I/I^2$ in $B\otimes_{P_0}\Omega^1_{P_0/A}$ is covered by $B\otimes_{P_1}\Omega^1_{P_1/A}$, as desired.
\end{proof}

We now show that the cotangent complex may be calculated by other free resolutions as well.

\begin{theo}\label{freeresolcotang}
Let $Q_\bullet\to B$ be a simplicial resolution of the $A$-algebra $B$ whose terms are free $A$-algebras. We have a quasi-isomorphism $$L_{B/A}\cong C(B_\bullet\otimes_{Q_\bullet}\Omega^1_{Q_\bullet/A})$$ of complexes of $B$-modules.
\end{theo}

The proof will be in several steps. We begin with a general lemma that will serve in other contexts as well.

\begin{lem}\label{flattensorquasi}
Let $A_\bullet$ be a simplicial ring, and let  $ E_\bullet\to F_\bullet$ be a morphism of $A_\bullet$-modules that induces a quasi-isomorphism on associated chain complexes. Tensoring by an $A_\bullet$-module $L_\bullet$ that is termwise flat over $A_\bullet$ yields a map $E_\bullet\otimes_{A_\bullet} L_\bullet\to F_\bullet\otimes_{A_\bullet} L_\bullet$ that also induces a quasi-isomorphism.
\end{lem}
\begin{proof} Assume first that  $A_\bullet$ is a constant simplicial ring defined by a ring $A$. In this case the lemma is a consequence of the K\"unneth formula applied to the tensor products of the associated complexes of $A$-modules.
In the general case consider the standard simplicial resolution  ${F(L_n)}_\bullet\to L_n$ of each $A_n$-module $L_n$. These assemble to a bisimplicial object ${F(L_\bullet)}_\bullet$ equipped with a map ${F(L_\bullet)}_\bullet\to L_\bullet$. Moreover, we have a commutative square of bisimplicial objects
\begin{equation*}
\xymatrix{
E_\bullet \otimes_{A_\bullet} F(L_\bullet)_\bullet  \ar[r]\ar[d] & F_\bullet \otimes_{A_\bullet} F(L_\bullet)_\bullet\ar[d] \\
E_\bullet\otimes_{A_\bullet} L_\bullet\ar[r] & F_\bullet\otimes_{A_\bullet} L_\bullet
}
\end{equation*}
viewing the simplicial objects in the lower row as `constant bisimplicial objects'. For fixed $n\geq 0$ the vertical maps $E_n\otimes_{A_n} F(L_n)_\bullet\to E_n\otimes_{A_n}L_n$ and $F_n\otimes_{A_n} F(L_n)_\bullet\to F_n\otimes_{A_n}L_n$ are quasi-isomorphisms because $F(L_n)_\bullet\to L_n$ is a flat resolution of the flat $A_n$-module $L_n$.  It follows that both vertical arrows induce quasi-isomorphisms on total chain complexes, and therefore it suffices to verify the same for the upper horizontal arrow.
By construction of the standard resolution, for fixed $m,n\geq 0$ the $A_m$-module $F(L_n)_m$ is isomorphic to the  free $A_n$-module $A_n^{(X_{n,m})}$ with basis a set $X_{n,m}$. Denoting by $\Z^{(X_{n,m})}$ the similarly constructed free $\Z$-module, we thus have isomorphisms of simplicial modules  $E_\bullet\otimes_{A_\bullet}F(L_\bullet)_m\cong E_\bullet\otimes_{\Z_\bullet}\Z^{(X_{\bullet,m})}$ for each $m$, where $\Z_\bullet$ is the constant simplicial ring defined by $\Z$. The same holds for $F_\bullet$ in place of $E_\bullet$, and therefore by the case of a constant base ring the maps $E_\bullet\otimes_{A_\bullet}F(L_\bullet)_m\to F_\bullet\otimes_{A_\bullet}F(L_\bullet)_m$ induce quasi-isomorphisms for all $m$. This gives a quasi-isomorphism on total complexes, as required.
\end{proof}

\begin{cor}\label{corflattens}
If $Q_\bullet\to B_\bullet$ is a simplicial resolution of the $A$-algebra $B$ with free terms, we have a quasi-isomorphism of complexes
$$
C\Omega^1_{Q_\bullet/A}\cong C(B_\bullet\otimes_{Q_\bullet}\Omega^1_{Q_\bullet/A}).
$$
\end{cor}

\begin{proof}
Apply the lemma with $A_\bullet=E_\bullet=Q_\bullet$, $F_\bullet=B_\bullet$ and $L_\bullet=\Omega^1_{Q_\bullet/A}$.
\end{proof}

Next, assume given a simplicial $A$-algebra $B_\bullet$. The standard resolutions $P_\bullet(B_n)\to B_n$ for each $n$ assemble to a bisimplicial $A$-algebra $P_\bullet(B_\bullet)$. Applying the functor $\Omega^1_{\cdot/A}$ yields a bisimplicial $A$-module $\Omega^1_{P_\bullet(B_\bullet)/A}$, whence an associated double complex $C\Omega^1_{P_\bullet(B_\bullet)/A}$ and finally a total complex ${\rm Tot}(C\Omega^1_{P_\bullet(B_\bullet)/A})$, taken with the direct sum convention.

\begin{lem}\label{totcompiso}
 Let $C_\bullet\to B$ be a simplicial resolution of $A$-algebras. The induced map
$$
{\rm Tot}(C\Omega^1_{P_\bullet(C_\bullet)/A})\to C\Omega^1_{P_\bullet(B)/A}
$$
is a quasi-isomorphism.
\end{lem}

\begin{proof}
By Propositions \ref{kanhomotopy} and \ref{kanabelian} the underlying morphism $C_\bullet\to B_\bullet$ of simplicial sets induces a homotopy equivalence. Applying the functor $X\mapsto P_{n-1}(A[X])$ for fixed $n\geq 0$ (with the convention $P_{-1}=\id$) we obtain a homotopy equivalence $P_n(C_\bullet)\to P_n(B_\bullet)$ of simplicial $A$-algebras, whence a homotopy equivalence $\Omega^1_{P_n(C_\bullet)/A}\to \Omega^1_{P_n(B_\bullet)/A}$ of simplicial $A$-modules. As the latter is a constant simplicial module, it follows that $\Omega^1_{P_n(C_\bullet)/A}\to \Omega^1_{P_n(B)/A}$ is a simplicial resolution for each $n$. Thus the columns of the double complex $C\Omega^1_{P_\bullet(C_\bullet)/A}$ give free resolutions of the terms of the complex $C\Omega^1_{P_\bullet(B)/A}$, and therefore the total complex is indeed quasi-isomorphic to $C\Omega^1_{P_\bullet(B)/A}$.
\end{proof}

\noindent{\em Proof of Theorem \ref{freeresolcotang}.}
Applying the previous lemma to the simplicial resolution $Q_\bullet\to B$ yields a quasi-isomorphism
$$
{\rm Tot}(C\Omega^1_{P_\bullet(Q_\bullet)/A})\simeq C\Omega^1_{P_\bullet(B)/A}\simeq L_{B/A}
$$
using Corollary \ref{corflattens}.

On the other hand, for each fixed $n$ the simplicial map $P_\bullet(Q_n)\to {(Q_n)}_\bullet$ is a homotopy equivalence by Lemma \ref{standardtriv}, and therefore so is $\Omega^1_{P_\bullet(Q_n)/A}\to {(\Omega^1_{Q_n/A})}_\bullet$, so that $C\Omega^1_{P_\bullet(Q_n)/A}$ is an acyclic resolution of $\Omega^1_{Q_n/A}$. It follows that we have a quasi-isomorphism
$$
{\rm Tot}(C\Omega^1_{P_\bullet(Q_\bullet)/A})\simeq C\Omega^1_{Q_\bullet/A}
$$
which concludes the proof, again taking Corollary \ref{corflattens} into account.
\enddem

\begin{rema}\rm
In the model category of simplicial modules defined by Quillen  \cite{HA} the cofibrant replacements of an object correspond to projective resolutions of modules. In the model category structure on simplicial algebras (see \cite{HA} or \cite{quillennotes}) the simplicial resolutions considered in Theorem \ref{freeresolcotang} will not necessarily be cofibrant replacements. However, one may obtain cofibrant replacements by imposing an extra simplicial coherence condition. The resulting simplicial resolutions will be homotopy equivalent as simplicial algebras, whereas the ones in \ref{freeresolcotang} are only homotopy equivalent as simplicial sets. See also \cite{goerss} on these issues.
\end{rema}

If $A'$ is another $A$-algebra, we have a natural base change morphism
$$A'\otimes^L_A L_{B/A}\to L_{A'\otimes_AB/A'},$$
noting that the map $B\to A'\otimes_AB$ naturally extends to an $A_\bullet$-algebra map of corresponding standard simplicial resolutions.

Before stating the next lemma, recall that two $A$-algebras $A'$ and $B$ are called $\Tor$-independent if $\Tor_i^A(A',B)=0$ for $i>0$. If $A'$ is flat over $A$, then $A'$ and any $B$ are $\Tor$-independent.

\begin{lem}\label{cotangflatbasechange}
If  $A'$ and $B$ are $\Tor$-independent $A$-algebras,  the base change map induces a quasi-isomorphism $$A'\otimes_A^{L} L_{B/A}\overset{\sim}{\to}L_{A'\otimes_A B/A'}$$ of complexes of $A'\otimes B$-modules.
\end{lem}
\begin{proof}
Let $P_\bullet\to B$ be a standard simplicial resolution of the $A$-algebra $B$. Since $A'$ and $B$ are $\Tor$-independent, the associated chain complex of $A'\otimes_A P_\bullet$ is acyclic outside degree $0$ where its homology is $A'\otimes_A B$. In particular, $A'\otimes_A P_\bullet$ is a free simplicial resolution of the $A'$-algebra $A'\otimes_A B$ and hence may be used to compute $L_{A'\otimes_A B/A'}$ by Theorem \ref{freeresolcotang}. Finally, note that $$(A'\otimes_A B)\otimes_{A'\otimes_A P_\bullet}(\Omega^1_{A'\otimes_A P_\bullet/A})\cong A'\otimes_{A}(B\otimes_{P_\bullet}\Omega^1_{P_\bullet/A})$$ computes $A'\otimes_A^{L} L_{B/A}$ using again the $\Tor$-independence of $A'$ and $B$, noting that $L_{B/A}$ is a complex of free $B$-modules.
\end{proof}

We now come to one of the most important properties of the cotangent complex.

\begin{theo}[Transitivity triangle]\label{cotangexacttriang}
A sequence $A\to B\to C$ of ring maps induces an exact triangle in the derived category of complexes of $C$-modules
\begin{equation*}
C\otimes_B^{L} L_{B/A}\to L_{C/A}\to L_{C/B}\to C\otimes_B^{L} L_{B/A}[1]\ .
\end{equation*}
\end{theo}

\begin{proof}
 Let $P_\bullet\to B_\bullet$ be the standard resolution of the $A$-algebra $B$, and consider the constant simplicial module $C_\bullet$ as a $P_\bullet$-module via the composite homomorphism $P_\bullet\to B_\bullet\to C_\bullet$. The standard simplicial resolutions of each $C_n$ as a $P_n$-algebra assemble to a bisimplicial $A$-algebra $Q_{\bullet\bullet}$. The diagonal $Q_\bullet^\Delta$ of $Q_{\bullet\bullet}$ is a free $P_\bullet$-algebra in each degree, therefore the first exact sequence of differentials induces for each $n\geq 0$ a short exact sequence
\begin{equation*}
0\to Q_n^\Delta\otimes_{P_n}\Omega^1_{P_n/A}\to \Omega^1_{Q_n^\Delta/A}\to \Omega^1_{Q_n^\Delta/P_n}\to 0
\end{equation*}
of $Q_n^\Delta$-modules which splits since $\Omega^1_{Q_n^\Delta/P_n}$ is a free module. Tensoring with $C$ then gives rise to a short exact sequence
\begin{equation}\label{disttriangkeyexact}
0\to C_\bullet\otimes_{P_\bullet}\Omega^1_{P_\bullet/A}\to C_\bullet\otimes_{Q_\bullet^\Delta}\Omega^1_{Q_\bullet^\Delta/A}\to C_\bullet\otimes_{Q_\bullet^\Delta}\Omega^1_{Q_\bullet^\Delta/P_\bullet}\to 0
\end{equation}
of simplicial $C$-modules. We now show that after taking associated chain complexes this sequence represents the exact triangle of the theorem in the derived category.

The complex
$C_\bullet\otimes_{P_\bullet}\Omega^1_{P_\bullet/A}$ represents $C\otimes_B^{L} L_{B/A}$ as the simplicial $B$-module  ${B_\bullet\otimes_{P_\bullet}\Omega^1_{P_\bullet/A}}$ has free terms and the map
$P_\bullet\to C_\bullet$ factors through $B_\bullet$ by construction. Next, note that each term of $Q_\bullet^\Delta$ is free as an $A$-algebra, the $P_n$ being free over $A$ and the $Q_n^\Delta$  free over $P_n$. On the other hand, since the total complex $CQ_{\bullet\bullet}$ is acyclic by construction, the Eilenberg--Zilber Theorem (Theorem\ \ref{eilenbergzilber}) implies that $Q_\bullet^\Delta$ is a free simplicial resolution of the $A$-algebra $C$. Theorem \ \ref{freeresolcotang} then yields that the associated chain complex of $C_\bullet\otimes_{Q_\bullet^\Delta}\Omega^1_{Q_\bullet^\Delta/A}$ represents the cotangent complex $L_{C/A}$.

Finally, put $\overline{Q}_\bullet:=B_\bullet\otimes_{P_\bullet} Q_\bullet^\Delta$. Since each $Q_n^\Delta$ is free over $P_n$, Lemma \ref{flattensorquasi} applied to the resolution $P_\bullet\to B$ implies that the map $Q_\bullet^\Delta\to \overline{Q}_\bullet$ induces a quasi-isomorphism on normalized complexes. Since $Q_\bullet^\Delta$ is a free simplicial resolution of $C$ as an $A$-algebra, so is $\overline{Q}_\bullet$ as a $B$-algebra. The base change property of differentials implies that we have an isomorphism of simplicial $C$-modules $C_\bullet\otimes_{Q_\bullet^\Delta}\Omega^1_{Q_\bullet^\Delta/P_\bullet}\cong C_\bullet\otimes_{\overline{Q}_\bullet}\Omega^1_{\overline{Q}_\bullet/B}$, so Theorem \ \ref{freeresolcotang} yields that the associated chain complex of $C\otimes_{Q_\bullet^\Delta}\Omega^1_{Q_\bullet^\Delta/P_\bullet}$ represents the cotangent complex $L_{C/B}$.
\end{proof}

Theorem \ref{cotangexacttriang} and Lemma \ref{h0cotang} now imply:

\begin{cor}
In the situation of the theorem there is a long exact homology sequence
\begin{equation*}
\dots\to H_1(L_{B/A}\otimes^L_{B}C)\to H_1(L_{C/A})\to H_1(L_{C/B})\to \Omega^1_{B/A}\otimes_B C\to\Omega^1_{C/A}\to \Omega^1_{C/B}\to 0.
\end{equation*}
\end{cor}

We close this subsection by computing the cotangent complex in important special cases.

\begin{prop}\label{polcotangacyc}
If $B=A[X]$ is a free algebra on a set $X$ of generators, then the cotangent complex $L_{B/A}$ is acyclic in nonzero degrees.
\end{prop}
\begin{proof} By Lemma \ref{standardtriv} we have a homotopy equivalence between the constant simplicial algebra $A[X]_\bullet$ and its standard resolution $P_\bullet$. Applying the functor $\Omega^1_{\cdot/A}$ gives a homotopy equivalence between $\Omega^1_{P_\bullet/A}$ and $\Omega^1_{A[X]_\bullet/A}$, whence a quasi-isomorphism on associated chain complexes. But $C\Omega^1_{A[X]_\bullet/A}$ is a complex of free modules that is acyclic in nonzero degrees, so we conclude by Corollary \ref{corflattens}.
\end{proof}

The following case will be crucial for the calculations in the next section.

\begin{prop}\label{cotangquotientreg}
Assume that $A\to B$ is a surjective ring homomorphism with kernel $I=(f)$ generated by a nonzerodivisor $f\in A$. Then $L_{B/A}$ is quasi-isomorphic to the complex $I/I^2[1]$.
\end{prop}

\begin{proof}
We first treat the special case $A=\mathbb{Z}[x]$, $B=\mathbb{Z}$, $f=x$. Consider the exact triangle
\begin{equation*}
L_{\mathbb{Z}[x]/\mathbb{Z}}\otimes^L_{\mathbb{Z}[x]}\mathbb{Z}\to L_{\mathbb{Z}/\mathbb{Z}}\to L_{\mathbb{Z}/\mathbb{Z}[x]}\to L_{\mathbb{Z}[x]/\mathbb{Z}}\otimes^L_{\mathbb{Z}[x]}\mathbb{Z}[1]
\end{equation*}
associated by Theorem \ref{cotangexacttriang} to the sequence of ring maps $\mathbb{Z}\to \mathbb{Z}[x]\to\mathbb{Z}$. Lemma \ref{h0cotang} and Proposition \ref{polcotangacyc} imply that $L_{\mathbb{Z}/\mathbb{Z}}$ is acyclic and $L_{\mathbb{Z}[x]/\mathbb{Z}}$ is quasi-isomorphic to $\Omega^1_{\mathbb{Z}[x]/\mathbb{Z}}$ placed in degree 0. As the latter is a free module of rank 1, tensoring with $\mathbb{Z}$ over $\mathbb{Z}[x]$ yields that $L_{\mathbb{Z}[x]/\mathbb{Z}}\otimes_{\mathbb{Z}[x]}\mathbb{Z}$ is quasi-isomorphic to $\mathbb{Z}$ placed in degree 0. Hence the exact triangle implies that $L_{\mathbb{Z}/\mathbb{Z}[x]}$ is acyclic outside degree $1$. The isomorphism $H_1(L_{\mathbb{Z}/\mathbb{Z}[x]})\cong I/I^2$  follows from Fact \ref{diff} (\ref{secondexact}).

To treat the general case, consider the map $\mathbb{Z}[x]\to A$ sending $x$ to $f$. The $\mathbb{Z}[x]$-modules $A$ and $\mathbb Z$ are Tor-independent, because tensoring the short exact sequence $$0\to \mathbb{Z}[x]\overset{x}{\to}\mathbb{Z}[x]\to \mathbb{Z}\to 0$$ by $A$ over $\mathbb{Z}[x]$ yields the sequence $$0\to A\overset{f}{\to} A\to B\to 0$$ which is exact by the assumption that $f$ is nonzerodivisor. Therefore we may apply Lemma \ref{cotangflatbasechange} to obtain a quasi-isomorphism $A\otimes_{\mathbb{Z}[x]}^L L_{\mathbb{Z}/\mathbb{Z}[x]}\overset{\sim}{\to} L_{B/A}$, reducing the proposition to the special case treated above.
\end{proof}

\begin{rema}\rm
The above proposition can be easily extended to the case when $I$ is not necessarily principal but generated by a regular sequence.
\end{rema}

\subsection{First-order thickenings and the cotangent complex}

We continue the study of the cotangent complex by discussing its relation to first-order thickenings of $A$-algebras. Given an $A$-algebra $B$, a {\em first-order thickening} of $B$ is given by an extension
$$0\to I\to Y\to B\to 0$$
of $A$-algebras, where $I$ is an ideal satisfying $I^2=0$. Note that the condition $I^2=0$ implies that the natural $Y$-module structure on $I$ induces a $B$-module structure. Two first-order thickenings $Y_1$, $Y_2$ of $B$ whose kernels $I_1$, $I_2$ are isomorphic to $I$ as $B$-modules are called equivalent if there is a morphism $Y_1\to Y_2$ inducing the identity map on $B$ and a $B$-module isomorphism $I_1\cong I_2$. A Baer sum construction defines an abelian group structure on equivalence classes, denoted by ${\rm Exal}_A(B,I)$.

\begin{prop}\label{exalcommext1cot}
For a $B$-module $I$ we have a canonical isomorphism
$$\mathrm{Exal}_A(B,I)\stackrel{\sim}{\to}\Ext^1_B(L_{B/A},I).$$
\end{prop}
\begin{proof} We construct a set-theoretic bijection and leave the verification of additivity to the reader.

Consider the standard resolution $P_\bullet\to B$ of the $A$-algebra $B$.
Given a first-order thickening $Y$ of $B$ with ideal $I$, we may lift the surjection $\epsilon_0:\,P_0\to B$ to an $A$-algebra map $\theta\colon P_0\to Y$ by freeness of $P_0$. By composing with the differential $d_1=\partial_0-\partial_1:\, P_1\to P_0$ of the chain complex $CP_\bullet$, we obtain a map $D=\theta\circ d_1$ from $P_1$ to $I\subset Y$ which is readily seen to be an $A$-derivation. It thus induces a $P_1$-linear map $\Omega^1_{P_1/A}\to I$, whence also a $B$-linear map $\overline{D}\colon B\otimes_{P_1}\Omega^1_{P_1/A}\to I\subset Y$ by base extension, noting that the $P_1$-module structure on $Y$ (and hence on $I$) is given by the augmentation $\epsilon_1:\,P_1\to B$. Next, note that the differential $d_2$ of  $CP_\bullet$ induces a map $B\otimes_{P_2}\Omega^1_{P_2/A}\to B\otimes_{P_1}\Omega^1_{P_1/A}$. Its composite with  $\overline{D}$ factors through the map $B\otimes_{P_2}\Omega^1_{P_2/A}\to B\otimes_{P_0}\Omega^1_{P_0/A}$ induced by $d_1\circ d_2$, and hence is the zero map. Since
\begin{align*}
\Ext^1_B(L_{B/A},I)=H_1(\Hom(B\otimes_{P_\bullet}\Omega^1_{P_\bullet/A}, I)),
\end{align*}
the map $\overline D$ defines a class in $\Ext^1_B(L_{B/A},I)$. This class does not depend on the choice of the lifting $\theta$. Indeed, if $\theta'\colon P_0\to Y$, the relation $I^2=0$ implies that the difference $\theta-\theta'\colon P_0\to I$ is an $A$-derivation and hence gives rise to a map $B\otimes_{P_0}\Omega^1_{P_0/A}\to I$ as above. Composition with $d_1$ then yields a map in $\Hom_B(B\otimes_{P_1}\Omega^1_{P_1/A},I)$ which is $\overline{D}-\overline{D'}$ by construction, where $\overline D'\in\Hom_B(B\otimes_{P_1}\Omega^1_{P_1/A},I)$ is the map coming from $\theta'$. Thus $\overline D$ and $\overline D'$ define the same class in $\Ext^1_B(L_{B/A},I)$.

We construct an inverse map $\Ext^1_B(L_{B/A},I)\to \mathrm{Exal}_A(B,I)$ by reversing the above procedure. A class $\alpha$ in $\Ext^1_B(L_{B/A},I)$ is represented by a $B$-linear map $\overline D:\,B\otimes_{P_1}\Omega^1_{P_1/A}\to I$ whose restriction to the second factor gives rise to an $A$-derivation $D\colon P_1\to I$ such that $D\circ d_2=0$. Since ${\rm Im}(d_2)=\Ker(d_1)$, we have $\Ker(d_1)\subseteq\Ker(D)$. Note that this implies that the augmentation map $\epsilon_1:\, P_1\to B$ defining the $P_1$-module structure on $I$ has a set-theoretic section with values in $\Ker(D)$. Indeed, for $p_0\in P_0$ we have equalities $\partial_0(\sigma_0(p_0))=\partial_1(\sigma_0(p_0))=p_0$
where  $\sigma_0\colon P_0\to P_1$ is the degeneracy map and $\partial_i:\, P_1\to P_0$ the face maps, and therefore sending $x\in I$ to $\sigma_0(p_0)$ with some $p_0\in \epsilon_0^{-1}(x)$ defines such a section. It follows that $D(P_1)$ is a $B$-submodule of $I$, because for $p_1\in P_1$ and the above $p_0$ and $b$ we have $bD(p_1)=\epsilon_1(\sigma_0(p_0))D(p_1)=D(\sigma_0(p_0)p_1)$ by the Leibniz rule.

This being said, consider the $A$-module direct sum $P_0\oplus I$ equipped with the multiplication defined by $(p_0, i)(p_0', i')=(p_0p_0', p_0i'+p_0'i)$. It is an $A$-algebra in which $(0, I)$ is an ideal of square zero. Moreover, the $A$-module $Y$ defined as the cokernel of the $A$-module map $(d_1, D):\, P_1\to P_0\oplus I$ inherits an $A$-algebra structure from $P_0\oplus I$. Indeed, ${\rm Im}(d_1, D)$ is an ideal in $P_0\oplus I$ as all $p_0'\in P_0$, $p_1\in P_1$ and $x\in I$ satisfy
$$
(d_1(p_1), D(p_1))(p_0', x)=(d_1(p_1)p_0', d_1(p_1)x+p_0'D(p_1))=(d_1(p_1)p_0', \epsilon_0(p_0')D(p_1))\in {\rm Im}(d_1, D)
$$
since $d_1(p_1)\in \Ker(\epsilon_0)$ which is an ideal in $P_0$, and $D(p_1)\subset I$ is a $B$-submodule. It now follows that the surjection $(\epsilon_0, 0):\,P_0\oplus I\to B$ induces an $A$-algebra extension
$$
0\to I\to Y\to B\to 0,
$$
defining an object of ${\rm Exal}_A(B,I)$. We recover $D:\, P_1\to I$ as the derivation associated with the thickening $Y$ by the procedure of the previous paragraph, which shows that the two constructions are inverse to each other.
\end{proof}

Given an $A$-algebra $B$, first-order thickenings of $B$ naturally form a category $\underline{\mathrm{Exal}}_A(B)$ whose morphisms are $A$-algebra homomorphisms compatible with the surjections onto $B$.

\begin{prop}\label{Yuniv}
 If $\Omega^1_{B/A}=0$, the category $\underline{\mathrm{Exal}}_A(B)$ has an initial object.
\end{prop}
\begin{proof}
In view of Lemma \ref{h0cotang}, the assumption $\Omega^1_{B/A}=0$ implies that we may identify $\Ext^1_B(L_{B/A},I)$ with $\Hom_B(H_1(L_{B/A}),I)$ for all $B$-modules $I$. In particular, the identity map of $H_1(L_{B/A})$ yields a class in $\Ext^1_B(L_{B/A},H_1(L_{B/A}))$, which in turn corresponds to a first-order thickening $Y_{univ}$ of $B$ by Proposition \ref{exalcommext1cot}, with kernel $H_1(L_{B/A})$. That $Y_{univ}$ is an initial object follows by a Yoneda type argument from the functoriality of the  isomorphism $\mathrm{Exal}_A(B,I)\stackrel{\sim}{\to}\Hom_B(H_1(L_{B/A}),I)$  in $I$.
\end{proof}

We shall call $Y_{univ}$ the {\em universal first-order thickening of $B$.\/}\medskip

\begin{ex}\label{exyuniv}\rm In the case when $A\to B$ is a surjective morphism with kernel $J$, the condition $\Omega^1_{B/A}=0$ holds. In this case it is easy to describe $Y_{univ}$ by hand: it is given by the extension
$$
0\to J/J^2\to A/J^2\to B\to 0.
$$
In particular, we have an isomorphism $\mathrm{Exal}_A(B,I)\cong \Hom_B(J/J^2, I)$.
\end{ex}

Starting from Proposition \ref{exalcommext1cot}, Chapter III of \cite{I1} develops a deformation theory of algebras with the aid of the cotangent complex. We shall need two statements from this theory which we now explain.

Assume given a ring $A$ and an ideal $I\subset A$ of square zero. Given moreover an $A/I$-algebra $\overline B$ and a $\overline B$-module $J$ together with an $A$-module map $\lambda:\,I\to J$, one may ask whether there exists an $A$-algebra extension $B$ making the diagram
\begin{equation}\label{diagexistB}
\xymatrix{
0\ar[r]& J\ar[r] & B\ar[r] & \overline B\ar[r] & 0\\
0\ar[r]& I\ar[r]\ar[u]^\lambda & A\ar[r]\ar[u] & A/I\ar[r]\ar[u] & 0
}
\end{equation}
commute.

\begin{prop}\label{defext}
If $L_{\overline B/(A/I)}=0$, there is an $A$-algebra extension $B$ of $\overline B$ by $J$ making diagram (\ref{diagexistB}) commute, and such a $B$ is unique up to isomorphism.
\end{prop}

\begin{proof}
The extension class of $A$ defines a class in $\Ext^1_{A/I}(L_{{(A/I)}/A},I)$ by Proposition \ref{exalcommext1cot}, mapping to a class in $\Ext^1_{A/I}(L_{{(A/I)}/A},J)\cong \Ext^1_{\overline B}(L_{{(A/I)}/A}\otimes^L_{A/I}\overline B,J)$ via the map induced by $\lambda$ on Ext-groups. The assumption $L_{\overline B/(A/I)}=0$ yields a quasi-isomorphism $ L_{(A/I)/A}\otimes^L_{A/I}\overline B\cong L_{\overline B/A}$ by applying Theorem \ref{cotangexacttriang} to the sequence of maps  $A\to A/I\to \overline B$. We thus obtain a class in $\Ext^1_{\overline B}(L_{\overline B/A},J)$, giving rise to an $A$-algebra extension $B$ of $\overline B$ by $J$ via Proposition \ref{exalcommext1cot}. Going through the constructions in the proof of the said proposition one checks that up to isomorphism $B$ is the unique extension making diagram \eqref{diagexistB} commute.
\end{proof}

Now assume an $A$-algebra extension $B$ as above exists, and moreover $J=IB$ (and hence $\overline B=B/IB$). We have a natural surjection $I\otimes_{A/I}(B/IB)\to IB$ which is an isomorphism if $B$ is flat over $A$. Thus if $B$ is a {\em flat} $A$-algebra and $C$ an arbitrary $A$-algebra, every $A/I$-algebra map $\phi:\, B/IB\to C/IC$ gives rise to an $A/I$-algebra map $IB\to IC$ by tensoring with $I$. The map $\phi$ thus gives rise to a diagram with exact rows

\begin{equation}\label{extdiag}
\begin{CD}
0 @>>> IB @>>> B @>>> B/IB @>>> 0 \\
&& @V{\phi\otimes{\id}_I}VV && @VV{\phi}V \\
0 @>>> IC @>>> C @>>> C/IC @>>> 0
\end{CD}
\end{equation}

\begin{prop}\label{defmap}
In the above situation assume moreover $L_{(B/IB)/(A/I)}=0$. Then there exists a unique map $\widetilde \phi:\,B\to C$ making the diagram commute.
\end{prop}

\begin{proof}
First a word on uniqueness. The difference of two liftings of $\phi$ is an $A/I$-derivation $B/IB\to IC$. But the assumption $L_{(B/IB)/(A/I)}=0$ implies  $\Omega^1_{(B/IB)/(A/I)}=0$ in view of Proposition \ref{h0cotang}, so this derivation must be trivial.

For existence, observe first that the diagram (\ref{extdiag}) gives rise to two natural $A$-algebra extensions of $B/IB$ by $IC$: an extension $\widetilde B$ obtained as a pushout of the upper row by the map $\phi\otimes{\id}_I$, and an extension $\widetilde C$ obtained as the pullback of the lower row by the map $\phi$. The universal properties of pushout and pullback imply that a map $\widetilde\phi:\, B\to C$ as in the statement exists if and only if the extensions $\widetilde B$ and $\widetilde C$ are isomorphic.

By Proposition \ref{exalcommext1cot} both extensions have a class in $\Ext^1_{B/IB}(L_{(B/IB)/A}, IC)$. Theorem \ref{cotangexacttriang} applied to the sequence $A\to A/I\to B/IB$ gives an exact triangle
$$
(B/IB)\otimes^L_{A/I}L_{(A/I)/A}\to L_{(B/IB)/A}\to L_{(B/IB)/(A/I)}\to (B/IB)\otimes^L_{A/I}L_{(A/I)/A}[1].
$$
Applying the functor $\Ext^1_{B/IB}(\cdot, IC)$ gives an exact sequence
$$
\Ext^1_{B/IB}(L_{(B/IB)/(A/I)}, IC)\to \Ext^1_{B/IB}(L_{(B/IB)/A}, IC)\stackrel\rho\to \Ext^1_{A/I}(L_{(A/I)/A}, IC)
$$
where we may identify the last group with $\Hom_{A/I}(I, IC)$ by Proposition \ref{exalcommext1cot} and the Example \ref{exyuniv}. Furthermore, going through the constructions shows that the map $\rho$ sends both the class of $\widetilde B$ and that of $\widetilde C$ to the natural map $I\to IC$ induced by the structure map $A\to C$. But $\rho$ is injective since we have $\Ext^1_{B/IB}(L_{(B/IB)/(A/I)}, IC)=0$ by assumption. This shows $\widetilde B\cong \widetilde C$ as required.
\end{proof}

\subsection{The derived de Rham algebra}

We now come to the definition of the derived de Rham algebra $L\Omega^\bullet_{B/A}$.

Let $B$ be an $A$-algebra, and $P_\bullet\to B$ the standard simplicial resolution of $B$. The de Rham complex associated with the simplicial $A$-algebra $P_\bullet$ is given by the diagram
\begin{equation*}
\xymatrix{
\ddots & \vdots & \vdots &\vdots\\
\cdots\ar@<-.9ex>[r]\ar@<-.3ex>[r]\ar@<.3ex>[r]\ar@<.9ex>[r] & \Omega^2_{P_2/A} \ar@<-.6ex>[r]\ar[r]\ar@<.6ex>[r]\ar[u] & \Omega^2_{P_1/A} \ar@<-.3ex>[r]\ar@<.3ex>[r]\ar[u] & \Omega^2_{P_0/A } \ar[u]\\
\cdots\ar@<-.9ex>[r]\ar@<-.3ex>[r]\ar@<.3ex>[r]\ar@<.9ex>[r] & \Omega^1_{P_2/A} \ar@<-.6ex>[r]\ar[r]\ar@<.6ex>[r]\ar[u] & \Omega^1_{P_1/A} \ar@<-.3ex>[r]\ar@<.3ex>[r]\ar[u] & \Omega^1_{P_0/A } \ar[u]\\
 \cdots\ar@<-.9ex>[r]\ar@<-.3ex>[r]\ar@<.3ex>[r]\ar@<.9ex>[r] & P_{2}\ar@<-.6ex>[r]\ar[r]\ar@<.6ex>[r]\ar[u] & P_1\ar@<-.3ex>[r]\ar@<.3ex>[r]\ar[u] &P_{0}.\ar[u]
}
\end{equation*}
We may view it as a simplicial object in the category of differential graded $A$-algebras. By passing to the associated chain complex in the horizontal direction, we obtain a double complex $\Omega^\bullet_{P_\bullet/A}$.

\begin{defi}\rm The total complex (with the direct sum convention) of the double complex $\Omega^\bullet_{P_\bullet/A}$ is the {\em derived de Rham complex of $B$}. We denote it by $L\Omega^\bullet_{B/A}$.
\end{defi}

We sometimes view the derived de Rham algebra as an object in the bounded above derived category of $A$-modules, and sometimes as the complex itself. In the latter setting, we define the {\em  Hodge filtration} $F^iL\Omega^\bullet_{B/A}$ on $L\Omega^\bullet_{B/A}$ as the filtration induced by $$F^i(\Omega^\bullet_{P_\bullet/A})=\Omega^{\geq i}_{P_\bullet/A}$$ on the double complex $\Omega^\bullet_{P_\bullet/A}$.

The completion of $L\Omega^\bullet_{B/A}$ with respect to the Hodge filtration will play a crucial role in what follows. There is only one way to define it:

\begin{defi}\rm The {\em Hodge-completed derived de Rham complex of $B$} is defined as the projective system of complexes $L\widehat\Omega^\bullet_{B/A}:=(L\Omega^\bullet_{B/A}/F^i)$.
\end{defi}

To justify the terminology `de Rham {\em algebra}', we equip $L\Omega^\bullet_{P_\bullet/A}$ with the structure of a commutative differential graded algebra over $A$. We first define a product structure $P_i\otimes P_j \rightarrow P_{i+j}$ on the complex $CP_\bullet$ as the multiplication map on $P_0$ for $i=j=0$ and otherwise as the {\em shuffle map}

$$
x\otimes y \mapsto \sum_{(\mu,\nu)}\varepsilon(\mu,\nu)(\sigma_{\nu_1}\sigma_{\nu_2}\dots \sigma_{\nu_j}x)(\sigma_{\mu_1}\sigma_{\mu_2}\dots \sigma_{\mu_i}y)
$$
where $\sigma_{\nu_k}$ and $\sigma_{\mu_l}$ are the degeneracy maps in $P_\bullet$, the pair $(\mu,\nu)$ runs through the $(i,j)$-shuffles of the ordered set $(1,2,\dots,i+j)$ and $\varepsilon(\mu,\nu)$ is the sign of the shuffle as a permutation. (Recall that an $(i,j)$-shuffle is a permutation $\tau:\, (1,2,\dots,i+j)\to (1,2,\dots,i+j)$ with $\tau(1)<\tau(2)<\cdots <\tau(i)$ and $\tau(i+1)<\tau(i+2)<\cdots <\tau(i+j)$; this also makes sense if one of $i$ or $j$ is 0.)

The above product structure on $CP_\bullet$ induces a product structure on the double complex $\Omega^\bullet_{P_\bullet/A}$ and hence on the total complex $L\Omega^\bullet_{P_\bullet/A}$; we may therefore call it the {\em derived de Rham algebra of $B$}. In fact, with this product structure $L\Omega^\bullet_{P_\bullet/A}$ becomes a differential graded algebra over $A$. Moreover, one checks that the multiplicative structure on $\Omega^\bullet_{P_\bullet/A}$ is compatible with the Hodge filtration, and hence we may consider $L\widehat\Omega^\bullet_{B/A}$ as a projective system of differential graded algebras.

The  $i$-th graded piece with respect to the Hodge filtration on $L\Omega^\bullet_{P_\bullet/A}$ is computed as follows.
\begin{prop}\label{gridr}
There is a quasi-isomorphism of complexes of $A$-modules
\begin{equation}\label{grOmegaP}
\mathrm{gr}^i_F\Omega^\bullet_{P_\bullet/A}\stackrel\sim\to L\wedge^iL_{B/A}[-i]
\end{equation}
where $L_{B/A}$ is the cotangent complex of $B$.
\end{prop}

Here the object $L\wedge^iL_{B/A}$ is represented by $C\wedge^i(B_\bullet\otimes_{P_\bullet}\Omega^1_{P_\bullet/A})$, in accordance with Remarks \ref{derfunctrema} and \ref{remnc}.\medskip

\begin{dem}
First, note that
\begin{equation*}
\mathrm{gr}^i_FL\Omega^\bullet_{P_\bullet/A}\cong\Omega^i_{P_\bullet/A}[-i]=(\cdots\to \Omega^i_{P_j/A}\to\cdots\to\Omega^i_{P_1/A}\to \Omega^i_{P_0/A})
\end{equation*}
where on the right-hand side the term $\Omega^i_{P_j/A}$ has degree $j-i$ in the complex.

Consider now the constant simplicial ring $B_\bullet$, and view the augmentation map $P_\bullet\to B_\bullet$ as a morphism of simplicial $P_\bullet$-modules inducing a quasi-isomorphism on associated chain complexes. As $\Omega^i_{P_\bullet/A}$ is a simplicial $P_\bullet$-module with free terms, Lemma \ref{flattensorquasi} gives rise to the first quasi-isomorphism in the chain
$$ P_\bullet\otimes_{P_\bullet}\Omega^i_{P_\bullet/A}\overset{\sim}{\rightarrow}B_\bullet\otimes_{P_\bullet}\Omega^i_{P_\bullet/A}\cong B_\bullet\otimes_{P_\bullet}\wedge^i\Omega^1_{P_\bullet/A}\cong \wedge^i(B_\bullet\otimes_{P_\bullet}\Omega^1_{P_\bullet/A}).
$$
The quasi-isomorphism of the proposition follows.
\end{dem}

We next discuss the analogue of Theorem \ref{freeresolcotang}.

\begin{theo}\label{freeresolderham}
Let $Q_\bullet\to B$ be a simplicial resolution of the $A$-algebra $B$ whose terms are free $A$-algebras. We have a quasi-isomorphism of complexes $$L\Omega^\bullet_{B/A}\simeq {\rm Tot}(\Omega^\bullet_{Q_\bullet/A})$$ compatible with the product structure and the Hodge filtration.
\end{theo}

Here the product structure and the Hodge filtration on the right hand side are defined in the same way as on $\Omega^\bullet_{Q_\bullet/A}$.

The presentation below is influenced by unpublished notes of Illusie.\smallskip

\begin{proof}
The proof proceeds along the lines of that of Theorem \ref{freeresolcotang} but we have to be more careful concerning convergence issues.

We first fix $n\geq 0$ and start with the homotopy equivalence $P_n(Q_\bullet)\to P_n(B_\bullet)$ obtained during the proof of Lemma \ref{totcompiso}. By functoriality it induces a homotopy equivalence $\Omega^\bullet_{P_n(Q_\bullet)}\to \Omega^\bullet_{P_n(B_\bullet)}$ of simplicial objects in the category of complexes of $A$-modules, whence a quasi-isomorphism
${\rm Tot}(\Omega^\bullet_{P_n(Q_\bullet)})\simeq {\rm Tot}(\Omega^\bullet_{P_n(B_\bullet)})$ of associated complexes.

Consider now the double complex $C_{p,q}^Q={{\rm Tot}(\Omega^\bullet_{P_p(Q_\bullet)})}_q$ with horizontal differentials those of the complex associated with the simplicial object $[p]\mapsto {\rm Tot}(\Omega^\bullet_{P_p(Q_\bullet)})$ in the category of complexes of $A$-modules and vertical differentials given by those of ${\rm Tot}(\Omega^\bullet_{P_p(Q_\bullet)})$. We have a total complex ${\rm Tot}(C_{p,q}^Q)$ taken with the direct sum convention and a morphism of complexes ${\rm Tot}(C_{p,q}^Q)\to {\rm Tot}(C_{p,q}^B)$ with ${\rm Tot}(C_{p,q}^B)$ defined similarly starting from $C_{p,q}^B={{\rm Tot}(\Omega^\bullet_{P_p(B_\bullet)})}_q$. We claim that this map is a quasi-isomorphism.

Define subcomplexes $C_{\leq p_0}^Q\subset {\rm Tot}(C_{p,q}^Q)$ (resp. $C_{\leq p_0}^B\subset {\rm Tot}(C_{p,q}^B)$) by replacing the columns with $p>p_0$ in $C_{p,q}^Q$ (resp. $C_{p,q}^B$) by 0 and taking the associated total complexes. We have morphisms of complexes $C_{\leq p_0}^Q\to C_{\leq p_0}^B$ that form a direct system as $p_0$ goes to infinity. In the direct limit we recover the morphism of complexes ${\rm Tot}(C_{p,q}^Q)\to {\rm Tot}(C_{p,q}^B)$ considered above. It thus suffices to prove that each morphism $C_{\leq p_0}^Q\to C_{\leq p_0}^Q$ is a quasi-isomorphism. This follows by finite induction using  the exact sequences
$$
0\to C_{\leq p_0-1}^Q\to C_{\leq p_0}^Q\to {\rm Tot}(\Omega^\bullet_{P_{p_0}(Q_\bullet)})\to 0,\quad 0\to C_{\leq p_0-1}^B\to C_{\leq p_0}^B\to {\rm Tot}(\Omega^\bullet_{P_{p_0}(B_\bullet)})\to 0
$$
together with the quasi-isomorphisms ${\rm Tot}(\Omega^\bullet_{P_{p_0}(Q_\bullet)})\simeq {\rm Tot}(\Omega^\bullet_{P_{p_0}(B_\bullet)})$ established above. (Note that the above short exact sequences are actually split exact, as their terms are free $A$-modules.)

We thus obtain quasi-isomorphisms ${\rm Tot}(C_{p,q}^Q)\simeq {\rm Tot}(C_{p,q}^B)\simeq L\Omega^\bullet_{B/A}$ as $B_\bullet$ is a constant simplicial algebra. On the other hand, starting from the homotopy equivalences $P_\bullet(Q_n)\to {(Q_n)}_\bullet$ given by Lemma \ref{standardtriv} for each fixed $n$ and performing a similar construction as above, we obtain a quasi-isomorphism
$
{\rm Tot}(C_{p,q}^Q)\simeq{\rm Tot}(\Omega^\bullet_{Q_\bullet/A}).
$
Finally, compatibility with products and Hodge filtrations follows as the constructions involved in the above proof satisfy them.
\end{proof}

\begin{rema}\rm If one only wishes to prove the independence of the Hodge-completed derived de Rham algebra $L\widehat\Omega^\bullet_{B/A}$ of the resolution, the above argument simplifies as we do not have to worry about unbounded filtrations. Alternatively, once the Hodge-truncated versions of the maps ${\rm Tot}(C_{p,q}^Q)\to {\rm Tot}(C_{p,q}^B)$ and ${\rm Tot}(C_{p,q}^Q)\to{\rm Tot}(\Omega^\bullet_{Q_\bullet/A})$ used in the above proof have been constructed, we may reduce to Theorem \ref{freeresolcotang} by means of Proposition \ref{gridr}.
\end{rema}

\begin{ex}\label{barresol}\rm An important example of a simplicial resolution with free terms other than the standard resolution is given by the
{\em bar resolution} $Q_\bullet$ in the case $A=R[x]$, $B=R$ where the ring $R$ is viewed as an $R[x]$-algebra $R$ via the map $x\mapsto 0$.  Here

$$Q_n:=R[x][x_1,\dots,x_n]$$
and the face (resp.\ degeneracy) maps $\partial_0,\dots,\partial_n\colon Q_n\to Q_{n-1}$ (resp.\ $\sigma_0,\dots,\sigma_n\colon Q_n\to Q_{n+1}$) are defined by
\begin{eqnarray*}
\partial_i(x_j)&=&\begin{cases}
x_j &\text{if }n\neq j\leq i\\
x_{j-1}&\text{if }j>i\\
0&\text{if }j=i=n\ ,
\end{cases}\\
\sigma_i(x_j)&=&\begin{cases}
x_j &\text{if } j\leq i\\
x_{j+1}&\text{if }j>i
\end{cases}
\end{eqnarray*}
where by convention we put $x_0:=x$. That this is indeed a simplicial resolution is verified by direct computation.
\end{ex}

Given the theorem, we can establish the analogue of Lemma \ref{cotangflatbasechange} for derived de Rham algebras.

\begin{cor}\label{derbasechange}
Given Tor-independent $A$ algebras $A'$ and $B$, we have a canonical quasi-isomorphism $$A'\otimes^L_AL\Omega^\bullet_{B/A}\simeq L\Omega^\bullet_{A'\otimes_AB/A'}$$
 of associated derived de Rham algebras.
\end{cor}
\begin{proof} This is similar to the proof of Lemma \ref{cotangflatbasechange}, the main point being that for the standard resolution $P_\bullet\to B$ the base change ${A'\otimes_A}P_\bullet\to A'\otimes_AB$ gives a free resolution of $A'\otimes_AB$ by Tor-independence, and hence may be used to compute $L\Omega^\bullet_{A'\otimes_AB/A'}$ by the theorem.
\end{proof}

We now use the derived de Rham algebra to give an explicit construction of the universal first-order thickening of the previous section. Note first that by the compatibility of the multiplicative structure of $L\Omega^\bullet_{P_\bullet/A}$ with the filtration the group $H_0(L\Omega^\bullet_{P_\bullet/A}/F^2)$ is an $A$-algebra.

\begin{theo}\label{universal1order}
Assume $\Omega^1_{B/A}=0$. Then $H_0(L\Omega^\bullet_{P_\bullet/A}/F^2)$ is a universal first order thickening of the $A$-algebra $B$.
\end{theo}
\begin{proof} The truncated derived de Rham complex $L\Omega^\bullet_{P_\bullet/A}/F^2$ is the total complex
\begin{equation}\label{totdRBAFil2double}
\cdots\to \Omega^1_{P_2/A}\oplus P_1\to \Omega^1_{P_1/A}\oplus P_0\to \Omega^1_{P_0/A}
\end{equation}
of the double complex
\begin{equation}\label{dRBAFil2double}
\xymatrix{
\cdots\ar[r] & 0 \ar[r] & 0 \ar[r] & 0\\
\cdots\ar[r] & \Omega^1_{P_2/A} \ar[r]\ar[u] & \Omega^1_{P_1/A} \ar[r]\ar[u] & \Omega^1_{P_0/A } \ar[u]\\
 \cdots\ar[r] & P_{2}\ar[r]\ar[u] & P_1\ar[r]\ar[u] &P_{0}\ar[u]
}
\end{equation}
If we use homological indexing for the complex (\ref{totdRBAFil2double}), then $\Omega^1_{P_0/A}$ sits in degree $-1$ and ${\Omega^1_{P_1/A}\oplus P_0}$ in degree 0.
In view of Lemma \ref{h0cotang} the assumption $\Omega^1_{B/A}=0$ implies $$H_0(\Omega^1_{P_\bullet/A})=H_0(\Omega^1_{P_\bullet/A}\otimes_{P_\bullet}P_\bullet)=H_0(\Omega^1_{P_\bullet/A}\otimes_{P_\bullet}B_\bullet)=0$$ and hence the complex (\ref{totdRBAFil2double}) has trivial $H_{-1}$.  On the other hand, applying Proposition \ref{gridr} for $i=0,1$ yields an exact sequence
\begin{equation}\label{univf2}0\to H_1L_{B/A}\to H_0(L\Omega^\bullet_{B/A}/F^2)\to  B \to 0.
\end{equation}
By definition of the multiplication on $L\Omega^\bullet_{B/A}$ the square of $\Omega^1_{P_1/A}$ lies in $\Omega^2_{P_2/A}$, and therefore $H_1L_{B/A}$ has square zero in $H_0(L\Omega^\bullet_{B/A}/F^2)$, which shows that we have obtained a first-order thickening of the $A$-algebra $B$.

Now let $$0\to I\to Y\to B\to 0$$ be a first-order thickening of $B$. We view $Y$ as a differential graded algebra concentrated in degree zero. Lift the augmentation $\epsilon_0:\,P_0\to B$ to a morphism $\theta\colon P_0\to Y$, and construct the derivation $D=\theta\circ d_1$ as in the proof of Proposition \ref{exalcommext1cot}.  This gives rise to a morphism of differential graded $A$-algebras given by the diagram
\begin{equation}\label{dgamap}
\xymatrix{
\text{degree}& 1 & 0 & -1\\
\cdots\ar[r] & 0\ar[r] &  Y\ar[r] & 0\\
 \cdots\ar[r] & \Omega^1_{P_2/A}\oplus P_{1}\ar[r]\ar[u] & \Omega^1_{P_1/A}\oplus P_0\ar[r]\ar[u]_{{D}\oplus \theta} & \Omega^1_{P_0/A }\ar[u]
}
\end{equation}
which commutes in view of the identities $\epsilon_0\circ d_1=D\circ d_2=0$ seen in the proof of Proposition \ref{exalcommext1cot}. By passing to 0-th homology we obtain an $A$-algebra homomorphism $\varphi_Y\colon H_0(L\Omega^\bullet_{B/A}/F^2)\to Y$ lifting the map $H_0(L\Omega^\bullet_{B/A}/F^2)\to Y$ in (\ref{univf2}). In the case $Y=Y_{univ}$, where $Y_{univ}$ is as in Proposition \ref{Yuniv}, the restriction of $\varphi_{Y_{univ}}$ to the term $H_1L_{B/A}$ in (\ref{univf2}) is the identity map by construction, and we are done.
\end{proof}

\section{Differentials and the de Rham algebra for $p$-adic rings of integers}

\subsection{Modules of differentials for $p$-adic rings of integers}\label{secfontaine}

Let $K$ be a finite extension of $\Q_p$ with fixed algebraic closure $\overline K$. Denote by $\OK$ (resp. $\OKB$) their respective rings of integers and by $v$ the unique extension of the $p$-adic valuation.
The goal of this section is to present a fundamental calculation, due to Fontaine \cite{F}, of the module of differentials $\Omega^1_{\OKB/\OK}$.

Denote, as usual, by $\mu_{p^\infty}$ the torsion $\Z_p$-module of all $p$-primary roots of unity in $\overline K$. The logarithmic derivative defines a map of $\Z_p$-modules
$$
{\rm dlog:} \,\,\mu_{p^\infty}\to \Omega^1_{\OKB/\OK},\quad \zeta_{p^r}\mapsto d\zeta_{p^r}/\zeta_{p^r}
$$
with $\OKB$-linear extension
$$
{\rm dlog:}\,\, \OKB\otimes_{\Z_p}\mu_{p^\infty}\to \Omega^1_{\OKB/\OK}.
$$
Now taking the inverse limit $\Z_p(1)$  of the modules $\mu_{p^r}$ for all $r$, we have $\Q_p/\Z_p(1)=(\Q_p/\Z_p)\otimes\Z_p(1)\cong \mu_{p^\infty}$, and therefore after tensoring by the $\Z_p$-module $\OKB$ we obtain an isomorphism
\begin{equation}\label{divis}
(\overline K/\OKB)\otimes_{\Z_p}\Z_p(1)\cong \OKB\otimes_{\Z_p}\mu_{p^\infty}
\end{equation}
recalling that $\overline K=\OKB[p^{-1}]$.

\begin{theo}[Fontaine]\label{fontaine}
Denote by $K_0$ the maximal unramified subextension of $K|\Q_p$ and by $D_{K|K_0}$ the associated different. The map
$$
{\rm dlog:}\,\, (\overline K/\OKB)\otimes_{\Z_p}\Z_p(1)\to \Omega^1_{\OKB/\OK}
$$
induces an isomorphism
$$
(\overline K/I_K)\otimes_{\Z_p}\Z_p(1)\stackrel\sim\to \Omega^1_{\OKB/\OK}
$$
where
$$
I_K:=\{a\in \overline K:\, v(a)\geq -v(D_{K/K_0})-1/(p-1)\}.
$$
\end{theo}

Here are some easy corollaries of the theorem.

\begin{cor}\label{cp1}
The {\rm dlog} map induces an isomorphism
$$
{\Bbb C}_K(1)\stackrel\sim\to V_p(\Omega^1_{\OKB/\OK}).
$$
\end{cor}

Here, as usual, for a $\Z_p$-module $A$ the Tate module $T_p(A)$ is defined as the inverse limit of the $p^r$-torsion submodules ${}_{p^r}A$, and $V_p(A):=T_p(A)\otimes_{\Z_p}\Q_p$.

\begin{proof}
We have isomorphisms of $\Z_p$-modules
$$
{}_{p^r}(\overline K/I_K)= p^{-r}I_K/I_K\cong I_K/p^rI_K\cong \OKB/p^r\OKB.
$$
Passing to the inverse limit, we obtain the $p$-adic completion of $\OKB$ which is the ring of integers of ${\Bbb C}_K$. It remains to invert $p$ and apply the theorem.
\end{proof}

\begin{cor}\label{cp2}
The module of differentials $\Omega^1_{\OKB/\OK}$ is $p$-primary torsion and $p$-divisible. Moreover, the derivation $d:\, \OKB\to \Omega^1_{\OKB/\OK}$ is surjective.
\end{cor}

\begin{proof}
The first statement is immediate from the theorem together with formula (\ref{divis}). As for the second, pick an element $adb\in \Omega^1_{\OKB/\OK}$. As $\Omega^1_{\OKB/\OK}$ is $p$-primary torsion, we find  $r> 0$ such that $p^rda=p^rdb=0$. Since $\overline{K}$ is algebraically closed, there exists an element $x\in \OKB$ satisfying $x^{p^{2r}}+p^rx=b$ and hence also $p^{r}(p^rx^{p^{2r}-1}+1)dx=db$. On the other hand, we have $(p^rx^{p^{2r}-1}+1)db=db$ as $p^rdb=0$. Note that $(p^rx^{p^{2r}-1}+1)$ is invertible in $\OKB$, being congruent to $1$ modulo $p^r$. Therefore the above equalities yield $db=(p^rx^{p^{r}-1}+1)^{-1}db=p^rdx$, whence
\begin{align*}
adb=p^radx=d(p^rax)-xd(p^ra)=d(p^rax)-xp^rda=d(p^rax)
\end{align*}
showing $adb\in{\rm Im}(d)$.
\end{proof}

 Before starting the proof of the theorem we first recall some basic facts concerning extensions of local fields. All of them can be found in \cite{S}, Chapter III, \S\S 6,7.

\begin{facts}\label{localfacts}\rm
Let $L$ be a finite extension of $K$, with ring of integers $\OL$. There exists $b\in\OL$ such that $\OL=\OK[b]$. As an $\OK$-module $\OL$ is freely generated by finitely many powers of $b$. The module of differentials $\Omega^1_{\OL/\OK}$ is generated by a single element $db$ over $\OL$. Its annihilator is the different $D_{L/K}$ of the extension $L|K$; it is the principal ideal generated by $f'(b)$, where $f\in \OK[x]$ is the minimal polynomial of $b$.
\end{facts}

\begin{ex}\label{exzetap}\rm
As an example that will serve later, let us compute the different of the totally ramified extension $\Q_p(\zeta_{p^r})|\Q_p$, where $\zeta_{p^r}$ is a primitive $p^r$-th root of unity. The minimal polynomial of $\zeta_{p^r}$ over $\mathbb{Z}_p$ is $f=({x^{p^r}-1})/({x^{p^{r-1}}-1})$ with derivative $$f'=({p^rx^{p^r-1}(x^{p^{r-1}}-1)-p^{r-1}x^{p^{r-1}}(x^{p^r}-1)})/({(x^{p^{r-1}}-1)^2}).$$
Thus $f'(\zeta_{p^r})={p^r}/\zeta_{p^r}(\zeta_{p^r}^{p^{r-1}}-1)$, so setting $\zeta_p:=\zeta_{p^r}^{p^{r-1}}$ we see that the sought-after different is the principal ideal of $\Z_p[\zeta_{p^r}]$ generated by ${p^r}/(\zeta_p-1)$. The $p$-adic valuation of this element is $-1/(p-1)+r$, as  $\zeta_p-1$ is a uniformizer in the degree $p-1$ totally ramified extension $\mathbb{Q}_p(\mu_p)|\Q_p$.
\end{ex}

As observed by Beilinson, the use of the cotangent complex considerably simplifies Fontaine's original calculation, so we next compute $L_{\OKB/\OK}$.

\begin{lem}\label{simpresOmega}
Let $L|K$ be a finite extension of $p$-adic fields. The cotangent complex $L_{\OL/\OK}$  is acyclic in nonzero degrees.
\end{lem}

\begin{proof}
Writing $\OL=\OK[x]/(f)$ a some monic polynomial $f\in\OK[x]$ as above, we may consider the sequence of ring maps $\OK\to\OK[x]\to\OL$ where the second map is the quotient map. The associated transitivity triangle (Theorem \ref{cotangexacttriang}) reads
\begin{equation}\label{disttriang}
L_{\OK[x]/\OK}\otimes_{\OK[x]}^L\OL\to L_{\OL/\OK}\to L_{\OL/\OK[x]}\to L_{\OK[x]/\OK}\otimes_{\OK[x]}^L\OL[1].
\end{equation}
Here $L_{\OK[x]/\OK}$ is acyclic in nonzero degrees by Proposition \ref{polcotangacyc}, and $L_{\OL/\OK[x]}$ is acyclic in degrees $\neq 1$ by Proposition \ref{cotangquotientreg} where its homology is $(f)/(f^2)$. It thus remains to check $H_1(L_{\OL/\OK})=0$. A piece of the long exact sequence of (\ref{disttriang}) reads
\begin{equation*}
0\to H_{1}(L_{\OL/\OK})\to H_1(L_{\OL/\OK[x]})\to H_0(L_{\OK[x]/\OK}\otimes_{\OK[x]}^L\OL).
\end{equation*}
By Lemma \ref{h0cotang} and Proposition \ref{cotangquotientreg} we may identify the last map in this sequence with
$$
(f)/(f^2){\to} \Omega^1_{\OK[x]/\OK}\otimes_{\OK[x]}\OL
$$
which is a nonzero map of free $\OL$-modules of rank 1 sending the class of $f$ to $df$. This shows $H_{1}(L_{\OL/\OK})=0$ as desired.
\end{proof}

\begin{rema}\label{remsimpresomega}\rm We remark for later use that the same argument as in the above proof shows $L_{L/K}=\Omega^1_{L/K}=0$ for a finite separable extension $L|K$ of arbitrary fields.
\end{rema}

\begin{cor}\label{corsimpresomega}
The natural map
$$
L_{\OKB/\OK}\to H_0(L_{\OKB/\OK})\cong \Omega^1_{\OKB/\OK}
$$
is an isomorphism.
\end{cor}

\begin{proof}
Writing $\OKB$ as the direct limit of the $\OK$-algebras $\OL$ for each finite subextension $K\subset L\subset \overline K$ induces an isomorphism $\Omega^1_{\OKB/\OK}\cong \limind\Omega^1_{\OL/\OK}$. Similarly, the standard resolution $P_\bullet(\OKB)\to \OKB$ is the direct limit of the standard resolutions $P_\bullet(\OL)\to \OL$, so after applying the functor $\Omega^1_{\cdot/\OK}$ and tensoring with $\OL$ we obtain an isomorphism $L_{\OKB/\OK}\cong \limind L_{\OL/\OK}$. It remains to apply the isomorphisms $L_{\OL/\OK}\stackrel\sim\to\Omega^1_{\OL/\OK}$ given by the lemma.
\end{proof}

We now assemble some auxiliary statements to be used in the proof of the theorem.

\begin{lem}\label{lemfontaine} ${}$
\begin{enumerate}
\item The maps $\Omega^1_{\OL/\OK}\to \Omega^1_{\OKB/\OK}$ appearing in the above proof are injective.
\item If $\omega\in \Omega^1_{\OKB/\OK}$ comes from $\Omega^1_{\OL/\OK}$ and $I_L$ is its annihilator as an element of $\Omega^1_{\OL/\OK}$, then its annihilator $I_{\overline K}$ in $\OKB$ is $I_L\OKB$. In particular, $I_{\overline K}$ is principal.
\item If $K_0$ is the maximal unramified subextension of $K|\Q_p$ and $D_{K/K_0}$ is the associated different, we have an exact sequence
    $$0\to \OKB/D_{K/K_0}\OKB\to \Omega^1_{\OKB/\CO_{K_0}}\to \Omega^1_{\OKB/\OK}\to 0.$$
    Moreover, there is an isomorphism
    $$
    \Omega^1_{\OKB/\CO_{K_0}}\cong \Omega^1_{\OKB/\Z_p}.
    $$
\end{enumerate}
\end{lem}

\begin{proof}
The transitivity triangle of the cotangent complex (Theorem \ \ref{cotangexacttriang}) associated with the sequence of maps $\OK\to \OL\to \OKB$ reduces to a short exact sequence $$0\to \OKB\otimes_{\OL}\Omega^1_{\OL/\OK}\to \Omega^1_{\OKB/\OK}\to \Omega^1_{\OKB/\OL}\to 0$$ in view of Corollary \ref{corsimpresomega}.
Moreover, since $\OKB$ is a directed union of free $\OL$-submodules as recalled in Facts \ref{localfacts}, it is faithfully flat over $\OL$ and hence the natural map $\Omega^1_{\OL/\OK}\to \OKB\otimes_{\OL}\Omega^1_{\OL/\OK}$ is injective. Thus so is the composite $\Omega^1_{\OL/\OK}\to \Omega^1_{\OKB/\OK}$, whence statement (1). Statement (2) follows from the injectivity of the map $\OKB\otimes_{\OL}\Omega^1_{\OL/\OK}\to \Omega^1_{\OKB/\OK}$. Finally, for the exact sequence in statement (3) we use the transitivity triangle associated with the sequence $\CO_{K_0}\to \OK\to \OKB$ to obtain
$$0\to \OKB\otimes_{\OK}\Omega^1_{\OK/\CO_{K_0}}\to\Omega^1_{\OKB/\CO_{K_0}}\to \Omega^1_{\OKB/\OK}\to 0$$
and apply the definition of the different (Facts \ref{localfacts}). The last isomorphism is induced by the exact sequence of differentials associated with the sequence of maps  $\Z_p\to\CO_{K_0}\to \OKB$, noting that $\Omega^1_{\CO_{K_0}/\Z_p}=0$ as the ring extension $\CO_{K_0}|\Z_p$ is unramified.\end{proof}

\noindent{\em Proof of Theorem \ref{fontaine}.} Using Lemma \ref{lemfontaine} (3) we reduce to the case $K=\Q_p$. In this case $I_K=(1/(\zeta_p-1))\OKB$ for a primitive $p$-th root of unity $\zeta_p$ as $v(\zeta_p-1)=1/(p-1)$.

We first determine the kernel of the dlog map. As finitely generated submodules of $\mu_{p^\infty}$ are cyclic, we may write each element of $\OKB\otimes\mu_{p^\infty}$ in the form $a\otimes \zeta_{p^r}$ for some $a\in \OKB$ and $\zeta_{p^r}\in\mu_{p^\infty}$. This element is in the kernel of the dlog map if and only if $a$ annihilates $d\zeta_{p^r}$. Applying Lemma \ref{lemfontaine} (2) with $L=\Q_p(\zeta_{p^r})$ and the calculation in Example \ref{exzetap}, we obtain that $a\in (p^r/(\zeta_p-1))\OKB\subset (1/(\zeta_p-1))\OKB$, as desired.

For surjectivity, pick $\omega\in\Omega^1_{\OKB/\Z_p}$. By Lemma \ref{lemfontaine} (2) we have $\OKB\omega\cong \OKB/I_{\overline K}$ where $I_{\overline K}\subset \OKB$ is a principal ideal. If $a_\omega\in I_{\overline K}$ is a generator, we have $v(a_\omega)\leq -1/(p-1)+r$ for $r$ large enough. Now choose a finite extension $L|\Q_p$ such that $\omega$ comes from $\Omega^1_{\OL/\Z_p}$ and moreover $p^r/(\zeta_p-1)\in\OL$. As $\OL$ is a discrete valuation ring whose valuation is a multiple of $v$, the inequality $v(a_\omega)\leq -1/(p-1)+r$ implies $\OL a_\omega\supset \OL(p^r/(\zeta_p-1))$. But then
$$
\OKB\omega\cong \OKB/\OKB a_\omega\subset \OKB/ \OKB(p^r/(\zeta_p-1))\cong \OKB{\rm dlog}(\zeta_{p^r})
$$
by the calculation recalled above.
\enddem

\subsection{The universal $p$-adically complete first order thickening of $\OCK/\OK$}

We now combine the results of the previous two sections to compute the truncated de Rham algebra $L\Omega^\bullet_{B/A}/F^2$  in the special case $A=\OK/(p^n)$ and $B=\OKB/(p^n)$ for an integer $n>0$.

\begin{prop}\label{univ1orderOCKn} Let $K$ be a $p$-adic field with algebraic closure $\overline K$, and let $n>0$ be a fixed integer.
\begin{enumerate}
\item The truncated de Rham algebra $L\Omega^\bullet_{(\OKB/(p^n))/(\OK/(p^n))}/F^2$ is concentrated in degree 0.
\item  We have a short exact sequence
\begin{equation}\label{h0seq}
0\to {}_{p^n}\Omega^1_{\OKB/\OK}\to H_0(L\Omega^\bullet_{(\OKB/(p^n))/(\OK/(p^n))}/F^2)\to \OKB/(p^n)\to 0\
\end{equation}
where the term on the left identifies with the image of $H_0(F^1L\Omega^\bullet_{(\OKB/(p^n))/(\OK/(p^n))}/F^2)$.
\item  The $\OK/(p^n)$-algebra $H_0(L\Omega^\bullet_{(\OKB/(p^n))/(\OK/(p^n))}/F^2)$ is the universal first order thickening of  $\OKB/(p^n)$.
    \end{enumerate}
\end{prop}

\begin{proof} Consider the standard resolution $\widetilde P_\bullet$ of $\OKB/(p^n)$ as an $\OK/(p^n)$-module. The complex $L\Omega^\bullet_{(\OKB/(p^n))/(\OK/(p^n))}/F^2$ is computed by the total complex of
\begin{equation*}
\xymatrix{
\cdots\ar[r] & 0 \ar[r] & 0 \ar[r] & 0\\
\cdots\ar[r] & \Omega^1_{\widetilde P_2/(\OK/(p^n))} \ar[r]\ar[u] & \Omega^1_{\widetilde P_1/(\OK/(p^n))} \ar[r]\ar[u] & \Omega^1_{\widetilde P_0/(\OK/(p^n)) } \ar[u]\\
 \cdots\ar[r] & \widetilde P_{2}\ar[r]\ar[u] & \widetilde P_1\ar[r]\ar[u] & \widetilde P_{0}\ar[u].
}
\end{equation*}
Here the bottom row is the resolution $\widetilde P_\bullet$ of $\OKB/(p^n)$. The middle row computes the cotangent complex $L_{(\OKB/(p^n))/(\OK/(p^n))}$ which is quasi-isomorphic to $L_{\OKB/\OK)}\otimes^L_{\OK}(\OK/(p^n))$ by Lemma \ref{cotangflatbasechange} as $\OKB$ is flat over $\OK$. But by Corollary \ref{corsimpresomega} we have a quasi-isomorphism $L_{\OKB/\OK}\simeq \Omega^1_{\OKB/\OK}$, so we have quasi-isomorphisms
$$L_{(\OKB/(p^n))/(\OK/(p^n))}\simeq L_{\OKB/\OK}\otimes^L_{\OK}(\OK/(p^n)) \simeq \Omega^1_{\OKB/\OK}\otimes_{\OK}[\OK\stackrel{p^n}\to\OK].$$
Therefore $L\Omega^\bullet_{(\OKB/(p^n))/(\OK/(p^n))}/F^2$ is computed by the total complex of
$$
\begin{CD}
\Omega^1_{\OKB/\OK} @>{p^n}>> \Omega^1_{\OKB/\OK} \\
@AdAA @AAdA \\
\OKB @>{p^n}>> \OKB
\end{CD}
$$
which is placed in degrees $-1$, 0 and 1. Since $\OKB$ has no $p$-torsion and $d$ is surjective by Corollary \ref{cp2}, this complex is indeed concentrated in degree 0. Using the $p$-divisibility of $\Omega^1_{\OKB/\OK}$ we see that we have an exact sequence as in (\ref{h0seq}). The last statement follows from Theorem \ref{universal1order} once we check $\Omega^1_{(\OKB/(p^n))/(\OK/(p^n))}=0$.  But by the base change property of differentials
$$
\Omega^1_{(\OKB/(p^n))/(\OK/(p^n))}\cong \Omega^1_{\OKB/\OK}\otimes_{\OK}\OK/(p^n)
$$
and the right hand side is 0 as $\Omega^1_{\OKB/\OK}$ is $p$-divisible (Corollary \ref{cp2}).
\end{proof}

\begin{cor}\label{coruniv}
The inverse limit
\begin{equation*}
\varprojlim_n H_0(L\Omega^\bullet_{(\OKB/(p^n))/(\OK/(p^n))}/F^2)
\end{equation*}
fits into a short exact sequence
\begin{equation*}
0\to T_p(\Omega^1_{\OKB/\OK})\to \varprojlim_n H_0(L\Omega^\bullet_{(\OKB/(p^n))/(\OK/(p^n))}/F^2)\to \CO_{\mathbb{C}_K}\to 0\
\end{equation*}
and defines a universal first-order thickening of the $\OK$-algebra $\OCK$ in the category of $p$-adically complete $\OK$-algebras.
\end{cor}

\begin{proof}
In view of the proposition it remains to note that the inverse system of the exact sequences (\ref{h0seq}) satisfies the Mittag-Leffler condition and that $\OCK$ is none but the $p$-adic completion of $\OKB$.
\end{proof}

There is also an arithmetic approach to universal $p$-adically complete thickenings via Fontaine's ring $A_{\rm inf}$ that we explain next. We first begin with a quick proof of the basic facts concerning Witt rings of perfect rings by means of the cotangent complex, based on  ideas of Bhatt.

\begin{prop}\label{lifthom} Let $R$ be a perfect ring of characteristic $p>0$.
\begin{enumerate}
\item Up to isomorphism there is a unique $p$-adically complete flat $\Z_p$-algebra $W(R)$ with $W(R)/(p)\cong R$.

\item  Given moreover a $p$-adically complete ring $S$, every ring homomorphism $R\to S/(p)$ lifts uniquely to a $p$-adically continuous homomorphism $W(R)\to S$.
\end{enumerate}
\end{prop}

\begin{proof}

To prove (1), we construct by induction on $n$ flat $\Z/p^n\Z$-algebras $W_n(R)$ such that $W_1(R)=R$ and $W_i(R)\cong W_n(R)/(p^i)$ for all $1\leq i\leq n$. Assuming that $W_n(R)$ has been constructed, apply Proposition \ref{defext} with $A=\Z/p^{n+1}\Z$, $I=p^n\Z/p^{n+1}\Z$, $J=R$, $\overline B=W_n(R)$ and $\lambda:\,  p^n\Z/p^{n+1}\Z\to R\/$ the natural map to obtain a $\Z/p^{n+1}\Z$-algebra extension $W_{n+1}(R)$ of $W_n(R)$ by $R$. To be able to apply the proposition, we need to know that $L_{W_n(R)/(\Z/p^n\Z)}=0$. This vanishing follows from a more general statement, Lemma \ref{relperfcot0hodgetrunc} (1) below that we shall prove by an argument that uses only properties of the cotangent complex encountered so far. Note that $p^nW_n(R)=0$ implies $R\subset p^nW_{n+1}(R)$, and this inclusion is in fact an equality as the $W_{n+1}(R)$-module structure on $R$ coming from the extension structure is given by the composite of the surjections $W_{n+1}(R)\to W_n(R)\to R$. Thus we have isomorphisms $p^iW_{n+1}(R)/p^{i+1}W_{n+1}(R)\cong R$ for all $i\leq n$, whence we deduce $_{p^i}W_{n+1}(R)=p^{n+1-i}W_{n+1}(R)$ for all $1\leq i\leq n$ using the perfectness of $R$. This implies the flatness of $W_{n+1}(R)$ over $\mathbb{Z}/p^{n+1}$.

As for (2), by $p$-adic completeness it suffices to lift the map $R\to S/(p)$ inductively to maps $W_n(R)\to S/(p^n)$. Assume that a unique mod $p^n$ lifting exists. In view of the vanishing of $L_{W_n(R)/(\Z/p^n\Z)}$ already used above, the existence of a unique mod $p^{n+1}$ lifting follows from applying Proposition \ref{defmap} with $A=\Z/p^{n+1}\Z$, $B=W_{n+1}(R)$, $C=S/(p^{n+1})$ and $I=(p^n)$.
\end{proof}

The ring $W(R)$ is the Witt ring of $R$ as constructed e.g. in \cite{S}, \S II.5. For computational proofs of statement (2), see  \cite{BC}, Section 4.4 or \cite{Ked}.

Assume now that $R$ is a ring of characteristic $p>0$ on which the Frobenius morphism $x\mapsto x^p$ is surjective. We define the {\em perfection of $R$}  as the inverse limit $$R^{\rm perf}:=\varprojlim_{x\mapsto x^p}R.$$
Thus $R^{\rm perf}$ consists of sequences $(x_i)$ with $x_i^p=x_{i-1}$. On such sequences the map $x\mapsto x^p$ is bijective, hence $R^{\rm perf}$ is a perfect ring.

Following Fontaine, we set
$$
\Ainf:=W((\OKB/(p))^{\rm perf}).
$$
Since $\OCK$ is the $p$-adic completion of $\OKB$, we have $\OKB/(p)\cong \OCK/(p)$. By Proposition \ref{lifthom} (2), the natural surjection
$$
\bar\theta:\,(\OKB/(p))^{\rm perf}\twoheadrightarrow \OKB/(p)
$$
lifts to a surjection
$$
\theta:\,\Ainf\twoheadrightarrow \OCK.
$$

Note that $\Ainf$ is complete with respect to its $\ker(\theta)$-adic filtration. This follows from $p$-adic completeness and the fact that $(\OKB/(p))^{\rm perf}$ is complete with respect to the $\ker(\bar\theta)$-adic filtration.

Now a surjection $\rho:\, B\to A$ of $p$-adically complete $\OK$-algebras is {\em an order $k$ thickening} for some $k>0$ if $\ker(\rho)^{k+1}=0$. For fixed $A$ such pairs $(B,\rho)$ form a natural category, and an initial element in this category (if exists) is called a {\em universal $p$-adically complete $\OK$-thickening of order $k$}.

\begin{prop}[Fontaine]\label{propfontaine}
For each $k>0$ the $\OK$-algebra $\OCK$ has a universal $p$-adically complete $\OK$-thickening of order $k$, given by
$$
(\Ainf/\ker(\theta)^{k+1})\otimes_{\Z_p}\OK.
$$
\end{prop}

\begin{proof}
It suffices to treat the case $\OK=\Z_p$, as then the general case follows by base change. Furthermore, in view of Proposition \ref{lifthom} (2), given an order $k$ thickening $\rho:\, B\to\OCK$ it suffices to construct a map $\tau:\,(\OKB/(p)^{\rm perf}\to B/(p)$.

For an element $x\in \OCK/(p)=\OKB/(p)$ choose some lifting $\widehat{x}\in B/(p)$ via the mod $p$ reduction $\bar\rho$ of $\rho$. Given an element $(\dots,x_n,\dots,x_0)\in (\OKB/(p))^{\rm perf}$, set
$$\tau(\dots,x_n,\dots,x_0):=\lim_{n\to\infty} \widehat{x_n}^{p^n}\in B/(p).$$ Note that this limit exists since $\ker(\bar\rho)^k=0$ and we obtain a ring homomorphism. Also, this is the only possible definition as $\tau(\dots,x_{n+r},\dots,x_r)\equiv \widehat{x_r}\pmod{\ker(\bar\rho)}$ forces $$\tau(\dots,x_n,\dots,x_0)=\tau(\dots,x_{n+r},\dots,x_r)^{p^r}\equiv \widehat{x_r}^{p^r}\pmod{\ker(\bar\rho)}$$ for all $r\geq 0$.
\end{proof}

\begin{cor}\label{corfontaine}
We have a canonical isomorphism of $\OK$-algebras
$$
(\Ainf/\ker(\theta)^{2})\otimes_{\Z_p}\OK\cong \varprojlim_n H_0(L\Omega^\bullet_{(\OKB/(p^n))/(\OK/(p^n))}/F^2).
$$
In particular, we have an isomorphism
$$
\ker(\theta)/\ker(\theta)^{2}\otimes_{\Z_p}\OK\cong T_p(\Omega^1_{\OKB/\OK}).
$$
and therefore the ideal
$\ker(\theta)\otimes_{\Z_p}\OK\subset \Ainf\otimes_{\Z_p}\OK$
is principal.
\end{cor}

\begin{proof}
The two isomorphisms result from putting Corollary \ref{coruniv} and Proposition \ref{propfontaine} together. By Theorem \ref{fontaine} the $\OCK$-module $T_p(\Omega^1_{\OKB/\OK})$ is free of rank 1, hence so is the $\OCK$-module $\ker(\theta)/\ker(\theta)^{2}$. The last statement follows as $\Ainf$ is $\ker(\theta)$-adically complete.
\end{proof}

\subsection{Derived de Rham algebra calculations}

Our next goal is to compute the $p$-adic completion of the derived de Rham algebra $L\Omega^\bullet_{\OKB/\OK}$ for a finite extension $K|\Q_p$. The methods to do so stem from the preprint \cite{Bh2} of Bhargav Bhatt. This section is devoted to preliminary calculations.

Arguably the key step is the computation of $L\Omega^\bullet_{(\mathbb{Z}/p^n\Z)/(\mathbb{Z}/p^n\Z[x])}$, where $\Z/p^n\Z$ is viewed as a $\Z/p^n\Z[x]$-algebra via the natural projection sending $x$ to 0. To describe it, we  need the divided power algebra $\Gamma_A^\bullet(M)$ introduced in Lemma \ref{freepd} of the Appendix in the case where $M\cong A^n$ is a free $A$-module on generators $t_1,\dots, t_n$. We set
$$
A\langle t_1,\dots, t_n\rangle:=\Gamma_A^\bullet(A^n).
$$
We denote the kernel of the natural augmentation map $A\langle t_1,\dots, t_n\rangle\to A$ by $\langle t_1,\dots, t_n\rangle$. The {\em divided powers} of the ideal $\langle t_1,\dots, t_n\rangle$ are defined as follows. First, the maps $\gamma_i:\, At_1\oplus\dots\oplus At_n\to A\langle t_1,\dots, t_n\rangle$ extend to a unique divided power structure on $A\langle t_1,\dots, t_n\rangle$ by $\langle t_1,\dots, t_n\rangle$ by setting
\begin{equation}\label{gammakgammai}
\gamma_k(\gamma_{i_1}(x_1)\dots\gamma_{i_n}(x_n)):=\frac{\prod_{l=1}^n(ki_l)!}{k!\cdot\prod_{l=1}^n(i_l!)^k}\gamma_{ki_1}(x_1)\dots\gamma_{ki_n}(x_n). \end{equation}
Next, one defines the divided powers of the ideal by
$$
\langle t_1,\dots, t_n\rangle^{[i]}:=A[\gamma_{i_1}(x_1)\cdots \gamma_{i_r}(x_r)\mid x_j\in \langle t_1,\dots, t_n\rangle, i_1+\cdots + i_r\geq i].
$$

\begin{rema}\label{divpowrema}\rm These formulas are unfortunately complicated, but notice for later use that in the case where $A$ is a domain with fraction field $K$, the filtration by  $\langle t_1,\dots, t_n\rangle^{[i]}\otimes_AK$ on $K\langle t_1,\dots, t_n\rangle$ becomes the filtration by powers of $\langle t_1,\dots, t_n\rangle$.
\end{rema}

\begin{prop}[Bhatt]\label{dRPDpol}
The derived de Rham algebra $L\Omega^\bullet_{(\mathbb{Z}/p^n\Z)/(\mathbb{Z}/p^n\Z[x])}$  is concentrated in degree 0 and we have an isomorphism
$$
H_0(L\Omega^\bullet_{(\mathbb{Z}/p^n\Z)/(\mathbb{Z}/p^n\Z[x])})\cong \mathbb{Z}/p^n\Z\langle x\rangle.
$$
Moreover, under this isomorphism the Hodge filtration on the left hand side coincides with the filtration by divided powers $\langle x\rangle^{[i]}$  on the right.
\end{prop}

We start the proof of the proposition with the case $n=1$. It is based on the following splitting lemma.

\begin{lem}\label{Cartierpol}
We have a quasi-isomorphism
\begin{equation*}
L\Omega^\bullet_{\mathbb{F}_p/\mathbb{F}_p[x]}\simeq\bigoplus_{i= 0}^\infty L\wedge^i L_{(\mathbb{F}_p[x]/(x^p))/\mathbb{F}_p[x]}[-i]
\end{equation*}
where $\mathbb{F}_p[x]/(x^p)$ is viewed as an $\mathbb{F}_p[x]$-algebra via the natural projection.
\end{lem}

The proof of the lemma, which is a version of the decomposition technique of \cite{deligneillusie} for the de Rham complex,  will use some basic facts about relative Frobenii that we now recall.

\begin{facts}\label{frob}\rm Assume $A$ and $B$ are $\F_p$-algebras, and consider the $A$-algebra $A^{(1)}$ defined by $A$ with its $A$-algebra structure given by the Frobenius map $a\mapsto a^p$. We have a morphism of $A$-algebras $A\to A^{(1)}$ induced by Frobenius, whence a morphism $B\to B^{(1)}:=B\otimes_AA^{(1)}$ by base change. Furthermore, the commutative square of $A$-algebras
$$
\begin{CD}
A @>>> B \\
@VpVV @VVpV \\
A @>>> B
\end{CD}
$$
induces a morphism $B^{(1)}\to B$. When $A$ is perfect, the morphism $A\to A^{(1)}$ is an isomorphism by definition, hence so is the base change $B\to B^{(1)}$. If moreover $B$ is perfect, the morphism $B^{(1)}\to B$ induced by the diagram is an isomorphism as well.\end{facts}

\begin{proof} For $n\geq 0$ set $Q_n:=\F_p[x][x_1,\dots, x_{n}]$ and consider the above situation for $A=\F_p[x]$ and $B=Q_n$. Identifying $\F_p[x]$ with $\F_p[x^p]$ via the Frobenius map, the map $Q_n^{(1)}\to Q_n$ becomes identified with the map $\F_p[x] [x_1,\dots, x_{n}]\to \F_p[x][ x_1,\dots, x_{n}]$ that is the identity on $\F_p[x]$ and sends  $x_j$ to $x_j^p$. We may lift this map to a morphism of $\Z/p^2\Z[x]$-algebras
$$
(\Z/p^2\Z)[x][x_1,\dots, x_{n}]\to (\Z/p^2\Z)[x][x_1,\dots, x_{n}]
$$
sending $x_j$ to $x_j^p$. For all $i\geq 1$ there is an induced map
$$
F_i:\,\Omega^i_{(\Z/p^2\Z)[x][x_1,\dots, x_{n}]/(\Z/p^2\Z)[x]}\to \Omega^i_{(\Z/p^2\Z)[x][x_1,\dots, x_{n}]/(\Z/p^2\Z)[x]}
$$
on differential forms whose image is contained in $p^i\Omega^i_{(\Z/p^2\Z)[x][x_1,\dots, x_{n}]/(\Z/p^2\Z)[x]}$. As the $p$-torsion of the free $\Z/p^2\Z$-modules $\Omega^i_{(\Z/p^2\Z)[x][x_1,\dots, x_{n}]/(\Z/p^2\Z)[x]}$ is $p\Omega^i_{(\Z/p^2\Z)[x][x_1,\dots, x_{n}]/(\Z/p^2\Z)[x]}$, the map $$\omega\mapsto (1/p)F_i(\omega)$$ induces a well-defined map $$\overline{(1/p)F_i}:\, \Omega^i_{Q_n^{(1)}/\F_p[x]}\to \Omega^i_{Q_n/\F_p[x]}$$ after reducing modulo $p$; it is the zero map for $i>1$ but nonzero for $i=1$. Thus by construction we obtain a commutative diagram
\begin{equation*}
\xymatrix{
\Omega^1_{Q_n^{(1)}/\F_p[x]}\ar[rr]\ar[d]_d && \Omega^1_{Q_n/\F_p[x]}\ar[d]_d\\
\Omega^2_{Q_n^{(1)}/\F_p[x]}\ar[rr] && \Omega^2_{Q_n/\F_p[x]}
}
\end{equation*}
whose horizontal maps are respectively given by $\overline{(1/p)F_1}$ and $\overline{(1/p)F_2}$ modulo $p$. As the latter map is zero, we get a well-defined map of complexes $$\Omega^1_{Q_n^{(1)}/\F_p[x]}[-1]\to \Omega^\bullet_{Q_n/\F_p[x]}.$$ Taking the direct sum of the $i$-th wedge powers of $\overline{(1/p)F_1}$ for all $i$ and applying a similar argument, we obtain a map
\begin{equation}\label{cartier}\bigoplus_i\Omega^i_{Q_n^{(1)}/\F_p[x]}[-i]\to \Omega^\bullet_{Q_n/\F_p[x]}.\end{equation}
This map is a quasi-isomorphism for $n=0,1$ by  direct computation, and therefore for general $n$ by passing to tensor powers and using that $\Omega^\bullet_{Q_n/\F_p[x]}\cong {(\Omega^\bullet_{Q_1/\F_p[x]})}^{\otimes n}$. (The learned reader will recognize that (\ref{cartier}) induces the Cartier isomorphism on cohomology groups.)

Now consider the bar resolution of the $\Z/p^2\Z[x]$-algebra $\Z/p^2\Z$ introduced in Example \ref{barresol}. Reducing modulo $p$ we obtain the bar resolution $Q_\bullet$ of the $\F_p[x]$-algebra $\F_p$. Twisting by relative Frobenius gives the bar resolution $Q_\bullet^{(1)}\to \F_p^{(1)}$. Notice that
$$
\F_p^{(1)}\cong\F_p[x]\otimes_{\F_p[x]}\F_p\cong \F_p[x]/(x^p)
$$
where the $\F_p[x]$-module structure on $\F_p[x]$ in the tensor product is given by $x\mapsto x^p$. Applying the inverse of the isomorphism (\ref{cartier}) to the terms of the bar resolution, we obtain using Theorems \ref{freeresolcotang} and \ref{freeresolderham} quasi-isomorphisms
$$
L\Omega^\bullet_{\F_p/\F_p[x]}\simeq \Omega^\bullet_{Q_\bullet/\F_p[x]}\simeq \bigoplus_i \Omega^i_{Q_\bullet/\F_p[x]}[-i]\simeq \bigoplus_iL\wedge^iL_{\F_p[x]/(x^p)/\F_p[x]}[-i]
$$
as desired.
\end{proof}

Next we compute the right hand side of the quasi-isomorphism in Lemma \ref{Cartierpol}.

\begin{lem}\label{specialquillen}
The complex $L\wedge^i L_{\mathbb{F}_p[x]/(x^p)/\mathbb{F}_p[x]}$ is acyclic outside degree $i$. Its degree $i$ homology is isomorphic to a free $(\mathbb{F}_p[x]/(x^p))$-module generated by $\gamma_i(y)$, where $y$ is a generator of the rank 1 free module $(x^p)/(x^{2p})$.
\end{lem}
\begin{proof}
By Proposition \ref{cotangquotientreg} the cotangent complex $L_{\mathbb{F}_p[x]/(x^p)/\mathbb{F}_p[x]}$ is concentrated in degree $1$ where its homology is $(x^p)/(x^{2p})$. This is a free module of rank $1$ over $\mathbb{F}_p[x]/(x^p)$. Denoting by $y$ a free generator, we have a quasi-isomorphism of complexes $$L_{\mathbb{F}_p[x]/(x^p)/\mathbb{F}_p[x]}\simeq (\mathbb{F}_p[x]/(x^p)) y[1]. $$
Taking derived exterior powers, we obtain
\begin{equation*}
L\wedge^i L_{\mathbb{F}_p[x]/(x^p)/\mathbb{F}_p[x]}\cong L\wedge^i ((\mathbb{F}_p[x]/(x^p)) y[1])\cong L\Gamma^i(\mathbb{F}_p[x]/(x^p)) y)[i]
\end{equation*}
using Quillen's shift formula (Proposition \ref{quillen}). Finally, since free modules are acyclic for the functor $\Gamma^i$, we obtain $$L\Gamma^i((\mathbb{F}_p[x]/(x^p))y)=\Gamma^i((\mathbb{F}_p[x]/(x^p))y)\cong  (\mathbb{F}_p[x]/(x^p))\gamma_i(y).$$
\end{proof}

Now we can handle the case $n=1$ of Proposition \ref{dRPDpol}.

\begin{cor}\label{dRPDpolmodp}
The derived de Rham algebra $L\Omega^\bullet_{\mathbb{F}_p/\mathbb{F}_p[x]}$  is concentrated in degree 0, and we have an isomorphism
$$
H_0(L\Omega^\bullet_{\mathbb{F}_p/\mathbb{F}_p[x]})\cong \mathbb{F}_p\langle x\rangle.
$$
\end{cor}
\begin{proof}
Applying the previous two lemmas, we compute
$$
L\Omega^\bullet_{\mathbb{F}_p/\mathbb{F}_p[x]}\cong \bigoplus_{i=0}^\infty L\wedge^i L_{(\mathbb{F}_p[x]/(x^p))/\mathbb{F}_p[x]}[-i]\cong \bigoplus_{i= 0}^\infty\Gamma^i((\mathbb{F}_p[x]/(x^p)) y)[i][-i]\cong \mathbb{F}_p[x]/(x^p)\langle y\rangle\ .$$

Finally, we identify the right hand side with $\mathbb{F}_p\langle x\rangle$ as follows. Noting that $\mathbb{F}_p[x]/(x^p)\langle y\rangle$ is generated over $\F_p$ by the elements $x^j\gamma_i(y)$ $(0\leq j\leq p-1)$, we define a map $\mathbb{F}_p[x]/(x^p)\langle y\rangle\to \mathbb{F}_p\langle x\rangle$ by sending $x^j\gamma_i(y)$ to $j!\gamma_j(x)\gamma_i(\gamma_p(x))$. A calculation using formula (\ref{gammakgammai}) and the fact that ${(ip)!}/(i!p^i)$ is a unit in $\mathbb{F}_p$ shows that this map is an isomorphism.
\end{proof}

We shall need a consequence of this result for the analogous situation with $\Z_p$-coefficients.

\begin{cor}\label{dRPDpolZp}
The derived de Rham algebra $L\Omega^\bullet_{\Z_p/\Z_p[x]}$  is concentrated in degree 0, and $H_0(L\Omega^\bullet_{\Z_p/\Z_p[x]})$ is a torsion-free $\Z_p$-module.
\end{cor}

\begin{proof}
If we compute $L\Omega^\bullet_{\Z_p/\Z_p[x]}$ as the total complex of $\Omega^\bullet_{Q_\bullet/\Z_p[x]}$  where $Q_\bullet\to\Z_p$ is the bar resolution, we have $\Omega^i_{Q_n/\Z_p[x]}=0$ for $i>n$, which shows that $L\Omega^\bullet_{\Z_p/\Z_p[x]}$ is concentrated in nonnegative homological degrees. Furthermore, we have quasi-isomorphisms
$$ \F_p\otimes_{\Z_p}^LL\Omega^\bullet_{\Z_p/\Z_p[x]}\simeq \F_p[x]\otimes_{\Z_p[x]}^LL\Omega^\bullet_{\Z_p/\Z_p[x]}\simeq L\Omega^\bullet_{\F_p/\F_p[x]}\simeq \F_p\langle x\rangle$$ by Corollaries \ref{derbasechange} and \ref{dRPDpolmodp}, so $\F_p\otimes_{\Z_p}^LL\Omega^\bullet_{\Z_p/\Z_p[x]}$ is a complex concentrated in degree 0. On the other hand, its homologies are computed by the K\"unneth spectral sequence
$$
E_2^{ij}={\rm Tor}_i^{\Z_p}(\F_p,H_j(L\Omega^\bullet_{\Z_p/\Z_p[x]}))\Rightarrow H_{i+j}(\F_p\otimes_{\Z_p}^LL\Omega^\bullet_{\Z_p/\Z_p[x]}).
$$
Since $\Z_p$ has flat cohomological dimension 1, we have $E_2^{ij}=0$ for $i>1$, and therefore the spectral sequence degenerates at $E_2$. Thus the vanishing of the abutment for $i+j\neq 0$ implies ${\rm Tor}_1^{\Z_p}(\F_p,H_j(L\Omega^\bullet_{\Z_p/\Z_p[x]}))=0$ for all $j$, i.e. all homologies of $L\Omega^\bullet_{\Z_p/\Z_p[x]}$ are torsion free. To finish the proof, we show that they are also torsion for $j>0$. To do so, we compute the complex $\Q_p\otimes_{\Z_p}^LL\Omega^\bullet_{\Z_p/\Z_p[x]}\simeq L\Omega^\bullet_{\Q_p/\Q_p[x]}$ by means of the bar resolution $\Q_p\otimes_{\Z_p} Q_\bullet$. For fixed $n$ we have $\Omega^\bullet_{(\Q_p\otimes Q_n)/\Q_p[x]}\cong (\Omega^\bullet_{\Q_p[x, x_1]/\Q_p[x]})^{\otimes n}$, and there is a quasi-isomorphism  $$\Omega^\bullet_{\Q_p[x, x_1]/\Q_p[x]}=(\Q_p[x, x_1]\stackrel d\to \Q_p[x, x_1]dx_1)\simeq (\Q_p[x]\to 0).$$ Thus ${\rm Tot}(\Omega^\bullet_{(\Q_p\otimes Q_\bullet)/\Q_p[x]})$ is the chain complex associated with the constant simplicial object $\Q_p[x]_\bullet$; in particular, it is acyclic in positive degrees.
\end{proof}

To pass from the case $n=1$ of Proposition \ref{dRPDpol} to the general case, we need:

\begin{lem}[D\'evissage]\label{devissage}
Let $\phi\colon A_\bullet\to B_\bullet$ be a morphism of complexes of $\mathbb{Z}/p^n\Z$-modules. If the base change map $A_\bullet\otimes^L_{\mathbb{Z}/p^n\Z}\mathbb{F}_p\to B_\bullet\otimes^L_{\mathbb{Z}/p^n\Z}\mathbb{F}_p$ is a quasi-isomorphism, then so is $\phi$.
\end{lem}
\begin{proof}
By exactness of the derived tensor product we reduce to the case $B_\bullet=0$. Moreover, after replacing $A_\bullet$ by a complex with free terms we may assume $A_\bullet$ has free terms. Thus we have to show that the acyclicity of $A_\bullet\otimes_{\Z/p^n\Z}\F_p$ implies that of $A_\bullet$. We use induction on $n$. If $\alpha_i\in A_i$ satisfies $d\alpha_i=0$, we have $\alpha_i=d\alpha_{i+1}+p\beta_i$ for some $\alpha_{i+1}\in A_{i+1}$ and $\beta_i\in A_i$ by acyclicity of $A_\bullet/pA_\bullet=A_\bullet\otimes_{\Z/p^n\Z}\F_p$. Since $A_\bullet$ has free terms,  multiplication by $p$ induces an isomorphism $A_\bullet/pA_\bullet\cong pA_\bullet/p^2A_\bullet$, and hence the complex of $\Z/p^{n-1}\Z$-modules $pA_\bullet$ is acyclic by induction on $n$. As $d(p\beta_i)=0$ by construction, we then find $\beta_{i+1}\in A_{i+1}$ such that $p\beta_i=d(p\beta_{i+1})$, so finally $\alpha_i=d(\alpha_{i+1}+p\beta_{i+1})$.
\end{proof}

\noindent{\em Proof of Proposition \ref{dRPDpol}.} The key point is the construction of a map
$$
\Z_p\langle x\rangle \to L\Omega^\bullet_{\mathbb{Z}_p/\mathbb{Z}_p[x]}
$$
that lifts the map
$
\F_p\langle x\rangle \to L\Omega^\bullet_{\mathbb{F}_p/\mathbb{F}_p[x]}
$
inducing the isomorphism of Corollary \ref{dRPDpolmodp} and is compatible with the filtrations on both sides.
Once such a map has been constructed, we obtain maps
$$
\mathbb{Z}/p^n\Z\langle x\rangle\to L\Omega^\bullet_{(\mathbb{Z}/p^n\Z)/(\mathbb{Z}/p^n\Z)[x]}
$$
for all $n$ by reducing modulo $p^n$. These maps are isomorphisms modulo $p$ by Corollary \ref{dRPDpolmodp}, hence isomorphisms by Lemma \ref{devissage}.

The idea of the following construction is due to Bhargav Bhatt (private communication). By Corollary \ref{dRPDpolZp} we may replace $L\Omega^\bullet_{\Z_p/\Z_p[x]}$ by its $H_0$ and consider it as an honest $\Z_p[x]$-algebra. We then claim that the structure map $\varphi:\,\Z_p[x]\to L\Omega^\bullet_{\Z_p/\Z_p[x]}$ extends to a map $\widetilde\varphi:\,\Zp\langle x\rangle{\to}L\Omega^\bullet_{\Z_p/\Z_p[x]}$. To see this, denote by $I$ the kernel of the augmentation map $L\Omega^\bullet_{\Z_p/\Zp[x]}\to \Zp$. After reducing modulo $p$ we have $$L\Omega^\bullet_{\Z_p/\Zp[x]}\otimes_{\Zp}\F_p\cong L\Omega^\bullet_{\F_p/\Fp[x]}\cong \F_p\langle x\rangle$$ by Corollary \ref{dRPDpolmodp}, with $I$ mapping to the ideal $\langle x\rangle$. In particular, for $f\in I$ with image $\overline f$ in $\langle x\rangle$ we have $$\overline{f^p}=\overline{f}^p=p!\gamma_p(\overline{f})=0,$$ showing that $f^p$ is divisible by $p$ in $L\Omega^\bullet_{\Zp/\Zp[x]}$. As $L\Omega^\bullet_{\Z_p/\Z_p[x]}$ is torsion free by Corollary \ref{dRPDpolZp}, there is a unique element ${f^p}/{p}\in L\Omega^\bullet_{\Zp/\Zp[x]}$ with $p({f^p}/{p})=f^p$. Since $f^p\in I$ and $\Zp$ is torsion free, we in fact have ${f^p}/{p}\in I$. Applying this to ${f^p}/{p}$ in place of $f$ we find ${f^{p^2}}/{p^p}\in pI$. Iterating  $k$ times we deduce $p^{\frac{p^k-1}{p-1}}\mid f^{p^k}$. For a positive integer $n$ with $p$-adic expansion $n=a_rp^r+\dots+a_1p+a_0$  this implies the divisibility $$p^{\sum_{k=0}^ra_k\frac{p^k-1}{p-1}}\mid f^{\sum_{k=0}^ra_kp^k}=f^n\ .$$ Here the left hand side is exactly the $p$-part of $n!$, so we conclude (using torsion freeness again) that there is a unique element $(f^n/n!)\in L\Omega^\bullet_{\Zp/\Zp[x]}$ with $n!(f^n/n!)=f^n$. Applying this to $f=\varphi(x)\in I$ we may then unambiguously set $\widetilde{\varphi}(\gamma_n(x)):={\varphi(x)^n}/{n!}$ for all $n$, which defines $\widetilde\phi$.

For the compatibility of the divided power filtration with the Hodge filtration it suffices to show that $\widetilde{\varphi}(\gamma_n(x))\in F^n L\Omega^\bullet_{\Zp/\Zp[x]}$ for all $n$.  Since $F^1 L\Omega^\bullet_{\Zp/\Zp[x]}$ is the kernel of the augmentation map to $\Z_p$, it contains $\varphi(x)$, and therefore $F^n L\Omega^\bullet_{\Zp/\Zp[x]}$ contains  $\varphi(x)^n=\widetilde{\varphi}(x^n)=n!\widetilde{\varphi}(\gamma_n(x))$. However, the graded pieces of the Hodge filtration are given by  $$L\wedge^i L_{\Zp/\Zp[x]}[-i]\cong \Gamma^i((x)/(x^2))$$ by Propositions \ref{gridr}, \ref{cotangquotientreg} and \ref{quillen}, and these $\Z_p$-modules are torsion free for all $i\geq 0$.  We deduce $\widetilde{\varphi}(\gamma_n(x))\in F^nL\Omega^\bullet_{\Zp/\Zp[x]}$ as required.
\enddem

The last result in this section may be viewed as an analogue of Proposition \ref{cotangquotientreg}.

\begin{theo}[Bhatt]\label{dRPD}
Assume that $A\to B$ is a surjective homomorphism of flat $\mathbb{Z}/p^n\Z$-algebras with kernel $I=(f)$ generated by a nonzerodivisor $f\in A$. The derived de Rham algebra $L\Omega^\bullet_{B/A}$ is concentrated in degree $0$ and we have an isomorphism of $A$-algebras
$$
H_0(L\Omega^\bullet_{B/A})\cong A\langle t\rangle/(t-f).
$$
Moreover, the  Hodge filtration on $L\Omega^\bullet_{B/A}$ corresponds on the right hand side to the filtration induced by the divided power filtration of $A\langle t\rangle$.
\end{theo}

The proof uses the following lemma.

\begin{lem} The $\Z/p^n\Z[x]$-algebras $A$ and $\Z/p^n\Z\langle x\rangle$ are Tor-independent, where $A$ is
considered  as a $\Z/p^n\Z[x]$-algebra via the map $x\mapsto f$.
\end{lem}

\begin{proof} Take a resolution $F_\bullet \to \mathbb{Z}/(p^n)\langle x\rangle$ by free $\mathbb{Z}/p^n\Z[x]$-modules. We show that $A\otimes_{\mathbb{Z}/(p^n)[x]}F_\bullet$ is acyclic in positive degrees. To do so, we reduce by d\'evissage (Lemma \ref{devissage}) to proving acyclicity of $$\F_p\otimes_{\Z/p^n\Z}(A\otimes_{\mathbb{Z}/p^n\Z[x]}F_\bullet)\cong (\F_p\otimes_{\Z/p^n\Z}A)\otimes_{\F_p[x]}(\F_p\otimes_{\Z/p^n\Z}F_\bullet). $$
Since both $\Z/p^n\Z\langle x\rangle$ and the terms of $F_\bullet$ are free over $\Z/p^n\Z$, the base change $\F_p\otimes_{\Z/p^n\Z}F_\bullet$ is a free resolution of $\F_p\langle x\rangle$ over $\F_p[x]$, so we reduce to proving acyclicity of $A\otimes_{\mathbb{Z}/(p^n)[x]}F_\bullet$ in the case $n=1$. But then $\mathbb{F}_p\langle x\rangle$ is isomorphic to a direct sum of copies of $\mathbb{F}_p[x]/(x^p)$  as an $\mathbb{F}_p[x]$-module, so it suffices to show Tor-independence of $A$ and $\mathbb{F}_p[x]/(x^p)$ over $\F_p[x]$. This is verified as in the proof of Proposition \ref{cotangquotientreg}.
\end{proof}

\noindent{\em Proof of theorem \ref{dRPD}.}  As in the proof of Proposition \ref{cotangquotientreg}, we see that the  $\Z/p^n\Z[x]$-algebras $\Z/p^n\Z$ and  $A$ are Tor-independent, and therefore by the base change property of derived de Rham algebras we have a quasi-isomorphism
$$
L\Omega^\bullet_{(\mathbb{Z}/p^n\Z\otimes_{\mathbb{Z}/p^n\Z[x]} A)/A}\simeq
L\Omega^\bullet_{(\mathbb{Z}/p^n\Z)/({\mathbb{Z}/p^n\Z[x]})}\otimes_{\mathbb{Z}/p^n\Z[x]}^LA.
$$
On the other hand, we have
$$
L\Omega^\bullet_{(\mathbb{Z}/p^n\Z)/({\mathbb{Z}/p^n\Z[x]})}\otimes_{\mathbb{Z}/p^n\Z[x]}^LA\simeq \mathbb{Z}/p^n\Z\langle x\rangle\otimes_{\mathbb{Z}/p^n[x]}^L A\cong \mathbb{Z}/p^n\Z\langle x\rangle\otimes_{\mathbb{Z}/p^n[x]}A
$$
in view of Proposition \ref{dRPDpol} and the lemma above.

So we obtain that $L\Omega^\bullet_{B/A}$ is concentrated in degree zero and compute its $0$-th cohomology as
$$
\mathbb{Z}/p^n\Z\langle x\rangle\otimes_{\mathbb{Z}/p^n\Z[x]}A\cong\cok (\mathbb{Z}/p^n\Z[x]\langle t\rangle\overset{(x-t)\cdot}{\to}\mathbb{Z}/p^n\Z[x]\langle t\rangle)\otimes_{\mathbb{Z}/p^n\Z[x]}A
\cong A\langle t\rangle/(f-t).
$$
The equality of the Hodge filtration with the PD filtration follows from the equality of these filtrations in Proposition \ref{dRPDpol}.
\enddem

\begin{rema}\rm
Note that the divided power structure on $A\langle t\rangle$ induces one one the quotient by $(f-t)$. Indeed, since $f$ is not a zerodivisor in $B$, for $\alpha\in A\langle t\rangle$ we have $(f-t)\alpha\in \langle t\rangle$ if and only if $\alpha\in\langle t\rangle$, so we may consider the $A\langle t\rangle/(t-f)$-module $\langle t\rangle/(t-f)\langle t\rangle$ and induce divided power operations on the quotient using axioms (3) and (4) of Definition \ref{divpower}.
\end{rema}

\subsection{The $p$-completed derived de Rham algebra of $\OCK/\OK$.}

We now apply the theorem to the surjection $\theta:\, A_{\rm inf}\to \OCK$ introduced in the previous section. Modulo $p^n$ it induces a map $\theta_n:\, A_{\rm inf}/(p^n)\to\OCK/(p^n)$. By Corollary \ref{corfontaine} the kernel of $\theta_n$ is a principal ideal; denote by $\xi_n$ a generator.

\begin{cor}\label{dROCKAinf} The derived de Rham algebra $L\Omega^\bullet_{(\OCK/(p^n))/(A_{\rm inf}/(p^n))}$ is concentrated in degree 0, and we have a filtered isomorphism
$$H_0(L\Omega^\bullet_{(\OCK/(p^n))/(A_{\rm inf}/(p^n))})\cong (A_{\rm inf}/(p^n))\langle t\rangle/(t-\xi_n).$$
\end{cor}
\begin{proof} In order to be able to apply Theorem \ref{dRPD} we have to check that the generator $\xi_n$ of  $\ker(\theta_n)$ is not a zero-divisor.
Denote by $\xi_1$ its image in $A_{inf}/(p)=(\OKB/(p))^{\rm perf}$, and represent $\xi_1$ by a sequence $(\xi_1^{(k)})$ of elements of the form $\xi_1^{(k)}=\lambda_k+(p)$ with some choice  $\lambda_k=\sqrt[p^k]{p}$ of a compatible system of $p$-power roots of $p$. Were $\xi_1$ a zero divisor in $A_{inf}/(p)$, there would be a sequence $(\mu_k)\subset\OCK$ with $\mu_{k+1}^p=\mu_k$ such that $v_p(\mu_k\lambda_k)\geq 1$ for all $k\geq 0$. Since $v_p(\lambda_k)=1/p^k$, we obtain $v_p(\mu_k)\geq 1-1/p^k$ for all $k\geq 0$. However, this means that $v_p(\mu_k)=pv_p(\mu_{k+1})\geq p-1/p^k\geq 1$ for all $k\geq 0$, so the class of the sequence $(\mu_k+(p))$ is zero in $A_{inf}/(p)$, a contradiction. Thus $\xi_1$ is not a zero-divisor, and neither is $\xi_n$ by a d\'evissage argument.
\end{proof}

Assume now that $K|\Q_p$ is an {\em unramified extension}. In this case $\Ainf=W((\OKB/(p))^{\rm perf})$ has an $\OK$-algebra structure via the canonical map
$$
\OK=W(\OK/(p))\to W((\OKB/(p))^{\rm perf})
$$
lifting the inclusion $\OK/(p)\to (\OKB/(p))^{\rm perf}$ according to Proposition \ref{lifthom} (2). Moreover, we have an $\OK$-algebra map $\Ainf\to \OCK$. An important observation of Bhatt is that modulo $p^n$ we may compare the associated derived de Rham algebra with that of $\OKB$ over $\OK$. This is enabled by the following general lemma.

\begin{lem}\label{relperfcot0hodgetrunc}
Let $A\to B$ be a flat map of $\mathbb{Z}/p^n\Z$-algebras such that both $A/pA$ and $B/pB$ are perfect $\F_p$-algebras.
\begin{enumerate}
\item We have $L_{B/A}\simeq 0$.
\item If $C$ is a $B$-algebra, we have a quasi-isomorphism $$L\widehat\Omega^\bullet_{C/B}\simeq L\widehat\Omega^\bullet_{C/A}$$ of Hodge-completed derived de Rham algebras.
\end{enumerate}
\end{lem}

In the second statement the quasi-isomorphism is to be understood as a projective system of compatible quasi-isomorphisms $L\Omega^\bullet_{C/B}/F^i\simeq L\Omega^\bullet_{C/A}/F^i$.

\begin{proof} It is enough to verify the first statement for $n=1$ by d\'evissage (Lemma \ref{devissage}). So assume $A$ and $B$ are perfect $\F_p$-algebras, and recall the basics about relative Frobenii explained in Facts \ref{frob}. The $A$-isomorphism $B^{(1)}\stackrel\sim\to B$ established there induces an isomorphism of cotangent complexes $L_{B^{(1)}/A}\stackrel\sim\to L_{B/A}$. To compute it, consider the standard resolution $P_\bullet\to B$. As in the proof of Lemma \ref{Cartierpol}, for a free $A$-algebra $P=A[x_i\mid i\in I]$ if we identify $A$ with $A^{(1)}$ via the Frobenius map, the morphism of $A$-algebras $P^{(1)}\to P$ constructed above becomes identified with the map $P\to P$ which is the identity on $A$ and sends each $x_i$ to $x_i^p$. In view of the equalities $dx_i^p=px_i^{p-1}dx_i=0$ in characteristic $p$, the map $\Omega^1_{P^{(1)}/A}\to  \Omega^1_{P/A}$ induced by the morphism $P^{(1)}\to P$ is 0. Applying this to the terms of $P_\bullet$, we obtain that the isomorphism $B^{(1)}\stackrel\sim\to B$ induces the zero map $L_{B^{(1)}/A}\to L_{B/A}$, which is only possible if $L_{B/A}=0$.

For the second statement, observe that  by functoriality of modules of differentials the morphism $A\to B$  induces a morphism of derived de Rham algebras $L\Omega^\bullet_{C/A}\to L\Omega^\bullet_{C/B}$ compatible with the Hodge filtration, whence also a morphism $L\widehat\Omega^\bullet_{C/A}\to L\widehat\Omega^\bullet_{C/B}$ on Hodge completions. It is an isomorphism if and only if the induced map on associated graded objects is, so by Proposition \ref{gridr} it suffices to show that $L_{C/A}\to L_{C/B}$ is an isomorphism. In view of the transitivity triangle for cotangent complexes (Proposition \ref{cotangexacttriang}), this in turn follows from the vanishing of $L_{B/A}$, which is statement (1).
\end{proof}

\begin{rema}\label{conjrema}\rm
In (\cite{Bh2}, Corollary 3.8 and Lemma 8.3(5)), Bhatt proves that the conclusion of statement (2) holds also for the uncompleted derived de Rham algebras: under the assumptions of the lemma we have a quasi-isomorphism $L\Omega^\bullet_{C/B}\simeq L\Omega^\bullet_{C/A}$. The proof follows the same pattern as above, except that instead of the Hodge filtration on $L\Omega^\bullet_{B/A}$ it uses the {\em conjugate filtration} $F^{\rm conj}$. It is an {\em increasing} filtration induced by canonical truncations on the de Rham complexes $\Omega^\bullet_{P_n/A}$:

 $$F^{\rm conj}_i(\Omega^j_{P_n/A}):=\begin{cases} \Omega^j_{P_n/A}&\text{if }j<i\\ \Ker(\Omega^i_{P_n/A}\overset{d}{\to} \Omega^{i+1}_{P_n/A})&\text{if }j=i\\ 0&\text{if }j>i\ .\end{cases}$$
The key point then is that modulo $p$ the Cartier isomorphism splits the conjugate filtration: there is a direct sum decomposition
\begin{equation*}
\mathrm{gr}^i_{F_{\rm conj}}(L\Omega^\bullet_{B/A})\cong L\wedge^i L_{B^{(1)}/A}[i]
\end{equation*}
for a map of $\F_p$-algebras $A\to B$ induced by the Cartier isomorphism (\cite{Bh2}, Lemma 3.5). The rest of the argument is then the same as above.
\end{rema}

\begin{cor}\label{corperfect}
Assume $K|\Q_p$ is unramified. The $\OK$-algebra map $\Ainf\to\OCK$ induces an isomorphism of Hodge-completed derived de Rham algebras
\begin{equation}\label{corperfecteq}
L\widehat\Omega^\bullet_{(\OCK/(p^n))/\Ainf/(p^n)}\simeq L\widehat\Omega^\bullet_{(\OCK/(p^n))/(\OK/(p^n))}=L\widehat\Omega^\bullet_{(\OKB/(p^n))/(\OK/(p^n))}
\end{equation}
for all $n>0$. Hence $L\widehat\Omega^\bullet_{(\OKB/(p^n))/(\OK/(p^n))}$ is concentrated in degree 0, where its homology is isomorphic to the completion of $(A_{\rm inf}/(p^n))\langle t\rangle/(t-\xi_n)$ with respect to its divided power filtration.
\end{cor}

\begin{proof}
As $K|\Q_p$ is unramified, the ring $\OK/(p)$ is a finite field, so the map of $\F_p$-algebras $\OK/(p)\to \Ainf/(p)=(\OK/(p))^{\rm perf}$ is a morphism of perfect $\F_p$-algebras. Therefore Lemma \ref{relperfcot0hodgetrunc} applies to the map $\OK/(p^n)\to \Ainf/(p^n)$ and yields isomorphism (\ref{corperfecteq}). The second statement follows from Corollary \ref{dROCKAinf}.
\end{proof}

Finally, define the derived $p$-adic completion of $L\widehat\Omega^\bullet_{\OKB/\OK}$ by
$$
L\widehat\Omega^\bullet_{\OKB/\OK}\widehat\otimes\Z_p:=R\varprojlim(L\widehat\Omega^\bullet_{\OKB/\OK}\otimes^L\Z/p^n\Z)
$$
(see e.g. \cite{weibel}, \S 3.5 for derived inverse limits). Here the right hand side is to be understood as the projective system of the $R\varprojlim(L\Omega^\bullet_{\OKB/\OK}/F^i\otimes^L\Z/p^n\Z)$ for all $i$. Since $\OKB$ is flat over $\OK$, we have
$$
L\widehat\Omega^\bullet_{\OKB/\OK}\otimes^L\Z/p^n\Z\simeq L\widehat\Omega^\bullet_{(\OKB/(p^n))/(\OK/(p^n))}
$$
by base change (Corollary \ref{derbasechange}). The complexes on the right hand side are computed by Corollary \ref{corperfect}. In particular, they are concentrated in degree 0 and the maps in their inverse system are surjective. Hence the derived inverse limit is the usual inverse limit over $n$ (see \cite{weibel}, Proposition 3.5.7), and we obtain

\begin{cor}\label{corperfect2}
If  $K|\Q_p$ is unramified, the derived $p$-adic completion of $L\widehat\Omega^\bullet_{\OKB/\OK}$ is concentrated in degree 0, where its homology is isomorphic to the completion of $A_{\rm inf}\langle t\rangle/(t-\xi)$ with respect to its divided power filtration.
\end{cor}

Here $\xi$ is a generator of the kernel of $\theta:\, A_{\rm inf}\to \OCK$ (Corollary \ref{corfontaine}), and we have used the $p$-adic completeness of $\Ainf$.

\begin{rema}\label{acris}\rm
Using Remark \ref{conjrema} we also obtain that the uncompleted derived de Rham algebra $L\Omega^\bullet_{\OKB/(p^n)/(\OK/(p^n))}$ is concentrated in degree 0, where its homology is isomorphic to $(A_{\rm inf}/(p^n))\langle t\rangle/(t-\xi_n)$. In the inverse limit we obtain that
$L\Omega^\bullet_{\OKB/\OK}\widehat\otimes\Z_p$ is isomorphic to the $p$-adic completion of $A_{\rm inf}\langle t\rangle/(t-\xi).$
This is Fontaine's ring $A_{{\rm cris},K}$ as defined in \cite{F2}.
\end{rema}

\section{Construction of period rings}

\subsection{Construction and basic properties of $\Bdr$}

Let $K$ be a finite extension of $\Q_p$, with algebraic closure $\overline K$. Following Beilinson, we define
$$
A_{{\rm dR},K}:=L\widehat\Omega^\bullet_{\OKB/\OK}
$$
and
$$
\Bdr^+:=A_{{\rm dR},K}\widehat\otimes\Q_p=(A_{{\rm dR},K}\widehat\otimes\Z_p)\otimes\Q_p.
$$
Note that by construction these objects are equipped with an action of $G_K:=\Gal(\overline K|K)$ and are complete with respect to the Hodge filtration.

When clear from the context, we shall drop the subscript $K$ from the notation $A_{{\rm dR},K}$. However, $\Bdr^+$ does not depend on $K$ any more, as the following lemma shows.

\begin{lem} Let $K'|K$ be a finite extension. The natural map $A_{{\rm dR},K}{\to}A_{{\rm dR},K'}$ induced by the maps $\Omega^\bullet_{\cdot/K}\to \Omega^\bullet_{\cdot/K'}$ on modules of differentials
   gives rise to an isomorphism $$A_{{\rm dR},K}\widehat{\otimes}\mathbb{Q}_p\overset{\sim}{\to}A_{{\rm dR},K'}\widehat{\otimes}\mathbb{Q}_p.$$
\end{lem}
\begin{proof}
By Theorem \ref{cotangexacttriang} the sequence of maps $\CO_{K}\to \CO_{K'}\to \OKB$ yields an exact triangle
\begin{equation*}
\OKB\otimes_{\OK}^{L} L_{\CO_{K'}/\OK}\to L_{\OKB/\CO_{K}}\to L_{\OKB/\CO_{K'}}\to \OKB\otimes_{\OK}^{L} L_{\CO_{K'}/OK}[1]
\end{equation*}
of cotangent complexes. By Lemma \ref{simpresOmega} we have a quasi-isomorphism $ L_{\CO_{K'}/\OK}\simeq \Omega^1_{\CO_{K'}/\OK}$. According to the structure of $\Omega^1_{\CO_{K'}/\OK}$ recalled in Facts \ref{localfacts}, the latter is a finitely generated torsion $\Z_p$-module, i.e. a finite abelian $p$-group, whence
$$
L_{\CO_{K'}/\OK}\widehat{\otimes}\mathbb{Q}_p=0.
$$
Since the derived tensor product with $\Q_p$ is an exact functor, we conclude that there is a quasi-isomorphism
\begin{equation}\label{lo}
L_{\OKB/\CO_{K}}\widehat{\otimes}\mathbb{Q}_p\to L_{\OKB/\CO_{K'}}\widehat{\otimes}\mathbb{Q}_p.
\end{equation}
Now we use the Hodge filtration on $A_{{\rm dR},K}$. By Proposition \ref{gridr} we have quasi-isomorphisms
$$
{\rm gr}^i_FA_{{\rm dR},K}\simeq L\wedge^iL_{\OKB/\OK}[-i]
$$
and similarly for $A_{{\rm dR},K'}$. So from (\ref{lo}) we derive quasi-isomorphisms
$$
{\rm gr}^i_FA_{{\rm dR},K}\widehat{\otimes}\mathbb{Q}_p\simeq {\rm gr}^i_FA_{{\rm dR},K'}\widehat{\otimes}\mathbb{Q}_p
$$
for all $i$, whence quasi-isomorphisms
$$
(A_{{\rm dR},K}/F^i)\widehat{\otimes}\mathbb{Q}_p\simeq (A_{{\rm dR},K'}/F^i)\widehat{\otimes}\mathbb{Q}_p
$$
for all $i$ by induction on $i$. As by definition $A_{{\rm dR},K}$ is the projective system of the $A_{{\rm dR},K}/F^i$ and similarly for $A_{{\rm dR},K'}$, we are done.
\end{proof}

By the lemma, when computing $\Bdr^+$ we may assume it is defined as $A_{{\rm dR},K}\widehat\otimes\Q_p$ with a finite {\em unramified} extension $K|\Q_p$. Corollary \ref{corperfect2} then implies that we may view $\Bdr^+$ as a complete filtered ring and not just as a projective system of complexes.

\begin{prop}\label{BdRdiscreteval} The ring
$\Bdr^+$ is a complete discrete valuation ring with residue field $\mathbb{C}_K$. Moreover, its filtration ${\rm Fil}^i$ by powers of the maximal ideal satisfies a $G_K$-equivariant isomorphism
\begin{equation}\label{grfil}
{\rm Fil}^i/{\rm Fil}^{i+1}\cong\C_K(i).
\end{equation}
\end{prop}

\begin{proof} As already remarked, by the previous lemma we may assume $K$ is unramified over $\mathbb{Q}_p$. Then by Corollary \ref{corperfect2} the Hodge filtration on $A_{\rm dR}\widehat\otimes\Z_p$ is the filtration by divided powers of the ideal $\ker(\theta_{\rm dR})$, where $\theta_{\rm dR}$ is the natural surjection $A_{\rm dR}\widehat\otimes\Z_p\to (A_{\rm dR}/F^1)\widehat\otimes\Z_p=\OCK$. After tensoring by $\Q$ this becomes the filtration by powers of $\ker(\theta_{\rm dR}\otimes\Q)$ which is the maximal ideal of $\Bdr^+$ as the associated quotient is $\OCK\otimes\Q=\C_K$. Moreover, $\Bdr^+$ is complete with respect to the filtration since $A_{\rm dR}$ is the Hodge-completed de Rham algebra. By Corollaries \ref{coruniv} and \ref{cp1} we have a $G_K$-equivariant isomorphism
\begin{equation}\label{grfil2}
(\ker(\theta_{\rm dR})/\ker(\theta_{\rm dR})^2)\otimes\Q\cong \C_K(1)
\end{equation}
which shows in particular that this is a $\C_K$-vector space of dimension 1, and therefore $\ker(\theta_{\rm dR})\otimes\Q$ is a principal ideal by completeness of $\Bdr^+$. Its powers define the filtration ${\rm Fil}^i$, and the isomorphism (\ref{grfil}) follows from (\ref{grfil2}).
\end{proof}

\begin{prop}\label{GKsplitting}
There exists a $G_K$-equivariant embedding $\overline{K}\hookrightarrow \Bdr^+$ such that the composite $\overline K\to\Bdr^+\twoheadrightarrow \mathbb{C}_K$ is the natural embedding $\overline K\hookrightarrow \C_K$.
\end{prop}

\begin{proof}
By Corollaries \ref{cp2} and \ref{corsimpresomega} the cotangent complex $L_{\OKB/\OK}$ is concentrated in degree 0 where its homology is torsion. Thus $L_{\OKB/\OK}\otimes\Q\simeq 0$ and therefore $$L\widehat\Omega^\bullet_{\OKB/\OK}\otimes\Q\simeq (L\widehat\Omega^\bullet_{\OKB/\OK}/F^1)\otimes\Q=\overline K.$$ Thus the natural map $L\widehat\Omega^\bullet_{\OKB/\OK}\to L\widehat\Omega^\bullet_{\OKB/\OK}\widehat\otimes\Z_p$ induces the required map after tensoring with $\Q$.
\end{proof}

\begin{rema}\rm
There does \emph{not} exist a $G_K$-equivariant splitting $\mathbb{C}_K\hookrightarrow \Bdr^+$. This would entail a $G_K$-equivariant isomorphism  $\Bdr^+\cong\mathbb{C}_K[[ t]]$ which is not the case.
\end{rema}

\begin{cor} We have $(\Bdr^+)^{G_K}=K$.
\end{cor}

\begin{dem}
By Tate's theorem cited in the introduction, we have $\C_K(i)^{G_K}=0$ for $i>0$, hence $({\rm Fil}^1)^{G_K}=0$ by induction from the second statement of Proposition \ref{BdRdiscreteval}
and completeness. We thus obtain an injection $(\Bdr^+)^{G_K}\hookrightarrow (\Bdr^+/{\rm Fil^1})^{G_K}\cong \C_K^{G_K}=K$. Since we also have an injection $K={\overline K}^{G_K}\hookrightarrow (\Bdr^+/{\rm Fil^1})^{G_K}$ by the proposition, the corollary follows.
\end{dem}

\begin{defi}\rm
The field $\Bdr$ of $p$-adic periods is the fraction field  of the discrete valuation ring $\Bdr^+$.
\end{defi}

Thus $\Bdr$ comes equipped with a $G_K$-action for which $\Bdr^{G_K}=K$ by the last corollary, and a filtration ${\rm Fil}^i$  inherited from $\Bdr^+$. Its associated graded ring is
$$
B_{\rm HT}:={\rm gr}^\bullet_{\rm Fil}\Bdr\cong\bigoplus_{i\in Z}\C_K(i)
$$
in view of the last statement of Proposition \ref{BdRdiscreteval}.

\subsection{Deformation problems and period rings}

In Proposition \ref{propfontaine} we saw that truncations of Fontaine's ring $\Ainf$ yield solutions to a universal deformation problem. In \cite{F2}, Fontaine shows a similar universal deformation property for the ring $\Acris$ considered in Remark \ref{acris}. Our main goal in this subsection is establish a property of this type for the ring $\Adr$, thereby making the link with Fontaine's original constructions.

As $\Adr$ carries a divided power structure, the deformation problem will have to take it into account. First some definitions. A divided power ideal, or PD-ideal for short, in a ring $B$ is an ideal $I\subset B$ together with a divided power structure on $B$ by $I$ in the sense of Definition \ref{divpower}, such that the maps $\gamma_i:\, I\to B$ moreover satisfy $\gamma_1=\id_I$, $\gamma_i(I)\subset I$ for $i>1$ as well as the supplementary axiom
\begin{equation}
\gamma_n(\gamma_m(a))=\frac{(mn)!}{(m!)^nn!}\gamma_{nm}(a).
\end{equation}

\begin{exs}\rm ${}$\smallskip

\noindent (1) If $B$ is a $\Q$-algebra, the usual divided power operations $\gamma_i(a)=a^i/i!$ equip every ideal $I\subset B$ with a PD-structure.\smallskip

\noindent (2) If $K|\Q_p$ is a finite unramified extension, restricting the above divided power operations on $K$ to $(p)\subset \OK$ equips $(p)$ with the structure of a $p$-ideal. This follows from the well-known formula $v_p(n!)=[{n}/{p}]+[{n}/{p^2}]+[{n}/{p^3}]\dots$
However, there may not be a PD-structure on $(p)$ for general $K$; in fact, such a PD-structure exists if and only if the ramification index of $K|\Q_p$ is $<p$ (see \cite{BO}, Example 3.2(3)).
\end{exs}

The divided powers of a $PD$-ideal $I\subset B$ are defined by
$$
I^{[i]}:=(\gamma_{i_1}(x_1)\cdots \gamma_{i_r}(x_r)\mid x_j\in I, i_1+\cdots + i_r\geq i)
$$
generalizing the special case discussed before Remark \ref{divpowrema}. Note that $I^{[2]}=I^2$ but these ideals differ in general if $i>2$.

Finally, in the case when $B$ is an $\OK$-algebra for $K|\Q_p$ unramified, we say that the $PD$-structure on $I\subset B$ is compatible with that on $(p)\subset \OK$ if $\gamma_i(bp)=b^ip^i/i!$ for all $b\in B$ for which $bp\in I$.

\begin{defi}\label{pdthick}\rm A surjection $\rho:\, B\to A$ of $p$-adically complete $\OK$-algebras is {\em an order $i$ PD-thickening} for some $i>0$ if $\ker(\rho)$ is a PD-ideal with PD-structure compatible with that of $(p)\subset \OK$, and moreover $\ker(\rho)^{[i+1]}=0$.
\end{defi}

For fixed $A$ such pairs $(B,\rho)$ form a natural category, and an initial element in this category (if exists) is called a {\em universal $p$-adically complete PD-thickening of order $i$}.

\begin{theo} If $K|\Q_p$ is unramified, then for all $i$ the $\OK$-algebra $(\Adr/F^{i+1})\widehat\otimes\Z_p$ is the universal $p$-adically complete PD-thickening of order $i$ of $\OCK$ over $\OK$.
\end{theo}

\begin{proof}
We use the description of Corollary \ref{corperfect2}. For simplicity we treat the case  $K=\Q_p$; the general case follows by base change to $\OK$. Given a $p$-adically complete PD-thickening  $\rho:\,B\to\OCK$ of order $i$ over $\OK$, we first show that there exists a unique $p$-adically continuous homomorphism $\Ainf\to B$ making the diagram
$$
\begin{CD}
0 @>>> \ker(\theta) @>>> \Ainf @>{\theta}>> \OCK @>>> 0 \\
&& @VVV @VVV @VV{\id}V \\
0 @>>> \ker(\rho) @>>> B @>{\rho}>> \OCK @>>> 0
\end{CD}
$$
commute. As $B$ is $p$-adically complete, we reduce by Proposition \ref{lifthom} (2) to constructing a unique map $\tau:\,\Ainf/(p)\to B/(p)$ making the mod $p$ diagram commute. Given an element $(x_n)$ in $\Ainf/(p)=(\OKB/(p))^{\rm perf}$, consider the unique $p$-th root of $(x_n)$ in $\Ainf/(p)$, namely the shifted sequence $(x_{n+1})$. We must have $\tau((x_{n+1}))^p=\tau((x_n))^p$, and $(x_{n+1})$ maps to $x_1$ in $\OKB/(p)$. Therefore we must have $\tau(x_n)=\widehat x_1^p$ for a lifting $\widehat x_1\in B/(p)$ of $x_1\in \OKB/(p)$. On the other hand, the $p$-th power $\widehat x_1^p$ does not depend on $\widehat x_1$, for if $y_1$ is another lifting of $x_1$, then $\widehat{x_1}-y_1\in\ker(\rho\!\!\mod p)$, and therefore $\widehat{x_1}^p-y_1^p=(\widehat{x_1}-y_1)^p=p!\gamma_p(\widehat{x_1}-y_1)=0$ by compatibility of the PD-structure on $\ker(\rho)$ with that on $(p)\subset\OK$. This shows that the map $\tau(x_n):=\widehat x_1^p$ is well defined, and is the only possible choice for $\tau$.

Next, consider $\Ainf$ as a $\Z_p[t]$-algebra via the map $t\mapsto\xi$, where $\xi$ is a generator of $\ker(\theta)$ (Corollary \ref{corfontaine}). By the diagram the composite map $\Z_p[t]\to \Ainf\to B$ sends $t$ to an element of the PD-ideal $\ker(\rho)$, hence it extends uniquely to a $\Z_p$-algebra map $\Z_p\langle t\rangle \to B.$ The commutative diagram
$$
\begin{CD}
\Z_p[t] @>>> \Ainf \\
@VVV @VVV \\
\Z_p\langle t\rangle @>>> B
\end{CD}
$$
induces a map $\Z_p\langle t\rangle\otimes_{\Z_p[t]}\Ainf\to B$. But
$$\Z_p\langle t\rangle\otimes_{\Z_p[t]}\Ainf\cong \Ainf\langle t\rangle/(t-\xi)$$
and the map respects the filtration by powers of PD-ideals, so we conclude by Corollary \ref{corperfect2}.
\end{proof}

\begin{remas}\rm ${}$
\begin{enumerate}
\item The proof above shows that $\Adr\widehat\otimes\Z_p$ is the universal $p$-adically complete pro-PD-thickening of $\OCK$ over $\OK$, where a pro-PD thickening $(B, \rho)$ is a PD-thickening complete with respect to divided powers of $\ker(\rho)$. After inverting $p$, the divided power structure gets killed, which shows that $\Adr\widehat\otimes\Q_p$ identifies with Fontaine's ring $\Bdr^+$ as defined in (\cite{F2}, \S 1.5).
\item The proof also shows that the ring $\Acris=L\Omega^\bullet_{\OKB/\OK}\widehat\otimes\Z_p$ introduced in Remark \ref{acris} is the universal $p$-adically complete PD-thickening of $\OCK$ over $\OK$, where the definition of a PD-thickening is the same as in \ref{pdthick} except that we do not require $\ker(\rho)^{[i+1]}=0$ (compare \cite{F2}, \S 2.2).
\end{enumerate}
\end{remas}

\subsection{The Fontaine element} Recall that we defined the $G_K$-module $\Z_p(1)$ as the inverse limit $\limproj \mu_{p^r}$; it can be also viewed as the Tate module of the torsion $\Z_p$-module $\mu_{p^\infty}$. Our aim is now to construct a canonical $G_K$-equivariant map $$\Z_p(1)\to {\rm Fil}^1\Bdr^+.$$ The image of a generator of $\Z_p(1)$ will be an analogue of the complex period $2\pi i$ called the Fontaine element.

\begin{cons}\label{consfonel}\rm
Represent an element  of $\Zp(1)$ by a sequence $(\varepsilon_n)$ of $p$-power roots of unity with $\varepsilon_0=1$ and $\varepsilon_{n+1}^p=\varepsilon_n$. Reducing the sequence modulo $p$ we obtain an element $\varepsilon$ in $(\OKB/(p))^{\rm perf}$ with multiplicative representative $[\varepsilon]$ in the Witt ring $W(\OKB/(p))^{\rm perf})=A_{\rm inf}$. As the canonical surjection $\theta:\, A_{\rm inf}\twoheadrightarrow \OCK$ lifts the projection $(\OKB/(p))^{\rm perf}\to\OCK/(p)$ sending the mod $p$ class of $\varepsilon$ to 1, we have $\theta([\varepsilon])=1$. On the other hand, we may view $[\varepsilon]$ as an element of $\Adr\widehat\otimes \Z_p\subset \Bdr^+$ via the embedding $A_{\rm inf}\hookrightarrow\Adr\widehat\otimes \Z_p$ given by Corollary \ref{corperfect2}. It follows that the assignment
\begin{equation}\label{fonel}
(\varepsilon_n)\mapsto  \sum_{n=1}^\infty(-1)^{n+1}\frac{([\varepsilon]-1)^n}{n}
\end{equation}
gives a well-defined map  $$\iota:\,\Z_p(1)\to {\rm Fil}^1\Bdr^+$$ as $[\varepsilon]-1$ lies in the maximal ideal ${\rm Fil}^1\Bdr^+=\ker(\theta\otimes\Q)$ of the complete discrete valuation ring $\Bdr^+$. We may also view the right hand side of (\ref{fonel}) as the $p$-adic logarithm of $[\varepsilon]$ in the $p$-adically complete ring $\Bdr^+$. This shows that $\iota$ is $G_K$-equivariant (for $K=\Q_p$): given $g\in G_K$, we have by definition $g\varepsilon=\varepsilon^{\chi(g)}$ where $\chi$ is the cyclotomic character, whence $g[\varepsilon]=[\varepsilon]^{\chi(g)}$ by the multiplicativity of the lifting $[\varepsilon]\mapsto \varepsilon$. Taking the logarithm finally gives $g\iota((\varepsilon_n))=\chi(g)\iota((\varepsilon_n)$ as desired.\end{cons}

\begin{defi}\rm
We define the Fontaine element $t\in {\rm Fil}^1\Bdr^+$ as the image of a fixed generator of $\Z_p(1)$ under $\iota$.\end{defi}

Thus the Fontaine element depends on the choice of a generator up to multiplication by an element in $\Zp^\times$ and the Galois group acts on it via the cyclotomic character. This is the promised analogue of the complex period $2\pi i$.

In Subsection \ref{secproof} we shall see that the Fontaine element may also be defined by means of the $p$-adic comparison isomorphism. Here is a first step towards this claim. By passing to the quotient by $\Fil^2\Bdr^+$ the map $\iota$ induces a $G_K$-equivariant embedding
$$
\overline{\iota}\colon \Z_p(1)\hookrightarrow{\rm gr}^1\Bdr^+.$$ Another such map is constructed as follows. Recall from Subsection \ref{secfontaine} that the dlog map $\mu_{p^\infty}\to \Omega^1_{\OKB/\OK}$ induces an embedding $\Z_p(1)=T_p(\mu_{p^\infty})\to T_p(\Omega^1_{\OKB/\OK})$ and the latter is identified with ${\rm gr}^1_FL\widehat\Omega^\bullet_{\OKB/\OK}\widehat\otimes\Z_p\subset {\rm gr}^1\Bdr^+$ by Corollary \ref{coruniv}. So we have another Galois-equivariant embedding
\begin{equation}\label{dlog}
\Z_p(1)\to {\rm gr}^1\Bdr^+.
\end{equation}

\begin{prop}\label{fontaineelement} The map (\ref{dlog}) coincides with $\overline\iota$.

\end{prop}
\begin{proof} Note first that since $[\varepsilon]-1\in {\rm Fil^1}\Bdr^+$, the element $\overline\iota((\varepsilon_n))$ is the class of $[\varepsilon]-1$ in ${\rm gr}^1\Bdr^+.$ We have to show that this class corresponds to that of ${\rm dlog}(\varepsilon_n)\in T_p(\mu_{p^\infty})$ under the identification of  Corollary \ref{coruniv}.  We do this modulo $p^n$.

The map $\Ainf\to H_0(L\widehat\Omega^\bullet_{\OKB/(p^n)/\OK/(p^n)}/F^2)$  coming from Corollary \ref{corperfect} takes the multiplicative representative $[\varepsilon]$ to the $p$-adic limit of the elements $\widetilde{\varepsilon_{n+m}}^{p^{n+m}}$, where $\widetilde{\varepsilon_{n+m}}$ is an arbitrary lift of the class of  $\varepsilon_{n+m}$ under the surjection $H_0(L\widehat\Omega^\bullet_{(\OCK/p^n)/(\OK/p^n)}/F^2)\twoheadrightarrow \OCK/(p^n)$.  Let $P_\bullet$ be the standard free simplicial resolution of the $\OK$-algebra $\OKB$ and for an element $b\in P_i$ put $x_b\in P_{i+1}=\OK[P_i]$ for the corresponding variable for all $i\geq -1$ (with the convention $P_{-1}:=\OKB$). By definition, $L\widehat\Omega^\bullet_{\OKB/(p^n)/\OK/(p^n)}/F^2\simeq (L\widehat\Omega^\bullet_{\OKB/\OK}/F^2)\otimes^L[\Z\stackrel{p^n}\to \Z]$ is
quasi-isomorphic to the total complex of the double complex
\begin{equation}\label{doublecomplextens}
\xymatrix{
\cdots\ar[r]& P_1\oplus \Omega^1_{P_2/\OK}\ar[r] & P_0\oplus \Omega^1_{P_1/\OK}\ar[r] & \Omega^1_{P_0/\OK}\ar[r] & 0\\
\cdots\ar[r]& P_1\oplus \Omega^1_{P_2/\OK}\ar[r]\ar[u]^{p^n} & P_0\oplus \Omega^1_{P_1/\OK}\ar[r]\ar[u]^{p^n} & \Omega^1_{P_0/\OK}\ar[r]\ar[u]^{p^n} & 0.
}
\end{equation}
As each term in \eqref{doublecomplextens} is torsion free, the total complex is quasi-isomorphic to
\begin{equation}\label{singlecomplex}
\cdots\to P_1/(p^n)\oplus \Omega^1_{P_2/\OK}/p^n\stackrel{d_1}\to P_0/(p^n)\oplus \Omega^1_{P_1/\OK}/p^n\stackrel{d_0}\to \Omega^1_{P_0/\OK}/(p^n)\to 0,
\end{equation}
a complex placed in homological degrees $\geq -1$. By Proposition \ref{univ1orderOCKn} (2) we have an exact sequence
$$
0\to {}_{p^n}\Omega^1_{\OKB/\OK}\to H_0(L\Omega^\bullet_{(\OKB/(p^n))/(\OK/(p^n))}/F^2)\to \OKB/(p^n)\to 0,
$$
and the proof of the proposition shows that the term ${}_{p^n}\Omega^1_{\OKB/\OK}$ comes from setting the summands $P_i/(p^n)$ in the terms of (\ref{doublecomplextens}) to 0. Now the element $(x_{\varepsilon_{2n+m}}^{p^n},0)\in P_0\oplus \Omega^1_{P_1/\OK}$ has image $dx_{\varepsilon_{2n+m}}^{p^n}=p^nx_{\varepsilon_{2n+m}}^{p^n-1}dx_{\varepsilon_{2n+m}}$ in $\Omega^1_{P_0/\OK}$, hence its mod $p^n$ class lies in the kernel of the map $d_0$ of (\ref{doublecomplextens}). Therefore $(x_{\varepsilon_{2n+m}}^{p^n},0)$ defines a class $\widetilde{\varepsilon_{n+m}}$ in $H_0(L\widehat\Omega^\bullet_{(\OCK/p^n)/(\OK/p^n)}/F^2)$ which lifts that of $\varepsilon_{n+m}$ in $\OCK/(p^n)$ by the above description of the map $H_0(L\Omega^\bullet_{(\OKB/(p^n))/(\OK/(p^n))}/F^2)\to \OKB/(p^n)$. So we have to compute the class of $(x_{\varepsilon_{2n+m}}^{p^{2n+m}}-1,0)$ in $(L\widehat\Omega^\bullet_{\OKB/\OK}/F^2)\otimes^L \mathbb{Z}/p^n\Z$. The mod $p^n$ class of  $(x_{\varepsilon_{2n+m}}^{p^{2n+m}}-1,0)$ in \eqref{singlecomplex} is again in the kernel of $d_0$, and under the quasi-isomorphism of the total complex of \eqref{doublecomplextens} with \eqref{singlecomplex} it corresponds to the class of $(x_{\varepsilon_{2n+m}}^{p^{2n+m}}-1,0, p^{n+m}x_{\varepsilon_{2n+m}}^{p^{2n+m}-1}dx_{\varepsilon_{2n+m}})\in P_0\oplus \Omega^1_{P_1/\OK}\oplus \Omega^1_{P_0/\OK} $ in the total complex of \eqref{doublecomplextens}. As $x_{\varepsilon_{2n+m}}^{p^{2n+m}}-1$ maps to 0 under the map $P_0\to \OKB$, this class represents an element of ${}_{p^n}\Omega^1_{\OKB/\OK}$, and therefore comes from the double complex obtained by deleting the $P_i$ terms. Chasing through the construction shows that the element of ${}_{p^n}\Omega^1_{\OKB/\OK}$ thus obtained is the image of $p^{n+m}x_{\varepsilon_{2n+m}}^{p^{2n+m}-1}dx_{\varepsilon_{2n+m}}$, which is $p^{n+m}\varepsilon_{2n+m}^{p^{2n+m}-1}d\varepsilon_{2n+m}=p^{n+m}d\log \varepsilon_{2n+m}=d\log \varepsilon_n$.
\end{proof}

\begin{cor}
The Fontaine element $t$ generates the maximal ideal of the discrete valuation ring $\Bdr^+$, and hence $\Bdr=\Bdr^+[t^{-1}]$.\end{cor}

\begin{rema}\rm
It can be checked that the same construction as in \ref{consfonel} realizes $t$ as an element of $A_{\rm cris}$. One defines $B_{\rm cris}=A_{\rm cris}[t^{-1}]$. The Frobenius endomorphism of $\OCK/(p)$ lifts uniquely to an endomorphism $\varphi$ of $A_{\rm cris}$, and satisfies $\varphi(t)=pt$ (see \cite{F2}, 2.3.4).
\end{rema}

\section{Beilinson's comparison map}

\subsection{Sheaf-theoretic preliminaries}
In this section we construct the map comparing de Rham and $p$-adic \'etale cohomology following Beilinson's approach.
We begin by assembling general facts about sheaves in certain Grothendieck topologies needed for the construction. First a general comparison result for Grothendieck topologies due to Verdier:

\begin{theo}\label{verdier}
Assume $F: C\to C'$ is a functor between small categories, and $C'$ is equipped with a Grothendieck topology. Equip $C$ with the induced Grothendieck topology, i.e. the finest topology in which the pullback of a sheaf on $C'$ by $F$ is a sheaf on $C$.

If $F$ is fully faithful, and moreover every object of $C'$ has a covering by objects in the image of $F$, then the pullback functor induces an equivalence of category of sheaves on $C'$ with the category of sheaves on $C$.
\end{theo}

See \cite{sga4}, Expos\'e III, Theorem 4.1 or \cite{tamme}, Chapter I, \S 3.9. The main point of the proof is that under the conditions of the theorem one may construct a push-forward functor from sheaves on $C$ to sheaves on $C'$ which is right adjoint to pullback.

Beilinson needs a refinement of the above theorem for functors that are faithful but not necessarily fully faithful. He therefore replaces the covering condition in Verdier's theorem by the following more complicated one:\smallskip

\noindent {\em Condition (*)}. For every $V\in C'$ and a finite family of pairs $(W_\alpha, f_\alpha)$ with $W_\alpha\in C$ and $f_\alpha: V\to F(W_\alpha)$ morphisms in $C'$ there exists a set of objects $W_\beta\in C$ together with morphisms $F(W_\beta)\to V$ in $C'$ satisfying:
\begin{itemize}
\item The morphisms $F(W_\beta)\to V$  form a covering family of $V$.
\item Every composite morphism $F(W_\beta)\to V\to F(W_\alpha)$ is in the image of a morphism $W_\beta\to W_\alpha$ via $F$.
\end{itemize}    \smallskip

Under this condition Beilinson proves in \S 2.2 of \cite{Bei}:

\begin{theo}\label{beilh} If $C, C'$ are as in the previous theorem and $F:\, C\to C'$ is a faithful functor satisfying condition (*), then  the pullback functor induces an equivalence of the category of sheaves on $C'$ with the category of sheaves on $C$ for the topology induced by $F$.
\end{theo}

In the case where the initial  $B_\alpha$ form an empty set, Beilinson's condition reduces to Verdier's.\medskip

We now recall the notion of {\em Godement resolutions.} This is a canonical way to construct a flabby resolution of a sheaf on a site that has enough points.

\begin{cons}\rm Suppose for simplicity that $\calf$ is a sheaf for the Zariski topology on a scheme $Y$; this is the only case we need. Given a point $y\in Y$, we may consider the inclusion map $i_y:\, \Spec k(y)\to Y$ and the push-forward sheaf $i_{y*}\calf_y$, where $\calf_y$ is the stalk of $\calf$ at $y$ considered as a constant sheaf. The rule
$$
U\mapsto C^0(\calf)(U):=\prod_{y\in U}i_{y*}\calf_y
$$
with the obvious restriction maps defines a flabby sheaf $C^0(\calf)$ on $X$ and there is a natural injective morphism of sheaves
$$
\iota:\,\calf\to C^0(\calf),\quad s\mapsto (s_y)
$$
where $s_y$ is the image of a section $s$ in the stalk $\calf_y$. Now we define
$$
C^1(\calf):=C^0(\cok(\iota)).
$$
By construction, there is a natural map
$$
d^0:\, C^0(\calf)\to C^1(\calf).
$$
We now construct inductively sheaves
$$
C^i(\calf):=C^0(\cok(d^{i-2})).
$$
and maps
$$
d^{i-1}:\, C^{i-1}(\calf)\to C^i(\calf).
$$
In this way we obtain a flabby resolution $\calf\to C^\bullet(\calf)$ canonically attached to $\calf$ and functorial in $\calf$, the Godement resolution. By definition, the complex $\Gamma(X, C^\bullet(\calf))$ represents $R\Gamma(X,\calf)$ in the bounded below derived category of abelian groups. This construction extends in the usual way from sheaves to bounded below complexes of sheaves.
\end{cons}

Finally, we need the definition of Voevodsky's $h$-topology.

\begin{defi}\rm
A morphism $\phi:\, X\to Y$ of schemes is a {\em topological epimorphism} if on underlying spaces it induces a topological quotient map (i.e. $\phi$ is surjective and the topology of $Y$ is the same as the quotient topology induced by $\phi$.) It is a {\em universal topological epimorphism} if for every morphism $Z\to Y$ the base change map $X\times_YZ\to Z$ is a topological epimorphism.

An {\em $h$-covering} of a scheme $X$ is a finite family of morphisms of finite type $X_i\to X$ such that
$\amalg X_i\to X$ is a universal topological epimorphism. We equip the category of schemes with the induced Grothendieck topology and call it the {\em $h$-topology.}

\end{defi}

\'Etale surjective maps and proper surjective maps are universal topological epimorphisms, so the $h$-topology is finer than the \'etale or proper topologies (defined respectively by finite surjective families of \'etale and proper maps). The following geometric fact is nontrivial, however.

\begin{fact}\rm Assume $X$ is a reduced connected Noetherian excellent scheme (for instance a reduced scheme of finite type over a field or a discrete valuation ring of characteristic 0). Every $h$-covering $\amalg X_i\to X$ has a refinement $\amalg Y_j\to X$ that factors as $\amalg Y_j\to  Y\to X$, where $\amalg Y_j\to Y$ is a  Zariski open covering and $Y\to X$ is proper and surjective (but usually not flat). See \cite{sv}, Corollary 10.4.
\end{fact}

This fact has an important consequence for $h$-hypercoverings, i.e. hypercoverings for the class of coverings in the $h$-topology. Namely, we may apply Theorem \ref{descent} of the Appendix to obtain:

\begin{cor}\label{cordescent} ${}$
\begin{enumerate}
\item In the category of reduced connected excellent schemes $h$-hypercoverings satisfy cohomological descent for torsion \'etale sheaves.

\item In the category of reduced $\C$-schemes of finite type  $h$-hypercoverings satisfy cohomological descent for the complex topology.
\end{enumerate}
\end{cor}

Moreover, applying Theorem \ref{limitdescent} of the Appendix we obtain:

\begin{cor}\label{cordescent2}
Let $X$ be as above, and let $A$ be a torsion abelian group. Denoting by $A_{\mbox{\scriptsize \rm\'et}}$ and $A_h$ the associated constant \'etale and $h$-sheaves on $X$, we have a canonical quasi-isomorphism
$$
R\Gamma(X_{\mbox{\scriptsize \rm\'et}},A_{\mbox{\scriptsize \rm\'et}})\cong R\Gamma(X_{h},A_h).
$$
\end{cor}

\begin{dem} For all $i>0$ we have a series of canonical isomorphisms
\begin{equation}\label{eqcordescent}
H^i_h(X,A_h)\cong \limind H^i(C(A_{\mbox{\scriptsize \rm\'et}}(Y_\bullet)))\cong \limind H^i_{\mbox{\scriptsize \rm\'et}}(Y_\bullet, A_{\mbox{\scriptsize \rm\'et}}^\bullet)
\end{equation}
where the first isomorphism comes from applying Theorem \ref{limitdescent} of the Appendix to the system of $h$-hypercoverings $Y_\bullet\to X$, and the second is proven by the same argument as the analogous fact for \v Cech cohomology (see e.g. \cite{tamme}, Chapter I, Theorem 2.2.3). Here $A_{\mbox{\scriptsize \rm\'et}}^\bullet$ denotes the constant simplicial sheaf coming from $A_{\mbox{\scriptsize \rm\'et}}$. Since $A_{\mbox{\scriptsize \rm\'et}}^\bullet$ is the pullback of the constant sheaf $A_{\mbox{\scriptsize \rm\'et}}$ on $X$ to $Y_\bullet$, the direct system on the right hand side of (\ref{eqcordescent}) is constant by Corollary \ref{cordescent} (1).
\end{dem}

\subsection{Preliminaries on logarithmic structures}\label{logsection}

We now give a summary of the notions from logarithmic geometry we shall use; they will be needed from Section \ref{arithsec} onwards. Our basic reference for log structures is Kato's paper \cite{kato}. A gentle introduction is contained in sections 2-3 of \cite{abr}; a textbook by A. Ogus is expected.

A {\em monoid} is a commutative semigroup with unit. Every monoid $M$ has a group completion $M^{\rm gp}$ which is the universal object for monoid morphisms of $M$ into groups. It can be constructed as the quotient of $M\times M$ where two pairs $(x,y)$ and $(z,t)$ are identified if $axt= ayz$ for some $a\in M$. There is a natural map $M\to M^{\rm gp}$ induced by $x\mapsto (x,1)$; if it is injective, then $M$ is called an {\em integral} monoid.

A {\em pre-logarithmic ring} (or pre-log ring for short) is a triple $(A, M, \alpha)$, where $A$ is a commutative ring with unit, $M$ is a monoid and $\alpha:\, M\to A$ is a homomorphism in the multiplicative monoid of $A$. Morphisms $(A,M,\alpha)\to (B,N,\beta)$ of pre-log rings are given by pairs of morphisms $M\to N$ and $A\to B$ compatible via $\alpha$ and $\beta$. A pre-log ring is a {\em log ring} if the map $\alpha$ induces an isomorphism $\alpha^{-1}(A^\times)\stackrel\sim\to A^\times$, where $A^\times$ is the group of units in $A$. Every pre-log ring as above has an associated log ring given by a morphism $\alpha^a:\, M^a\to A$, where $M^a$ is the quotient of $A^\times \times M$ where two pairs $(a,x)$ and $(b,y)$ are identified if there exist $c,d\in A^\times$ such that $a\alpha(d)=b\alpha(c)$ and $cx=dy$. One checks that this is indeed a log ring, with $\alpha^a$ induced by the map $(a,x)\mapsto a\alpha(x)$.

Given a scheme $X$, one can define pre-log structures and log structures on $X$ by sheafifying the above notions for the \'etale topology. Thus a pre-log scheme is a triple $(X, M,\alpha)$, where $X$ is a scheme,  $M$ is a sheaf of monoids on the small \'etale site of $X$ and  $\alpha:\, M\to \calo_X$ is a morphism of \'etale monoid sheaves, where $\calo_X$ is considered as a monoid for its multiplicative structure.  A morphism $(Y,N, \beta)\to (X,M, \alpha)$ of pre-log schemes is given by a morphism $\phi:\, Y\to X$ of schemes and a morphism $\phi^{-1}M\to N$ of \'etale monoid sheaves whose composite with $\beta$ equals the composite $\phi^{-1}M\to\phi^{-1}\calo_X\to \calo_Y$, where the first map is the pullback of $\alpha$ and the second is induced by $\phi$. A pre-log scheme is a {\em log scheme} if moreover $\alpha$ induces an isomorphism $\alpha^{-1}(\calo_X^\times)\stackrel\sim\to\calo_X^\times$.  From now on we shall drop $\alpha$ from the notation when considering (pre-)log schemes.

One may define  an associated log scheme $(X, M^a)$ for every pre-log scheme $(X,M)$ by sheafifying the construction for pre-log rings described above.
A log scheme $(X,M)$ is {\em coherent} (resp. {\em integral\/}) if \'etale locally there exists a morphism $P_X\to \calo_X$ whose associated log structure is isomorphic to $M$, where $P_X$ is a constant sheaf of monoids defined by a finitely generated (resp. integral) monoid $P$. The log structure is {\em fine} if it is coherent and integral.
We shall only use the two most important examples of log structures: the {\em trivial log structure}, given by $M=\calo_X^\times$ and the natural inclusion $\alpha: \calo_X^\times\to\calo_X$, and the canonical log structure associated with a pair $(X,D)$, where $X$ is a regular scheme and $D\subset X$ a normal crossing divisor (see the beginning of the next section for a reminder). In the latter case the map $M\to\calo_X$ is given by the inclusion of $(\calo_X\cap j_*\calo_U^\times)\to \calo_X$, where $j:\, U\to X$ is the inclusion map of the open complement $U$ of $D$ in $X$.

Given a morphism $(A,M,\alpha)\to (B,N,\beta)$ of pre-log rings, one defines the $B$-module $\Omega^1_{(B,N)/(A,M)}$ of log differentials as the quotient of $\Omega^1_{B/A}\oplus (B\otimes_\Z\cok(M^{\rm gp}\to N^{\rm gp}))$ by the submodule generated by elements of the form $(d\beta(n),0)-(0,\beta(n)\otimes n)$ for $n\in N$. It comes equipped with natural maps $d: B\to \Omega^1_{(B,N)/ (A,M)}$, ${\rm dlog}: N\to \Omega^1_{(B,N)/ (A,M)}$ related by the formula $\beta(n){\rm dlog}(n)=d\beta(n)$ for all $n\in N$. One can show that the operation of taking associated log rings induces an isomorphism on log differentials.

Given a morphism $\phi:\,(Y,N)\to (X,M)$ of pre-log schemes, one defines the $\calo_Y$-module $\Omega^1_{(Y,N)/ (X,M)}$ of log differentials by performing the above construction in the context of \'etale monoid sheaves. If moreover $\phi$ is a morphism of fine log schemes, one says that $\phi$ is {\em log smooth} if the underlying scheme morphism is locally of finite presentation and $f$ satisfies a log analogue of the infinitesimal lifting property (\cite{kato}, 3.3 or \cite{abr}, Definition 3.10). In this case one can show that the sheaf $\Omega^1_{(Y,N)/ (X,M)}$ is locally free of finite rank. The fundamental example of a log smooth morphism is given by a regular flat scheme $X\to S$, where $S$ is the spectrum of a discrete valuation ring and $X$ has semi-stable reduction over $S$. Here the log structures on $X$ and $S$ are the canonical ones associated with the special fibre of $X\to S$ and the closed point of $S$, respectively.

Once log differentials have been defined, one has a notion of a log cotangent complex. Olsson's paper \cite{olsson}, which is the main reference on the subject, contains two constructions; Beilinson works with that of Gabber, explained in Section 8 of \cite{olsson}. The main point is that Gabber defines an analogue of the free $A$-algebra $A[B]$ used in the construction of the usual cotangent complex $L_{B/A}$ for a morphism $(A,M)\to (B, N)$ of pre-log rings. This is the pre-log ring given by the free $A$-algebra $A[B\amalg N]$ on the disjoint union on the underlying sets of $B$ and $N$ together with the morphism of monoids $M\oplus {\Bbb N}^N\to A[B\amalg N]$ induced by the structure map $M\to A$ and the map sending the basis element of the free monoid ${\Bbb N}^N$ corresponding to $n\in N$ to the generator of $A[B\amalg N]$ given by $n$. There is a natural morphism $(A,M)\to (A[B\amalg N],M\oplus {\Bbb N}^N)$ of pre-log rings and the associated module of log differentials is isomorphic to the free $B$-module with basis $B\amalg N$.

With the notion of a free algebra attached to $(A,M)\to (B, N)$ at hand, one defines a canonical free resolution ${P_{(A,M)}(B,N)}_\bullet\to (B, N)$ in this context by mimicking Construction \ref{standardres} (resolution meaning here that the underlying morphism of simplicial sets is a trivial fibration). One then defines
$$
L_{(B,N)/(A,M)}:={\Omega^1}_{{P_{(A,M)}(B,N)}_\bullet/(A,M)}\otimes_{{P_{(A,M)}(B,N)}_\bullet}B
$$
where the tensor product is taken over the underlying simplicial ring of ${P_{(A,M)}(B,N)}_\bullet$.

Similarly, one defines the log de Rham algebra
$$
L\Omega^\bullet_{(B,N)/(A,M)}:={\rm Tot}({\Omega^\bullet}_{{P_{(A,M)}(B,N)}_\bullet/(A,M)})
$$
which has a Hodge filtration and a Hodge-completed version  $L\widehat\Omega^\bullet_{(B,N)/(A,M)}$. The graded pieces of the Hodge filtration are given by shifts of derived exterior powers of $L_{(B,N)/(A,M)}$ as in Proposition \ref{gridr}. All these notions reduce to the usual ones in case the pre-log structures are trivial (i.e. given by the unit submonoids).

It is proven in Olsson's paper (\cite{olsson}, Theorem 8.20) that passing to associated log rings for a morphism $(A,M)\to (B, N)$ of pre-log rings induces an isomorphism on associated log cotangent complexes. Furthermore, the logarithmic cotangent complex enjoys properties analogous to those of the usual one. In particular, one has a natural map $L_{(B,N)/(A,M)}\to \Omega^1_{(B,N)/(A,M)}$ which induces an isomorphism on $H_0$ and is a quasi-isomorphism if $(B,N)$ is a free algebra over $(A,M)$ defined by Gabber's construction (\cite{olsson}, Lemmas 8.9 and 8.10). Most importantly, there is a log analogue of the exact transitivity triangle (Theorem \ref{cotangexacttriang}) for a sequence $(A,M)\to (B,N)\to (C, L)$ of morphisms of pre-log rings (\cite{olsson}, Theorem 8.18).

Finally, given a morphism $\phi:\,(Y,N)\to (X,M)$ of pre-log schemes, one defines sheafified variants $L_{(Y,N)/(X,M)}$, $L\Omega^\bullet_{(Y,N)/(X,M)}$ and $L\widehat\Omega^\bullet_{(Y,N)/(X,M)}$ of the above constructions by performing analogous operations starting from the natural morphism $(\phi^{-1}\calo_X, \phi^{-1}M)\to (\calo_Y, N)$ of pairs of \'etale sheaves induced by $\phi$. In the case when $\phi$ is a morphism of fine log schemes that is log smooth and integral (the latter is a technical condition satisfied in our basic example of semistable reduction), the natural morphism $L_{(Y,N)/(X,M)}\to \Omega^1_{(Y,N)/(X,M)}$ is a quasi-isomorphism (see \cite{olsson}, 3.7 and 8.34).

\subsection{The geometric side of the comparison map}\label{geomside}

In what follows, by `variety' we mean a separated scheme of finite type over a field.

Let $Y$ be a smooth variety over a field $k$ of characteristic 0.
Choose a smooth compactification $j:\, Y\hookrightarrow \overline Y$ such that $D:=\overline Y\setminus Y$ is a divisor with normal crossings. Such a compactification exists by Hironaka's theorem.

Recall the notion of divisor with normal crossings: there is a family of \'etale morphisms $\phi_i:\,\overline Y_i\to \overline Y$ with $\bigcup \im(\phi_i)=\overline Y$ (i.e. a covering of $\overline Y$ in the \'etale topology) such that each $Y_i$ sits in a cartesian square
$$
\begin{CD}
D\times_{\overline Y}\overline Y_i @>>> \overline Y_i\\
@VVV @VV{\rho_i}V \\
V(t_1\cdots t_r) @>>> \A^n_k
\end{CD}
$$
where the morphism $\rho_i$ is \'etale, the $t_1,\dots, t_n$ are coordinate functions on  $\A^n_k$ and $r\leq n$.

In the above situation, we have the notion of the logarithmic de Rham complex.

\begin{defi}\rm
Given a pair $(\overline Y, D)$ as above, the logarithmic de Rham complex $ \Omega^{\bullet}_{\overline Y/k}(\log D)$ is the subcomplex of  $j_*\Omega^{\bullet}_{ Y/k}$ whose terms have local sections $\omega\in j_*\Omega^{i}_{Y/k}(U)$ such that both $\omega$ and $d\omega$ have a simple pole along $D$ (i.e. $f\omega$ is a section of $\Omega^{i}_{\overline  Y/k}$ and $fd\omega$ is a section of $\Omega^{i+1}_{\overline  Y/k}$ for a local equation $f$ of $D$ in $U$ sufficiently small).
\end{defi}

\begin{rema}\rm It can be shown (see \cite{deligneeqdiff}, \S 3 and \cite{Hodge2}, \S 3.1) that $ \Omega^{i}_{\overline Y/k}(\log D)=\wedge^i\Omega^1_{\overline Y/k}(\log D)$ and if we pull back $ \Omega^{1}_{\overline Y/k}(\log D)$ to an \'etale neighbourhood $\overline Y_i$ as above, it becomes freely generated by $dt_1/t_1,\dots, dt_r/t_r, dt_{r+1},\dots, dt_n$.
\end{rema}

Now set
$$
R\Gamma_{\rm dR}(Y/k):=R\Gamma(\overline Y, \Omega^{\bullet}_{\overline Y/k}(\log D)),
$$

Notice that $R\Gamma_{\rm dR}(Y/k)$ is an object in the bounded derived category of abelian groups, and its cohomologies are the groups
$$
H^i_{\rm dR}(Y/k)=\hyp^i(\overline Y, \Omega^{\bullet}_{\overline Y/k}(\log D)).
$$
These groups are equipped with the Hodge filtration defined by
$$
F^pH^i_{\rm dR}(Y/k):=\im(\hyp^i(\overline Y, F^p\Omega^{\bullet}_{\overline Y/k}(\log D))\to \hyp^i(\overline Y, \Omega^{\bullet}_{\overline Y/k}(\log D)))
$$
where
$$
F^p\Omega^{\bullet}_{\overline Y/k}(\log D)=(0\to\Omega^{p}_{\overline Y/k}(\log D)\to \Omega^{p+1}_{\overline Y/k}(\log D)\to \cdots)
$$
In case $Y$ is proper, we may take $\overline Y=Y$ and hence
$$
H^i_{\rm dR}(Y/k)=\hyp^i(Y, \Omega^{\bullet}_{ Y/k}).
$$

\begin{remas}\rm ${}$
\begin{enumerate}
\item Deligne has shown (see \cite{Hodge2}, Theorem 3.2.5) that $R\Gamma_{\rm dR}(Y/k)$ does not depend on the choice of $\overline Y$.
\item When $k=\C$  and $Y$ is not necessarily proper but smooth, we have isomorphisms
$$
\hyp^i(Y, \Omega^{\bullet}_{ Y/k})\stackrel\sim\to H^i(Y^{\rm an}, \C)\stackrel\sim\leftarrow \hyp^i(\overline Y, \Omega^{\bullet}_{\overline Y/k}(\log D)).
$$
However, the filtration on $H^i(Y^{\rm an}, \C)$ coming from the filtration on the left hand side induced by
$$
F^p\Omega^{\bullet}_{ Y/k}=(0\to\Omega^{p}_{ Y/k}\to \Omega^{p+1}_{ Y/k}\to \cdots)
$$
is {\em not\/} the same as the Hodge filtration coming from the right hand side (in fact, it is often trivial).
\end{enumerate}
\end{remas}

We now discuss a sheafified variant of the above notions.
Fix an algebraically closed base field $k$ of characteristic 0, and consider pairs $(U, \overline U)$, where $U$ is a smooth $k$-variety with smooth compactification $\overline U$ such that $\overline U\setminus U$ is a divisor $D$ with normal crossings. These form a category ${\mathcal P}_k$ whose morphisms $(U, \overline U)\to (U', \overline U')$  are defined as morphisms $\overline U\to \overline U'$ mapping $U$ into $U'$. We have a contravariant functor on this category given by
\begin{equation}\label{prepreadr}
(U, \overline U)\mapsto\Omega^\bullet_{(U, \overline U)/k}:=\Omega^\bullet_{\overline U/k}(\log D).
\end{equation}
We consider this functor as a presheaf on ${\mathcal P}_k$ and would like to sheafify its total derived functor but the latter takes values in a derived category. However, following Illusie \cite{illsurvey}, we may use Godement resolutions to find a canonical complex representing it:
\begin{equation}\label{preadr}
(U, \overline U)\mapsto \Gamma(\overline U,C^\bullet(\Omega^\bullet_{(U, \overline U)/k}))
\end{equation}
We thus have a presheaf on the category ${\mathcal P}_k$ that is a derived version of (\ref{prepreadr}). Via the forgetful functor ${\mathcal P}_k\to {\rm Var}_k$ given by $(U, \overline U)\mapsto U$ we may restrict Voevodsky's $h$-topology on ${\rm Var}_k$ to ${\mathcal P}_k$, and therefore the following definition makes sense.

\begin{defi}\rm We define
${\mathcal A}_{\rm dR}$ to be the complex of $h$-sheaves associated with the presheaf (\ref{preadr}) on the category ${\mathcal P}_k$.
\end{defi}

Note that ${\mathcal A}_{\rm dR}$ carries a Hodge filtration induced from the one on $\Omega^\bullet_{(U, \overline U)/k}$.

The good news is that ${\mathcal A}_{\rm dR}$ defines a filtered complex of $h$-sheaves on the whole of ${\rm Var}_k$, by virtue of the following theorem.

\begin{theo}\label{hequiv}
The forgetful functor ${\mathcal P}_k\to{\rm Var}_k$ induces an equivalence of categories between $h$-sheaves on ${\mathcal P}_k$ and $h$-sheaves on ${\rm Var}_k$.
\end{theo}

\begin{proof} Apply Beilinson's Theorem \ref{beilh} in the situation where $C'$ is ${\rm Var}_k$ equipped with the $h$-topology, and $F$ is the (faithful) forgetful functor $(V, \overline V)\to V$ from the category ${\mathcal P}_k'$ of pairs $(V, \overline V)$ consisting of a $k$-variety $V$ and a proper $k$-variety $\overline V$ containing $V$ as a dense open subset. Notice that condition (*) is satisfied: given a finite family of maps $V\to V_\alpha$ with $V_\alpha$ having a compactification $\overline V_\alpha$, embed $V$ in a proper $k$-variety $V'$ (such a $V'$ exists by Nagata's theorem) and let $\overline V$ be the closure of the image of the embedding $V\to \overline V'\times\prod \overline V_\alpha$. Then $\overline V$ is proper and the second projection
induces maps $(V, \overline V)\to (V_\alpha, \overline V_\alpha)$.

Next notice that the inclusion functor ${\mathcal P}_k\to {\mathcal P}_k'$ is fully faithful, hence we may apply Verdier's Theorem \ref{verdier} to it: we have to check that each pair $(V, \overline V)$ in  ${\mathcal P}_k'$  has an $h$-covering $(U, \overline U)\to (V, \overline V)$ by a pair in  ${\mathcal P}_k$. This follows from Hironaka's theorem or de Jong's alteration theorem over fields \cite{dejong}.
\end{proof}

We now come to the main result of this subsection. The morphism of filtered complexes of presheaves
$$
C^\bullet(\Omega^\bullet_{(U, \overline U)/k})\to {\mathcal A}_{\rm dR}
$$
gives rise to a morphism
\begin{equation}\label{drmap}
R\Gamma_{\rm dR}(Y/k)\to R\Gamma_h(Y,{\mathcal A}_{\rm dR})
\end{equation}
for a smooth $k$-variety $Y$, inducing maps
$$H^n_{\rm dR}(Y/k)\to \hyp^n(Y, {\mathcal A}_{\rm dR}).$$

\begin{theo}\label{adr} For a smooth variety $Y$ over $k$ the maps (\ref{drmap}) are filtered quasi-isomorphisms.
\end{theo}

\begin{proof} By a `Lefschetz principle' type argument we reduce to the case $k=\C$.

Choose a  smooth normal crossing compactification $\overline Y$ for $Y$ with complement $D$. By (\cite{Hodge3}, (6.2.8)) or in more detail  (\cite{conraddescent}, Theorem 4.7) there exists an $h$-hypercovering $V_\bullet\to Y$ such that each $V_n$ is a smooth $k$-scheme of finite type and furthermore there is a simplicial compactification $V_\bullet\hookrightarrow \overline V_\bullet\to \overline Y$ such that  $\overline V_n$ is proper and smooth  with $D_n:= \overline V_n\setminus V_n$ a normal crossing divisor. On $V_\bullet$ consider the simplicial  complex of presheaves $\Omega^{\bullet}_{\overline V_\bullet/k}(\log D_\bullet)$.

According to Grothendieck \cite{grothdr} (see also \cite{Hodge2}), we have a filtered quasi-isomorphism
$$
R\Gamma(\overline V_\bullet, \Omega^{\bullet}_{\overline V_\bullet/\C}(\log D_\bullet))\cong R\Gamma_{\rm sing}(\overline V_\bullet,\C)
$$
where on the right hand side we have complex singular cohomology. Similarly, we have
$$
R\Gamma(\overline Y, \Omega^{\bullet}_{\overline Y/\C}(\log D)\cong R\Gamma_{\rm sing}(\overline Y,\C).
$$
The two isomorphisms induce commutative diagrams for all $n$
$$
\begin{CD}
\hyp^n(\overline V_\bullet, \Omega^{\bullet}_{\overline V_\bullet/\C}(\log D_\bullet)) @>{\cong}>> H^n_{\rm sing}(\overline V_\bullet,\C)\\
@VVV @VVV \\
H^n_{\rm dR}(Y)=\hyp^n(\overline Y, \Omega^{\bullet}_{\overline Y/\C}(\log D)) @>{\cong}>> H^n_{\rm sing}(\overline Y,\C)
\end{CD}
$$
By Corollary \ref{cordescent} the right vertical map is an isomorphism, hence so is the left one. Recall that $R\Gamma(\overline Y, \Omega^{\bullet}_{\overline Y/\C}(\log D))$ is computed (in the Zariski topology) by $\Gamma(\overline Y,C^\bullet(\Omega^\bullet_{(Y, \overline Y)/\C}))$ and similarly for the simplicial version.  It follows that the direct system  $H^n(\overline V_\bullet, C^\bullet\Omega^{\bullet}_{\overline V_\bullet/\C}(\log D_\bullet))$ for all $V_\bullet$ as above is constant. Since the direct limit of this system is $H^n(Y, {\mathcal A}_{\rm dR})$ by Theorem \ref{limitdescent}, we are done.
\end{proof}

We shall in fact need a Hodge-completed version of the above theorem. Define the Hodge-completed de Rham complex $\widehat\Omega^\bullet_{(U, \overline U)/k}$ as the projective system defined by the quotients $\widehat\Omega^\bullet_{(U, \overline U)/k}/F^i$ by steps of the Hodge filtration. Next, denote by $\widehat{\mathcal A}_{\rm dR}$ the $h$-sheaf (of projective systems of complexes) associated with
$$
(U, \overline U)\mapsto \Gamma(\overline U,C^\bullet(\widehat\Omega^\bullet_{(U, \overline U)/k})).
$$
As above, this gives rise to an $h$-sheaf on the category of varieties, whence morphisms
\begin{equation}\label{compdrmap}
R\Gamma_{\rm dR}(Y)^\wedge\to R\Gamma_h(Y,\widehat{\mathcal A}_{\rm dR})
\end{equation}
where on the left-hand side we have hypercohomology of the Hodge-completed de Rham complex. Its cohomology groups are the same as those of the non-completed complex as the Hodge filtration on each fixed group $H^n_{\rm dR}(Y)$ is finite. Hence we have:

\begin{theo}\label{coradr} For a smooth variety $Y$ over $k$ the maps (\ref{compdrmap}) are filtered quasi-isomorphisms.
\end{theo}

\subsection{The arithmetic side of the comparison map}\label{arithsec}

We now consider an arithmetic version of the previous constructions. Let $K$ be a finite extension of $\Q_p$, and $\calo_K$ its ring of integers, with residue field $\kappa$. A semistable pair over $ K$ will consist of a smooth $K$-variety $U$ and an open immersion $j:\, U\to \calu$, where $\calu$ is a regular scheme proper and flat over $\calo_K$, and $\calu\setminus U$ is a divisor $D$ with normal crossings. This divisor consists of two parts. There is a `horizontal part' $D_h$ consisting of the components that are flat (hence surjective) over $\calo_K$. It yields a normal crossing divisor $\calu_K\setminus U$ after passing to the generic fibre. The other components form the `vertical part' $D_v$; it is concentrated in the special fibre $\calu_\kappa$. Locally the situation can be described as follows. Assume that a point $u\in \calu$ lies on $r$ components of $D_v$ and $s$ components of $D_h$. Then there is an \'etale morphism $V\to \calu$ whose image contains $u$ and another \'etale morphism $V\to\Spec \calo_K[t_1,\dots, t_n]/(t_1\cdots t_r-\pi)$, where $r\leq n$ and $\pi$ is a uniformizer in $\calo_K$. Moreover, the trace $D_h\times_\calu V$ is described by the cartesian diagram:
$$
\begin{CD}
D_h\times_{\calu}V @>>>  V\\
@VVV @VV{\rho_i}V \\
V(t_{r+1}\cdots t_{r+s}) @>>> \Spec \calo_K[t_1,\dots, t_n]/(t_1\cdots t_r-\pi)
\end{CD}
$$
A semi-stable pair over $\overline K$ will be a pair $(V, \calv)$ defined by an open immersion of a $\overline K$-variety $V$ in a flat proper $\calo_{\overline K}$-scheme $\calv$  which comes by base change from a semi-stable pair $(U', \calu')$ defined over some finite extension $K'|K$ in the above sense.

Equip $\calv$ with the canonical log structure defined by $M=\calo_\calv\cap  j_*\calo_V^\times$; we denote this log scheme again by $(V, \calv)$. There is a morphism of log schemes $(V, \calv)\to\Spec\calo_K$ induced by the composite $\calv\to\Spec\calo_{\overline K}\to\Spec\calo_K$; here $\Spec\calo_K$ is equipped with the trivial log structure given by $\calo_K^\times$. It therefore makes sense to consider the derived log de Rham algebra $L\Omega^\bullet_{(V, \calv)/\calo_K} $ and its completed version $L\widehat\Omega^\bullet_{(V, \calv)/\calo_K} $ introduced in Subsection \ref{logsection}. As above, the rule
$$
(V, \calv)\mapsto \Gamma(\calv, L\widehat\Omega^\bullet_{(V, \calv)/\calo_K})
$$
defines a contravariant functor on the category of semi-stable pairs over $\overline K$. To make the derived functor $R\Gamma(\calv, L\widehat\Omega^\bullet_{(V, \calv)/\calo_K})$ a presheaf on this category, we again use the Godement resolution for the Zariski topology:
$$
(V, \calv)\mapsto \Gamma(\calv, { C}^\bullet(L\widehat\Omega^\bullet_{(V, \calv)/\calo_K}))
$$
By definition, the right hand side is a projective system of complexes of Zariski sheaves $(\Gamma(\calv, C^\bullet( L\Omega^\bullet_{(V, \calv)/\calo_K}/F^i)))$.

Now consider the $h$-topology on the category ${\rm Var}_{\overline K}$ of $\overline K$-varieties, and pull it back to the category ${\mathcal {SS}}_{\overline K}$ of semistable pairs over $\overline K$ via the forgetful functor ${\mathcal {SS}}_{\overline K}\to{\rm Var}_{\overline K}$. Sheafifying the above presheaf on ${\mathcal {SS}}_{\overline K}$ for the $h$-topology, we obtain an $h$-sheaf that we denote by ${\mathcal A}^\natural_{\rm dR}$ following Beilinson.

Again this defines an $h$-sheaf on the whole category of $\overline K$-varieties by the following analogue of Theorem \ref{hequiv}:

\begin{theo}\label{hequiv1}
The forgetful functor ${\mathcal {SS}}_{\overline K}\to{\rm Var}_{\overline K}$ induces an equivalence of categories between $h$-sheaves on ${\mathcal {SS}}_{\overline K}$ and $h$-sheaves on ${\rm Var}_{\overline K}$.
\end{theo}

\begin{dem} As in the proof of Theorem \ref{hequiv}, we proceed in two steps. We first apply Beilinson's Theorem \ref{beilh} in the situation where $C'$ is ${\rm Var}_{\overline K}$ equipped with the $h$-topology, and $F$ is the forgetful functor $(V, \calv)\to V$ from the category ${\mathcal {PP}}_{\overline K}$ of pairs $(V, \calv)$ consisting of a $\overline K$-variety $V$ and a reduced proper flat $\calo_{\overline K}$-scheme $\calv$ containing $V$ as a dense open subscheme. By the same arguments as in the geometric case, condition $(*)$ is satisfied, hence we obtain an equivalence of categories of $h$-sheaves.

Next we apply Theorem \ref{verdier} to the fully faithful inclusion of categories ${\mathcal {SS}}_{\overline K}\to {\mathcal {PP}}_{\overline K}$. We have to check that each pair $(V, \calv)$ in  ${\mathcal {PP}}_{\overline K}$  has an $h$-covering $(U, \calu)\to (V, \calv)$ by a pair in  ${\mathcal {SS}}_{\overline K}$. This follows from one of de Jong's alteration theorems \cite{dejong}: choosing a model $(V', \calv')$ of $(V, \calv)$ over a finite extension $K'|K$, there exists, up to replacing $K''$ by a finite extension, a semistable pair $(U', \calu')$ over $K'$ equipped with an alteration $(U', \calu')\to (V', \calv')$. As alterations are surjective and proper (and generically finite), this is an $h$-covering.
\end{dem}

We now compare the sheaf ${\mathcal A}^\natural_{\rm dR}$ with the sheaf $\widehat{\mathcal A}_{\rm dR}$ defined at the end of the previous subsection.

\begin{prop}\label{adrcomp}
We have a canonical isomorphism
$$
{\mathcal A}_{\rm dR}^\natural\otimes\Q\cong \widehat{\mathcal A}_{\rm dR}
$$
of projective systems of complexes of $h$-sheaves on ${\rm Var}_{\overline K}$.
\end{prop}

\begin{dem}
First, consider a pair $(U, \overline U)$ of $\overline K$-varieties such that $\overline U$ is proper smooth over $\overline K$ and $\overline U\setminus U$ is a normal crossing divisor. As before, equip $\overline U$ with the canonical log structure and $\overline K$ with the trivial log structure. Consider the derived logarithmic de Rham complex $L\Omega^\bullet_{(U, \overline U)/\overline K}$ arising from these data. Since $\overline U$ is log smooth and integral over $\overline K$, we have
$$
L\Omega^\bullet_{(U, \overline U)/\overline K}\simeq \Omega^\bullet_{(U, \overline U)/\overline K}
$$
where on the right hand side we have the non-derived logarithmic de Rham complex of the previous section.

Furthermore, we have $L_{\overline K/K}=\Omega^1_{\overline K/K}=0$ by Remark \ref{remsimpresomega} and a direct limit argument, so
$$
L\Omega^\bullet_{(U, \overline U)/\overline K}\simeq L\Omega^\bullet_{(U, \overline U)/K}.
$$
Finally, assume that $\overline U$ is the generic fibre of a proper flat $\calo_{\overline K}$-scheme $\calu$ such that $(U, \calu)$ is a semistable pair over $\overline K$ in the sense defined above, equipped with its log structure. Then by construction
$$
L\Omega^\bullet_{(U, \overline U)/K}\simeq L\Omega^\bullet_{(U, \calu)/\calo_K}\otimes\Q.
$$
Putting everything together, we thus have
$$
L\Omega^\bullet_{(U, \calu)/\calo_K}\otimes\Q\simeq \Omega^\bullet_{(U, \overline U)/\overline K}
$$
Passing to global sections of the associated Godement resolutions, we obtain an isomorphism of projective systems of filtered complexes
$$
(\Gamma(\calu, C^\bullet( L\widehat\Omega^\bullet_{(U, \calu)/\calo_K}/F^i))\otimes\Q)\simeq (\Gamma(\overline U, C^\bullet( \widehat\Omega^\bullet_{(U, \overline U)/\overline K}/F^i))).
$$
Passing to associated $h$-sheaves on ${\rm Var}_{\overline K}$, we finally obtain the stated isomorphism.
\end{dem}

Now recall that we defined $A_{\rm dR}:=L\widehat\Omega^\bullet_{\overline\calo_K/\calo_K}$. Consider the morphisms of log schemes
$$
(U, \calu)\stackrel\pi\to \Spec \calo_{\overline K}\to \Spec \calo_{ K}
$$
where the latter two schemes are equipped with a trivial log structure. This gives rise to a transitivity triangle of log cotangent complexes, whence a map
$$
\pi^*{L_{\calo_{\overline K}/\calo_K}}\to L_{(U, \calu)/\calo_K}.
$$
Similarly, there is a map of derived log de Rham complexes
$$
\pi^*{L\Omega^\bullet_{\calo_{\overline K}/\calo_K}}\to L\Omega^\bullet_{(U, \calu)/\calo_K}.
$$
Modding out by $F^i$, we may identify the left hand side with the constant (Zariski) sheaf on $\calu$ associated with $A_{\rm dR}/F^i$. As its higher cohomologies are trivial, we have a morphism of complexes (with $A_{\rm dR}/F^i$ placed in degree 0)
$$
A_{\rm dR}/F^i\to \Gamma(\calu, C^\bullet(L\Omega^\bullet_{(U, \calu)/\calo_K}/F^i)).
$$
Sheafifying for the $h$-topology we obtain morphisms
$$
A_{\rm dR}/F^i\to {\mathcal A}_{\rm dR}^\natural/F^i
$$
for all $i$, where we have a constant $h$-sheaf on the left hand side. Now we have:

\begin{theo}[Beilinson's $p$-adic Poincar\'e lemma]\label{poincare} The above maps induce quasi-isomorphisms
$$
(A_{\rm dR}/F^i)\widehat\otimes \Z_p\to ({\mathcal A}_{\rm dR}^\natural/F^i)\widehat\otimes\Z_p
$$
for all $i$.
\end{theo}

The proof will be given in the next section.

\begin{cor}
Assume $X$ is a smooth $K$-variety having a smooth normal crossing compactification. There are filtered quasi-isomorphisms
$$
R\Gamma_{\mbox{\rm\scriptsize\'et}}(X_{\overline K},\Z_p)\otimes_{\Z_p}(B_{\rm dR}^+/F^i)\stackrel\sim\to R\Gamma_h(X_{\overline K}, {\mathcal A}_{\rm dR}^\natural/F^i)\widehat\otimes\Q_p
$$
for all $i$, giving rise to a filtered quasi-isomorphism
$$
R\Gamma_{\mbox{\rm\scriptsize\'et}}(X_{\overline K},\Z_p)\otimes_{\Z_p}B_{\rm dR}^+\stackrel\sim\to R\Gamma_h(X_{\overline K}, {\mathcal A}_{\rm dR}^\natural)\widehat\otimes\Q_p
$$
in the limit.
\end{cor}

\begin{dem}
We start with the quasi-isomorphisms
$$
R\Gamma_{\mbox{\scriptsize\'et}}(X_{\overline K},\Z_p)\otimes^{L}_{\Z_p}(A_{\rm dR}/F^i)\simeq R\Gamma_{\mbox{\scriptsize\'et}}(X_{\overline K},A_{\rm dR}/F^i).
$$
Taking completed tensor product with $\Z_p$ (which is an exact functor) we obtain
$$
R\Gamma_{\mbox{\scriptsize\'et}}(X_{\overline K},\Z_p)\otimes^{L}_{\Z_p}(A_{\rm dR}/F^i)\widehat\otimes\Z_p\simeq R\Gamma_{\mbox{\scriptsize\'et}}(X_{\overline K},(A_{\rm dR}/F^i)\widehat\otimes\Z_p).
$$
Next, Corollary \ref{cordescent2} yields a quasi-isomorphism
$$
R\Gamma_{\mbox{\scriptsize\'et}}(X_{\overline K},(A_{\rm dR}/F^i)\widehat\otimes\Z_p)\simeq R\Gamma_{h}(X_{\overline K},(A_{\rm dR}/F^i)\widehat\otimes\Z_p)
$$
Applying the Poincar\'e lemma yields
$$
R\Gamma_{h}(X_{\overline K},(A_{\rm dR}/F^i)\widehat\otimes\Z_p)\simeq R\Gamma_{h}(X_{\overline K},({\mathcal A}^{\natural}_{\rm dR}/F^i)\widehat\otimes\Z_p)
$$
so, putting the above together (and using exactness of $\widehat\otimes\Z_p$ again)
$$
R\Gamma_{\mbox{\scriptsize\'et}}(X_{\overline K},\Z_p)\otimes^{L}_{\Z_p}(A_{\rm dR}/F^i)\widehat\otimes\Z_p\simeq R\Gamma_{h}(X_{\overline K},{\mathcal A}^{\natural}_{\rm dR}/F^i)\widehat\otimes\Z_p.
$$
On the other hand, by definition we have
$$
(A_{\rm dR}/F^i)\widehat\otimes\Z_p\otimes\Q\cong B_{\rm dR}^+/F^i,
$$
so the corollary follows by tensoring with $\Q$.
\end{dem}

\begin{cons}\rm
We are finally in the position to construct the comparison maps
$$
{\rm comp}_n:\, H^n_{\rm dR}(X)\otimes_K B_{\rm dR}\to H^n_{\mbox{\scriptsize\'et}}(X_{\overline K},\Z_p)\otimes_{\Z_p}B_{\rm dR}
$$
for $X$ as in the previous corollary following Beilinson's approach.

First recall that by Corollary \ref{coradr} and Proposition \ref{adrcomp}  we have filtered quasi-isomorphisms
$$
R\Gamma_{\rm dR}(X_{\overline K})^\wedge\simeq R\Gamma_h(X_{\overline K}, \widehat{\mathcal A}_{\rm dR})\simeq R\Gamma_h(X_{\overline K}, {\mathcal A}_{\rm dR}^\natural)\otimes\Q.
$$
On the other hand, there is a natural map
$$
R\Gamma_h(X_{\overline K}, {\mathcal A}_{\rm dR}^\natural)\to R\Gamma_h(X_{\overline K}, {\mathcal A}_{\rm dR}^\natural)\widehat\otimes\Z_p
$$
so after tensoring by $\Q$ and composing with the preceding isomorphisms we obtain a map
$$
R\Gamma_{\rm dR}(X_{\overline K})^\wedge\to R\Gamma_h(X_{\overline K}, {\mathcal A}_{\rm dR}^\natural)\widehat\otimes\Q_p.
$$
Applying  the previous corollary, we therefore have a natural map

$$
R\Gamma_{\rm dR}(X_{\overline K})^\wedge\to R\Gamma_{\mbox{\scriptsize\'et}}(X_{\overline K},\Z_p)\otimes_{\Z_p}B_{\rm dR}^+.
$$
Composing by the natural map $R\Gamma_{\rm dR}(X)^\wedge\to R\Gamma_{\rm dR}(X_{\overline K})^\wedge$ and extending $B_{\rm dR}^+$-linearly, this yields a map

$$
R\Gamma_{\rm dR}(X)^\wedge\otimes_K B_{\rm dR}^+\to R\Gamma_{\mbox{\scriptsize\'et}}(X_{\overline K},\Z_p)\otimes_{\Z_p}B_{\rm dR}^+
$$
compatible with filtrations. Passing to the fraction field of $B_{\rm dR}^+$ and taking cohomology, we obtain the announced comparison maps
$$
{\rm comp}_n:\, H^n_{\rm dR}(X)\otimes_K B_{\rm dR}\to H^n_{\mbox{\scriptsize\'et}}(X_{\overline K},\Z_p)\otimes_{\Z_p}B_{\rm dR}
$$
that are compatible with filtrations and Galois action. (Here we have used again that the Hodge filtration on the groups $H^n_{\rm dR}(X)$ is finite.)
\end{cons}

\section{The comparison theorem}

\subsection{Proof of the comparison isomorphism}\label{secproof}

This subsection is devoted to the proof of:

\begin{theo}\label{comptheo}{\rm (De Rham comparison theorem)}
The maps ${\rm comp}_n$ are filtered isomorphisms for all smooth quasi-projective $X$ and all $n$.
\end{theo}

We begin with the crucial case $X={\bf G}_{m,K}=\Spec K[x, x^{-1}]$. Since it is connected of dimension 1, only the case $n=1$ is nontrivial.
\begin{prop}\label{g_m}
The map ${\rm comp}_1$ induces a Galois-equivariant filtered isomorphism
$$
 H^1_{\rm dR}({\Bbb G}_{m, K})\otimes_K B_{\rm dR}\to H^1_{\mbox{\scriptsize\rm\'et}}({\Bbb G}_{m,\overline K},\Q_p)\otimes_{\Q_p}B_{\rm dR}.
$$
\end{prop}

\begin{proof} We may assume $K=\Q_p$ by a base change argument and drop the subscript from ${\Bbb G}_{m,{\Q_p}}$.
Since ${\rm comp}_1$ is compatible with filtrations, it suffices to show that it induces an isomorphism on associated graded rings. By Proposition \ref{BdRdiscreteval} we have ${\rm gr}^i_{\rm Fil}\Bdr\cong \C_p(i)$ for all $i$. On the other hand, since ${\Bbb G}_m$ is affine of dimension 1, we have $F^0H^1_{\rm dR}({\Bbb G}_m)=F^1H^1_{\rm dR}({\Bbb G}_{m})=H^1_{\rm dR}({\Bbb G}_{m})$ and $F^iH^1_{\rm dR}({\Bbb G}_{m})=0$ for $i>1$, whence an isomorphism ${\rm gr}^1_FH^1_{\rm dR}({\Bbb G}_{m})\cong H^1_{\rm dR}({\Bbb G}_{m})$. Thus it will suffice to show that ${\rm comp}_1$ induces an isomorphism
\begin{equation}\label{comp1}
H^1_{\rm dR}({\Bbb G}_{m})\otimes_{\Q_p}\C_p\stackrel\sim\to H^1_{\mbox{\scriptsize\rm\'et}}({\Bbb G}_{m,\overline\Q_p}, \Q_p)\otimes_{\Q_p}\C_p(1)
\end{equation}
as on the other graded pieces the maps will be just Galois twists of this one.

Both sides of (\ref{comp1}) are 1-dimensional $\C_p$-vector spaces. A generator for the left hand side is given by the logarithmic differential ${\rm dlog}(x)$, and of the right hand side by the compatible system $c_x$ of the images of the coordinate function $x$ by the Kummer maps $$
H^0_{\mbox{\scriptsize\'et}}({\Bbb G}_{m,\overline\Q_p}, \G)\to H^1_{\mbox{\scriptsize\'et}}({\Bbb G}_{m,\overline\Q_p}, \mu_{p^n})
$$
for all $n$. Another description of the class $c_x$ is as follows. The \'etale fundamental group $\Pi:=\pi_1({\Bbb G}_{m,\overline\Q_p})$ is pro-cyclic, whence an isomorphism $H^1_{\mbox{\scriptsize\'et}}({\Bbb G}_{m,\overline\Q_p}, \mu_{p^n})\cong \Z/p^n\Z$, a generator being given by the class of the $\mu_{p^n}$-torsor $\widetilde\G$ coming from the map $x\mapsto x^{p^n}$ on  ${\Bbb G}_{m,\overline\Q_p}$. The compatible system of these for all $n$ forms a pro-torsor whose class generates  $H^1_{\mbox{\scriptsize\'et}}({\Bbb G}_{m,\overline\Q_p}, \Z_p(1))$.

It thus suffices to check that ${\rm comp}_1$ sends the class of ${\rm dlog}(x)$ to that of $c_x$ modulo the identification ${\rm gr}^1_{\rm Fil}\Bdr\cong \C_p(1)$. As we have seen in Subsection \ref{secfontaine}, this isomorphism is induced by the map $\mu_{p^n}\to {}_{p^n}\Omega^1_{\OKB/\OK}\cong{\rm gr}^1_F A_{\rm dR}\otimes^L\Z/p^n\Z$. The latter group, viewed as a constant $h$-sheaf, is isomorphic to ${\rm gr}^1_F{\mathcal A}^\natural_{\rm dR}\otimes^L\Z/p^n\Z$ by the Poincar\'e lemma (Theorem \ref{poincare}). Restricting to the \'etale topology we thus have a map $\rho:\, H^1_{\mbox{\scriptsize\'et}}({\Bbb G}_{m,\overline\Q_p}, \mu_{p^n})\to H^1_{\mbox{\scriptsize\'et}}({\Bbb G}_{m,\overline\Q_p}, {\rm gr}^1_F{\mathcal A}^\natural_{\rm dR}\otimes^L\Z/p^n\Z).$ On the other hand, we may identify  ${\rm dlog}(x)\in {\rm gr}^1_FH^1_{\rm dR}({\Bbb G}_{m})$ with a class in $H^1_h({\Bbb G}_{m,\overline\Q_p}, {\rm gr}^1_F{\mathcal A}_{\rm dR})$ via Theorem \ref{adr}. As this class is defined over $\OKB$, we may view it as a cohomology class with values in ${\rm gr}^1_F{\mathcal A}^\natural_{\rm dR}$ (as an $h$-sheaf on ${\rm Var}_{\overline\Q_p}$) and send it to a class in $H^1_{h}({\Bbb G}_{m,\overline\Q_p}, {\rm gr}^1_F{\mathcal A}^\natural_{\rm dR}\otimes^L\Z/p^n\Z)\cong H^1_{\mbox{\scriptsize\'et}}({\Bbb G}_{m,\overline\Q_p}, {\rm gr}^1_F{\mathcal A}^\natural_{\rm dR}\otimes^L\Z/p^n\Z).$ We compute the latter group as group cohomology of $\Pi$ with values in ${\mathcal A}^\natural_{\rm dR}\otimes[\Z\stackrel{p^n}\to\Z]$. Both classes are represented by an element in the 1-cochain group $C^0(\Pi, {\mathcal A}^\natural_{\rm dR}[1])\oplus C^1(\Pi, {\mathcal A}^\natural_{\rm dR})$. The 0-cochain group $C^0(\Pi, {\mathcal A}^\natural_{\rm dR}[1])$ maps to this group via multiplication by $p^n$ in the first component and the natural identification in the second with a minus sign. Now let $\widetilde x$ be the coordinate function on the $\mu_{p^n}$-torsor $\widetilde G\cong \G\stackrel{p^n}\to\G$. We represent ${\rm dlog}(\widetilde x)$ by a 0-cochain with values in ${\mathcal A}^\natural_{\rm dR}[1]$ and compute $p^n{\rm dlog}(\widetilde x)={\rm dlog}(\widetilde x^{p^n})={\rm dlog}(x)$. On the other hand, under the identification $C^0(\Pi, {\mathcal A}^\natural_{\rm dR}[1])\cong C^1(\Pi, {\mathcal A}^\natural_{\rm dR})$ the class ${\rm dlog}(\widetilde x)$ goes over to the 1-cocycle $\sigma\mapsto \sigma({\rm dlog}(\widetilde x))-{\rm dlog}(\widetilde x)$ which represents $\rho(c_x)$. Therefore the two 1-cocyles are cohomologous.
\end{proof}

\begin{rema}\rm
The isomorphism of the proposition sends the class of the element ${\rm dlog}(x)\otimes 1$ in ${H^1_{\rm dR}({\Bbb G}_{m, K})\otimes_K B_{\rm dR}}$ to $c_x\otimes (\iota\otimes{\C_p})\in H^1_{\mbox{\scriptsize\rm\'et}}({\Bbb G}_{m,\overline K},\Q_p(1))\otimes_{\Q_p}B_{\rm dR}(-1)$, where $c_x$ is as in the above proof and $\iota:\, \Z_p(1)\to B_{\rm dR}$ is the map of Construction \ref{consfonel} defining the Fontaine element. Indeed, the elements ${\rm dlog}(x)\otimes 1$ and $c_x$ are equal up to multiplication by an element $\lambda\in B_{\rm dR}(-1)$ in the 1-dimensional $\Bdr$-vector space $H^1_{\mbox{\scriptsize\rm\'et}}({\Bbb G}_{m,\overline K},\Q_p)\otimes_{\Q_p}B_{\rm dR}$, and the calculation in the above proof together with Proposition \ref{fontaineelement} shows that $\lambda$ and $\iota\otimes{\C_p}$ coincide modulo ${\rm Fil}^2$. Hence their difference is a Galois-invariant element in ${\rm Fil}^2\Bdr(-1)$, which must be 0 by Proposition \ref{BdRdiscreteval} and Tate's theorem cited in formula (\ref{tatetheo}) of the introduction.

This is analogous to the isomorphism of complex de Rham theory for $\G$ that maps ${\rm dlog}(x)$ to the linear map $H_1(\G,\Z)\to\C$ with value $2\pi i$ on a generator of $H_1(\G,\Z)\cong\Z$.
\end{rema}

The next crucial point is compatibility of the comparison map with Gysin maps in codimension 1. We explain these for the \'etale theory; the de Rham theory is similar. In fact, as explained in (\cite{BO}, \S 2), both \'etale cohomology and algebraic de Rham cohomology satisfy the axioms of a `Poincar\'e duality theory with supports' in the sense of that paper, and the properties of cohomology we are to use are all valid for theories satisfying these axioms.

Given a pair $Y\subset X$ of $\overline K$-varieties, there are cohomology groups with support $H^i_{Y}(X, \Q_p(r))$ fitting into a long exact sequence
$$
\cdots\to  H^n_{\mbox{\scriptsize\'et}}(X, \Q_p(r))\to H^n_{\mbox{\scriptsize\'et}}(X\setminus Y, \Q_p(r))\to H^{n+1}_{Y}(X, \Q_p(r))\to\cdots
$$
One can in fact construct this sequence (and the similar one in de Rham cohomology) by defining $R\Gamma_Y(X)$ to be the cone of the natural pullback map $R\Gamma(X)\to R\Gamma(X\setminus Y)$. As a consequence of this cone construction, we may extend the definition of the comparison maps  ${\rm comp}_n$ to cohomology with support in $Y$.

If moreover both $X$ and $Y$ are smooth and $Y$ is of codimension 1 in $X$, there are  purity isomorphisms (sometimes called Gysin isomorphisms)
$$
H^n_{\mbox{\scriptsize\'et}}(Y, \Q_p(r))\cong H^{n+2}_Y(X, \Q_p(r+1))
$$
for cohomology with support. Composing with the natural map $$H^{n+2}_Y(X, \Q_p(r+1))\to H^{n+2}_{\mbox{\scriptsize\'et}}(X, \Q_p(r+1))$$ we obtain the Gysin maps
$$
H^n_{\mbox{\scriptsize\'et}}(Y, \Q_p(r))\to H^{n+2}_{\mbox{\scriptsize\'et}}(X, \Q_p(r+1))
$$

We first study the Gysin map in a special situation. Consider a line bundle $\mathcal L$ on a smooth $Y$; this is a locally free $\calo_Y$-module of rank 1. The corresponding geometric line bundle is denoted by $V({\mathcal L})\to Y$. As such, it is equipped with the zero section $Y\to V({\mathcal L})$ which identifies $Y$ with a smooth codimension 1 closed subscheme in ${V}({\mathcal L})$.

\begin{lem}
The maps ${\rm comp}_n$ are compatible with Gysin isomorphisms
associated with closed embeddings $i:\, Y\hookrightarrow {V}({\mathcal L})$ as above.
\end{lem}

\begin{dem} In both the \'etale and the de Rham theories, the projection $\pi:\, V({\mathcal L})\to Y$ induces a map of cohomology rings $\pi^*:\, H^*(Y)\to H^*_Y(V({\mathcal L}))$ that equips the latter ring with an $H^*(Y)$-module structure induced by the cup-product. The map $i_*$ respects this module structure, and therefore for all $\alpha\in H^n(Y)$ we have $i_*(\alpha)=i_*(1)\cup \pi^*(\alpha)$ where $1\in H^0(Y)$. Thus we reduce to showing that the maps ${\rm comp}_2$ preserve the classes $i_*(1)\in H^2_Y(V({\mathcal L}))$.
Pick an open covering trivializing the line bundle $\mathcal L$. By the Mayer-Vietoris sequences
$$
\cdots \to H^{n-1}(U\cap V)\to H^n(U\cup V)\to H^n(U)\oplus H^n(V)\to\cdots
$$
in both theories (and their analogues with support) we reduce to the case where $\mathcal L$ is trivial, i.e. $V({\mathcal L})\cong Y\times\A^1$. Now consider the commutative diagram of pairs
$$
\begin{CD}
(Y\times \A^1, Y\times \{0\}) @>>> (\A^1, \{0\})\\
@AAA @AAA \\
(Y, Y) @>>> (\Spec\overline K, \Spec\overline K)
\end{CD}
$$
inducing a commutative diagram
$$
\begin{CD}
H^2_{Y\times \{0\}}(Y\times \A^1) @<<< H^2_{\{0\}}(\A^1)\\
@A{i_*}AA @AA{i_*}A \\
H^0(Y) @<\cong<< H^0(\Spec\overline K)
\end{CD}
$$
It shows that when identifying $i_*(1)$ we may reduce to the case where $Y$ is a point. But then the localization sequence induces an isomorphism
$$
H^1({\bf G}_{m, \overline K})\stackrel\sim\to H^2_{\{0\}}(\A^1)
$$
since $H^n(\A^1)=0$ for $n>0$, and one checks that under this isomorphism the elements $i_*(1)$ map to the distinguished elements described in Proposition \ref{g_m}. Thus the statement follows from the proposition.
\end{dem}

\begin{prop}\label{gysin}
The maps ${\rm comp}_n$ are compatible with all Gysin isomorphisms associated with closed embeddings of smooth codimension 1 subvarieties.
\end{prop}

The proof uses a `deformation to the normal cone' (in this case, normal bundle) construction that we recall next. A reference is \cite{dv}.

\begin{cons}\rm
Let ${Y\subset X}$ be a smooth codimension 1 pair as above, and denote by $\mathcal N$ the normal bundle of $Y$ in $X$. There exists a closed embedding $Y\times\A^1\hookrightarrow M^\circ$ in a $\overline K$-variety $M^\circ$ equipped with a projection $p:\, M^\circ\to \A^1$ such that the composite $Y\times\A^1\hookrightarrow M^\circ\to\A^1$ is the natural projection $p_2$, and moreover the following properties hold.
\begin{enumerate}
 \item There is an isomorphism $p^{-1}(\A^1\setminus \{0\})\cong X\times(\A^1\setminus \{0\})$
 making the diagram
 $$
 \begin{CD}
 p_2^{-1}({\A^1\setminus \{0\}}) @>>> p^{-1}(\A^1\setminus \{0\})\\
 @V=VV @VV{\cong}V \\
 {Y\times(\A^1\setminus \{0\})}@>>> X\times(\A^1\setminus \{0\})
 \end{CD}
 $$
 commute, where the bottom horizontal map is the natural inclusion.

\item There is an isomorphism $p^{-1}(0)\cong V({\mathcal N})$ making the diagram
$$
 \begin{CD}
 p_2^{-1}(0)  @>>> p^{-1}(0)\\
 @V{\cong}VV @VV{\cong}V \\
 Y @>>> V({\mathcal N})
 \end{CD}
 $$
commute, where the bottom horizontal map is the embedding of $Y$ via the zero section.\smallskip
\end{enumerate}
The construction of $M^\circ$ is as follows. Consider the closed embedding ${Y\times \A^1}\hookrightarrow X\times \A^1$ and blow up the closed subscheme $Y\times\{0\}$ in $X\times \A^1$. The resulting blowup $M\to X\times\A^1$ is equipped with a natural projection $p:\, M\to \A^1$ compatible with ${p_2: Y\times\A^1\to \A^1}$. Now over $\A^1\setminus \{0\}$ the situation is as above because the blowup did not change $X\times (\A^1\setminus \{0\})$. The fibre $p^{-1}(0)$ decomposes in two components $Z_1$ and $Z_2$. The component $Z_1$ is  isomorphic to the blowup of $Y$ in $X$, and $Z_2$ is the projective line bundle ${\bf P}({\mathcal N}\oplus \calo_Y)$. Furthermore, the inclusion $Z_1\cap Z_2\hookrightarrow Z_2$ is the inclusion of the `hyperplane at infinity' in ${\bf P}({\mathcal N}\oplus \calo_Y)$; its complement is $V({\mathcal N})$. Setting $M^\circ:=M\setminus Z_1$ we thus arrive at the situation described above.\end{cons}

\noindent{\em Proof of Proposition \ref{gysin}.} The geometric construction described above gives rise to commutative diagrams in both cohomology theories
$$
\begin{CD}
H^{n+2}_Y(X) @<<< H^{n+2}_{Y\times\A^1}(M^\circ) @>>> H^{n+2}_Y(V({\mathcal N})) \\
@AA{\cong}A @AA{\cong}A @AA{\cong}A \\
H^{n}(Y) @<{\cong}<< H^{n}(Y\times\A^1) @>{\cong}>> H^{n}(Y)
\end{CD}
$$
The vertical maps are Gysin isomorphisms and the horizontal maps are pullbacks associated with $Y\times \{1\}\to Y\times \A^1$ on the left and $Y\times \{0\}\to Y\times \A^1$ on the right (and the inclusions $X\times \{1\}\hookrightarrow M^\circ \hookleftarrow V({\mathcal N})$ above). The lower horizontal maps are isomorphisms by homotopy invariance of de Rham and \'etale cohomology, hence so are the upper horizontal maps. We thus reduce to the case treated in the previous lemma.
\enddem

\noindent{\em Proof of Theorem \ref{comptheo}.} First assume $X$ is smooth and projective of dimension $d$. Consider a smooth hyperplane section $H\subset X$. It exists by the Bertini theorem and is a smooth codimension 1 subvariety of $X$. It has a class $\eta\in H^2_{\mbox{\scriptsize\'et}}(X_{\overline K}, \Q_p(1))$ which is the image of 1 by the Gysin map $$H^0_{\mbox{\scriptsize\'et}}(X_{\overline K}, \Q_p)\cong H^2_{Y}(X_{\overline K}, \Q_p(1))\to H^2_{\mbox{\scriptsize\'et}}(X_{\overline K}, \Q_p(1)).$$
Similar facts hold for de Rham cohomology. The second map here comes from a long exact sequence associated to a cone of a pullback map, hence it commutes with the comparison map. From the previous proposition we therefore conclude that ${\rm comp}_2$ is compatible with the above Gysin map.
Furthermore, the $d$-fold cup-product $\eta^d$ generates the group $$H^{2d}_{\mbox{\scriptsize\'et}}(X_{\overline K}, \Q_p(d))\cong \Q_p$$
and similarly for de Rham cohomology. Since the maps ${\rm comp}_n$ are compatible with the product structures on de Rham and \'etale cohomology, we conclude that these isomorphisms are compatible with each other via ${\rm comp}_{2d}$; in particular, ${\rm comp}_{2d}$ is an isomorphism.

Now observe that both cohomology algebras are equipped with Poin\-car\'e duality pairings which are non-degenerate. Thus if $\alpha\in H^n_{\rm dR}(X)$ is a nonzero element, there is $\beta\in H^{2d-n}_{\rm dR}(X)$ such that $\alpha\cdot \beta\neq 0$. Therefore, since ${\rm comp}_{2d}$ is an isomorphism and the Poincar\'e duality pairing on \'etale cohomology is non-degenerate, we have ${\rm comp}_{n}(\alpha)\neq 0$. But then ${\rm comp}_{n}$ is injective for all $n$. On the other hand, we know that the source and the target of ${\rm comp}_n$ are finite-dimensional vector spaces of the same dimension over $B_{\rm dR}$. This results by a Lefschetz principle argument from the isomorphism $H^n_{\rm dR}(X_\C)\cong H^n(X_\C^{\rm an},\C)$ for complex smooth projective varieties recalled in the introduction to this paper, i.e. the comparison between algebraic and analytic de Rham cohomology and the complex Poincar\'e lemma. We conclude that ${\rm comp}_{n}$ is an isomorphism for all $n$.

Now if $X$ is only assumed to be smooth and quasi-projective, by Hironaka's theorem it has a smooth projective compactification $\overline X$ with complement a normal crossing divisor $D$ whose components are smooth. We prove the theorem by a double induction on the dimension $d$ of $X$ and the number $r$ of components of $D$; the case $r=0$ is the projective case treated above. Now fix a component $D_0$ of $D$, and let $D'$ be the union of the other components. Then $\overline X\setminus D'$ has $r-1$ components at infinity and $(\overline X\setminus D')\setminus (D_0\setminus (D_0\cap D'))=X$. There are localization sequences in both theories of the form
$$
H^{n-2}(D_0\setminus (D_0\cap D'))\to H^n(\overline X\setminus D')\to H^n(X)\to H^{n-1}(D_0\setminus (D_0\cap D'))
$$
coming from { exact Gysin triangles}, hence compatible with the comparison maps by the previous proposition. The comparison maps are isomorphisms for $D_0\setminus (D_0\cap D')$ by induction on $d$ and  for $\overline X\setminus D'$ by induction on $r$, so they are isomorphisms for $H^n(X)$ as well.
\enddem

\subsection{Proof of the Poincar\'e lemma}

This section is devoted to the proof of Beilinson's Poincar\'e Lemma (Theorem \ref{poincare}). We begin with auxiliary statements about log differentials.

Recall that for a semistable pair $(U, \calu)$ over $\calo_K$ we have denoted by $L_{(U, \calu)/\calo_K}$ the log cotangent complex where $\calo_K$ is equipped with the trivial log structure, and similarly for log differentials and the (derived) log de Rham algebra. We shall also consider these objects in the case where $\calo_K$ (or an extension of it) is equipped with the canonical log structure coming from the inclusion of the closed point in $\Spec\calo_K$; we denote the corresponding objects by $L_{(U, \calu)/(K,\calo_K)}$ and similarly for differentials. To compare the two, the following lemma will be handy.

\begin{lem}
 There is a natural quasi-isomorphism
$$
L_{(\Spec\overline K,\Spec\calo_{\overline K})/\calo_K}\cong \Omega^1_{(\overline K,\calo_{\overline K})/\calo_K},$$
where we have logarithmic 1-forms on the right hand side, and $\calo_K$ carries the trivial log structure. Moreover, the natural map
$$
\Omega^1_{\calo_{\overline K}/\calo_K}\to \Omega^1_{(\overline K,\calo_{\overline K})/\calo_K}
$$
from usual differentials is an isomorphism.
\end{lem}

\begin{dem} The first statement is proven exactly as its non-logarithmic analogue (Lemma \ref{simpresOmega}). The proof of that statement was based on two properties of the cotangent complex: the transitivity triangle and the computation of the cotangent complex of a polynomial algebra. As recalled in Subsection \ref{logsection}, both of these properties have analogues for Gabber's log cotangent complex, to be found in  \cite{olsson}.

For the second statement, we may replace $K$ by its maximal unramified extension. Consider first a finite extension $L|K$ generated by a uniformizer $\pi$ of $\calo_L$ with minimal polynomial $f$. As recalled in Facts \ref{localfacts}, the $\calo_L$-module $\Omega^1_{\calo_L/\OK}$ is generated by $d\pi$ with annihilator the principal ideal generated by $f'(\pi)$. Similarly, the construction of log differentials shows that $\Omega^1_{(L,\calo_{L})/\calo_K}$ is a quotient of the free module generated by $d\pi/\pi$ modulo the submodule generated by $f'(\pi)$. Thus the natural map $\Omega^1_{\calo_{L}/\calo_K}\to \Omega^1_{(L,\calo_{L})/\calo_K}$ can be identified with the inclusion
\begin{equation}\label{nonlog}
(\calo_{L}/f'(\pi)\calo_{L})d\pi\hookrightarrow (\pi^{-1}\calo_{L}/f'(\pi)\calo_{L})d\pi
\end{equation}
whose cokernel is killed by $\pi$ and hence by $p$. By passing to the direct limit, we deduce that the map $\Omega^1_{\calo_{\overline K}/\calo_K}\to \Omega^1_{(\overline K,\calo_{\overline K})/\calo_K}$ is injective with cokernel killed by $p$. To show that the cokernel is in fact trivial, it will suffice to verify that $\Omega^1_{\calo_{\overline K}/\calo_K}$ contains the $p$-torsion of $\Omega^1_{(\overline K,\calo_{\overline K})/\calo_K}$. Indeed, given $\omega\in\Omega^1_{(\overline K,\calo_{\overline K})/\calo_K}$, we have $p\omega\in \Omega^1_{\calo_{\overline K}/\calo_K}$, but the latter group is $p$-divisible by Corollary \ref{cp2}, so after modifying $\omega$ by a $p$-torsion element we obtain an element in $\Omega^1_{\calo_{\overline K}/\calo_K}$.

Assume therefore  $\omega\in\Omega^1_{(\overline K,\calo_{\overline K})/\calo_K}$ is a $p$-torsion element, coming from an element $\omega_L\in \Omega^1_{(L,\calo_{L})/\calo_K}$ for some finite extension $L|K$. As $p\omega_L$ maps to 0 in $\Omega^1_{\OKB/\calo_K}$, we conclude $p\omega_L=0$ from Lemma \ref{lemfontaine} (1). On the other hand, by Corollary \ref{cp2} the $\OKB$-module $\Omega^1_{\calo_{\overline K}/\calo_K}$, which is the direct limit of the modules $\Omega^1_{\calo_{L}/\calo_K}$,  is nonzero and $p$-divisible, and therefore for $L$ large enough we must have $p\Omega^1_{\calo_{L}/\calo_K}\neq 0$. In particular, $p$ does not lie in the annihilator $(f'(\pi))$ of $\Omega^1_{\calo_{L}/\calo_K}$, i.e. $p/f'(\pi)\notin\calo_L$. But then $f'(\pi)/p\in \calo_L$ and hence $f'(\pi)\omega_L=(f'(\pi)/p)p\omega_L=0$. Since under the inclusion (\ref{nonlog})
the left hand side becomes identified with the part of the right hand side killed by $f'(\pi)$, this means that $\omega_L$ comes from $\Omega^1_{\calo_{L}/\calo_K}$, as desired.
\end{dem}

Consider now a semi-stable pair $(V, \calv)$ over $\overline K$; recall that it comes from a semistable pair $(U, \calu)$ defined over a finite extension $K'|K$.

\begin{prop} We have a natural quasi-isomorphism $$L_{(V, \calv)/\calo_K}\cong \Omega^1_{(V, \calv)/\calo_K}.$$

Moreover, the right hand side sits in a short exact sequence of log $\calo_{\calv}$-modules
$$
0\to \calo_{\calv}\otimes_{\calo_{\overline K}}\Omega^1_{\calo_{\overline K}/\calo_K}\to \Omega^1_{(V, \calv)/\calo_K}\to \Omega^1_{(V, \calv)/(\overline K, \calo_{\overline K})}\to 0.
$$
Here the last term is locally free, hence the sequence is locally split.
\end{prop}

\begin{dem} Consider the exact triangle of log cotangent complexes
$$
\calo_{\calv}\otimes_{\calo_{\overline K}}L_{\calo_{(\overline K,\calo_{\overline K})/\calo_K}}\to L_{(V, \calv)/\calo_K}\to L_{(V, \calv)/(\overline K, \calo_{\overline K})}\to \calo_{\calv}\otimes_{\calo_{\overline K}}L_{\calo_{(\overline K,\calo_{\overline K})/\calo_K}}[1]
$$
coming from the sequence of morphisms of log schemes
$$
\calv\to\Spec\calo_{\overline K}\to\Spec\calo_K,
$$
where the first two terms carry the canonical log structure and the third the trivial one. Here the term $L_{(V, \calv)/(\overline K, \calo_{\overline K})}$ is a direct limit of cotangent complexes $L_{(U_{K'}, \calu_{K'})/(K', \calo_{K'})}$ for finite extensions $K'|K$. Since by assumption for sufficiently large $K'|K$ the morphisms $(U_{K'}, \calu_{K'})\to(K', \calo_{K'})$ are log smooth and integral, we have quasi-isomorphisms
$$L_{(U_{K'}, \calu_{K'})/(K', \calo_{K'})}\cong \Omega^1_{(U_{K'}, \calu_{K'})/(K', \calo_{K'})}$$ and the latter terms are locally free of finite rank independent of $K'$. Hence the same is true of $L_{(V, \calv)/(\overline K, \calo_{\overline K})}$. Using the first statement of the previous lemma we may thus rewrite the triangle as
$$
\calo_{\calv}\otimes_{\calo_{\overline K}}\Omega^1_{(\overline K,\calo_{\overline K})/\calo_K}\to L_{(V, \calv)/\calo_K}\to \Omega^1_{(V, \calv)/(\overline K, \calo_{\overline K})}\to \calo_{\calv}\otimes_{\calo_{\overline K}}\Omega^1_{(\overline K,\calo_{\overline K})/\calo_K}[1]
$$
We obtain the quasi-isomorphism $L_{(V, \calv)/\calo_K}\cong \Omega^1_{(V, \calv)/\calo_K}$ by comparing this triangle with the one coming from the exact sequence
$$
0\to \calo_{\calv}\otimes_{\calo_{\overline K}}\Omega^1_{(\overline K,\calo_{\overline K})/\calo_K}\to \Omega^1_{(V, \calv)/\calo_K}\to \Omega^1_{(V, \calv)/(\overline K, \calo_{\overline K})}\to 0
$$
of log differentials (which is exact on the left again because $(V, \calv)\to(\overline K, \calo_{\overline K})$ is a limit of log smooth integral maps). Finally, we identify $\Omega^1_{(\overline K,\calo_{\overline K})/\calo_K}$ with $\Omega^1_{\calo_{\overline K}/\calo_K}$ using the second statement of the previous lemma.
\end{dem}

Now recall that the graded pieces of the Hodge filtration on the logarithmic derived de Rham algebra $L\Omega^\bullet_{(V, \calv)/\calo_K}$ are described by
$$
{\rm gr}^i_F(L\Omega^\bullet_{(V, \calv)/\calo_K})\cong L\wedge^iL_{(V, \calv)/\calo_K}[-i]
$$
which we may rewrite using the previous proposition as
$$
{\rm gr}^i_F(L\Omega^\bullet_{(V, \calv)/\calo_K})\cong L\wedge^i\Omega^1_{(V, \calv)/\calo_K}[-i].
$$
Using the exact sequence of the proposition, we may unscrew these objects further as follows.

\begin{prop}
There exists a filtration $I_a$ on ${\rm gr}^i_FL\Omega^\bullet_{(V, \calv)/\calo_K}$ with graded pieces given by
$$
{\rm gr}_a^I{\rm gr}^i_FL\Omega^\bullet_{(V, \calv)/\calo_K}\cong
{\rm gr}^{i-a}_FA_{\rm dR}[-a]\otimes_{\calo_{\overline K}} \Omega^a_{(V, \calv)/(\overline K, \calo_{\overline K})}.
$$
\end{prop}

\begin{proof}
We apply Construction \ref{lwedgefilt} of the appendix to the exact sequence of the preceding proposition. It gives a filtration
$$
I_a=\im((\calo_{\calv}\otimes_{\calo_{\overline K}}L\wedge^{i-a}\Omega^1_{\calo_{\overline K}/\calo_K})\otimes_{\calo_\calv} \Omega^a_{(V, \calv)/(\overline K, \calo_{\overline K})})\to L\wedge^i\Omega^1_{(V, \calv)/\calo_{\overline K}})
$$
on
$$
L\wedge^i\Omega^1_{(V, \calv)/\calo_{\overline K}}={\rm gr}^i_FL\Omega^\bullet_{(V, \calv)/\calo_K}[i].
$$
Here the induced map on graded pieces is injective as the sequence is locally split. Moreover, recall that by definition
$$
{\rm gr}^{i-a}_FA_{\rm dR}={\rm gr}^{i-a}_FL\Omega^\bullet_{\calo_{\overline K}/\calo_K}=L\wedge^{i-a}\Omega^1_{\calo_{\overline K}/\calo_K}[a-i]
$$
whence the description of $
{\rm gr}_a^I{\rm gr}^i_F(L\Omega^\bullet_{(V, \calv)/\calo_K})$.
\end{proof}

We may sheafify the statement of the above proposition as follows. Apply the functor $R\Gamma(\calv, \cdot)$ to $L\Omega^\bullet_{(V, \calv)/\calo_K}$ and take the associated $h$-sheaf (using Godement resolutions in a by now familiar fashion). Further, denote by ${\mathcal G}^a$ the complex of $h$-sheaves associated with $$(V, \calv)\mapsto R\Gamma(\calv,\Omega^a_{(V, \calv)/\overline K})=\Gamma(\calv,C^\bullet\Omega^a_{(V, \calv)/\overline K}).$$ The proposition then yields:

\begin{cor}
There exists a filtration $I_a$ on ${\rm gr}^i_F{\mathcal A}_{\rm dR}^\natural$ with graded pieces given by
$$
{\rm gr}_a^I{\rm gr}^i_F{\mathcal A}_{\rm dR}^\natural\cong
{\rm gr}^{i-a}_FA_{\rm dR}[-a]\otimes^{\bf L}_{\calo_{\overline K}} {\mathcal G}^a.
$$
\end{cor}

This corollary enables us to make an important reduction in the proof of the Poincar\'e lemma.

\begin{cor}
Theorem \ref{poincare} follows from from the vanishing statements
$$\tau_{>0}{\mathcal G}^0\otimes^{\bf L}\Z/p\Z=0$$ and
$${\mathcal G}^a\otimes^{\bf L}\Z/p\Z=0
$$
for all $a>0$.
\end{cor}

\begin{proof} To prove the theorem, it suffices to prove that the maps
$$
(A_{\rm dR}/F^i)\otimes^{\bf L} \Z/p^r\Z\to ({\mathcal A}_{\rm dR}^\natural/F^i)\otimes^{\bf L}\Z/p^r\Z
$$
are quasi-isomorphisms for all $r>0$, for afterwards we may pass to the limit. Using induction along the exact sequences
$$
0\to \Z/p^{r-1}\Z\to \Z/p^r\Z\to \Z/p\Z\to 0
$$
we reduce to the case $r=1$.
This case amounts to proving that
$$
{\rm Cone}({\rm gr}^i_FA_{\rm dR}\to {\rm gr}^i_F{\mathcal A}_{\rm dR}^\natural)\otimes^{\bf L}\Z/p\Z=0.
$$
Consider the 0-th step of the $I$-filtration on ${\rm gr}^i_FA_{\rm dR}^\natural$. By definition, it is given by the term ${\rm gr}^{i}_FA_{\rm dR}\otimes^{\bf L}_{\calo_{\overline K}} {\mathcal G}^0$. The cohomology sheaf ${\mathcal H}^0{\mathcal G}^0$ is the $h$-sheaf associated with $(V, \calv)\mapsto H^0(\calv, \calo_\calv)$. If the smooth proper $\overline K$-scheme $\calv$ is connected (which we may assume), we have $H^0(\calv, \calo_\calv)=\calo_{\overline K}$, and therefore
$${\rm gr}^{i}_FA_{\rm dR}\otimes^{\bf L}_{\calo_{\overline K}} {\mathcal H}^0{\mathcal G}^0\cong {\rm gr}^{i}_FA_{\rm dR},$$
which means that ${\rm gr}^i_FA_{\rm dR}$ already sits inside $I_0({\rm gr}^i_F{\mathcal A}_{\rm dR}^\natural)$, and the cone of the map ${\rm gr}^i_FA_{\rm dR}\to I_0({\rm gr}^i_F{\mathcal A}_{\rm dR}^\natural)$ is ${\rm gr}^{i}_FA_{\rm dR}\otimes^{\bf L}_{\calo_{\overline K}} \tau_{>0}{\mathcal G}^0$. Thus the nullity of $
{\rm Cone}({\rm gr}^i_FA_{\rm dR}\to I_0{\rm gr}^i_F{\mathcal A}_{\rm dR}^\natural)\otimes^{\bf L}\Z/p\Z
$ follows from the first vanishing statement above, and the second one yields the vanishing of the higher graded pieces of ${\rm gr}^i_F{\mathcal A}_{\rm dR}^\natural$ in view of the previous corollary.
\end{proof}

Finally, we translate the vanishing conditions of the corollary in a more tractable form.

\begin{lem}
Assume that for every semistable pair $(V, \calv)$ over $\overline K$ there is an $h$-covering $h:\, (V', \calv')\to (V, \calv)$ of semistable pairs such that the induced maps $$h^* : H^b(\calv, \Omega^a_{V, \calv})\to H^b(\calv', \Omega^a_{V', \calv'})$$ factor as
$$H^b(\calv, \Omega^a_{V, \calv})\stackrel p\to H^b(\calv, \Omega^a_{V, \calv})\to  H^b(\calv', \Omega^a_{V', \calv'}),$$
where the first map is multiplication by $p$. Then the vanishing statements of the previous corollary hold.
\end{lem}

\begin{proof}
The vanishing statements in question mean that the cohomology sheaves ${\mathcal H}^b{\mathcal G}^a$ are uniquely $p$-divisible for all $(a,b)$ except for $a=b=0$. The condition above yields $p$-divisibility in view of the commutative diagram
$$
\begin{CD}
H^b(\calv, \Omega^a_{V, \calv})@>p>> H^b(\calv, \Omega^a_{V, \calv})\\
@VVV @VVV \\
H^b(\calv', \Omega^a_{V', \calv'}) @>p>> H^b(\calv', \Omega^a_{V', \calv'}).
\end{CD}
$$
On the other hand, if $\alpha\in H^b(\calv, \Omega^a_{V, \calv})$ satisfies $pf^*\alpha=0$ for an $h$-covering $f:\, (V', \calv')\to (V, \calv)$, taking a further $h$-covering $h:\, (V'', \calv'')\to (V', \calv')$ with the property of the lemma ensures that $(h\circ f)^*\alpha=0$.
\end{proof}

Since $(V, \calv)$ come from a semistable pair $(U, \calu)$ defined over some finite extension $K'|K$ and $$H^b(\calv, \Omega^a_{V, \calv})\cong H^b(\calu, \Omega^a_{U, \calu})\otimes_{\calo_{K'}}\calo_{\overline K},$$ to verify the condition of the lemma
it will suffice to prove the corresponding statement over $K'$. Without loss of generality we may assume $K'=K$, so the proof of Theorem \ref{poincare} finally reduces to proving

\begin{theo}\label{bhatt}
For every semistable pair $(U, \calu)$ over $ K$ there is an $h$-covering $h:\, (U', \calu')\to (U, \calu)$ of semistable pairs such that the induced maps $$h^* : H^b(\calu, \Omega^a_{U, \calu})\to H^b(\calu', \Omega^a_{U', \calu'})$$ factor through the multiplication-by-$p$ map
$$H^b(\calu, \Omega^a_{U, \calu})\stackrel p\to H^b(\calu, \Omega^a_{U, \calu})$$
for all $(a,b)\neq (0,0)$.
\end{theo}

We sketch the proof in the case where $U$ is proper and $a=0$; this was proven by Bhargav Bhatt in his paper \cite{bhatt}. It turns out that in this case the map $h$ can be chosen to be proper and surjective. The general proof follows a similar pattern but the technical details are a bit more complicated; see the original paper \cite{Bei} of Beilinson or Illusie's survey \cite{illsurvey}.

The key lemma is the following.

\begin{lem}
Let $X$ be a proper curve over a field. There exists a proper smooth curve $Y$ with geometrically connected components defined over a finite extension of $k$ and a proper surjection  $h:\, Y\to X$ such that the induced map $h^*:\, \Pic\, X\to\Pic\, Y$ factors through the multiplication-by-$p$ map $\Pic\, X\to\Pic\, X$.
\end{lem}

\begin{dem}
We are allowed to take finite covers of $X$ and work with one component at a time, so after extending the base field and normalizing $X$ in a finite extension of its function field we may assume $X$ is smooth connected of positive genus and has a $k$-point $O$. The Abel-Jacobi map $P\mapsto [P-O]$ defines a closed immersion $X\to\Pic^0 X\subset \Pic\, X$. Define $\tilde Y$ by the fibre square
$$
\begin{CD}
\tilde Y @>>> \Pic^0 X \\
@VVV @VVpV \\
X @>>> \Pic^0 X
\end{CD}
$$
and take $ Y$ to be the normalization of $\tilde Y$. This defines $h:\, Y\to X$. The map $\tilde Y\to\Pic^0 X$ induces a map $Y\to\Pic^0 X$ and factors through $\Pic^0 Y$ by the universal property of the Jacobian. We thus obtain a commutative diagram
$$
\begin{CD}
Y @>>> \Pic^0 Y \\
@V{\id}VV @VVV \\
Y @>>> \Pic^0 X \\
@VhVV @VVpV \\
X @>>> \Pic^0 X
\end{CD}
$$
The composite map $\Pic^0 Y\to \Pic^0 X$ on the right hand side is the map induced by the composition $Y\to X\to \Pic^0 X$. By autoduality of the Jacobian, the map $\Pic^0 X\to \Pic^0 Y$ on dual abelian varieties is the pullback induced by $h$. By construction it factors through the multiplication-by-$p$ map of $\Pic^0 X$.
\end{dem}

\begin{cor}
Let $X$ be a proper curve over a field having a rational point $O$. There exists $h:\, Y\to X$ as above such that the induced map $h^*:\, H^1(X, \calo_X)\to H^1(Y, \calo_Y)$ factors through the multiplication-by-$p$ map on $H^1(X, \calo_X)$.
\end{cor}

\begin{dem}
Identify $H^1(X, \calo_X)$ with the tangent space at $0$ of $\Pic^0 X$.
\end{dem}

When $X$ is semi-stable, then $\Pic^0 X$ is a semi-abelian variety. We can then use properties of semi-abelian varieties to establish a relative version of the corollary.

\begin{prop}
Assume $X\to T$ is a projective semi-stable relative curve with $T$ integral and excellent. There exists a pullback diagram
$$
\begin{CD}
X' @>{\pi}>> X \\
@V{\phi'}VV @VV{\phi}V \\
T' @>{\psi}>> T
\end{CD}
$$
where $\psi':\, T'\to T$ is an alteration,  the base change curve $X'\to T'$ is projective semistable and the pullback map $\psi^*R^1\phi_*\calo_X\to R^1\phi'_*\calo_{X'}$ is divisible by $p$ in  $\Hom(\psi^*R^1\phi_*\calo_X, R^1\phi'_*\calo_{X'})$.
\end{prop}

Recall that a diagram as above always defines base change morphisms $\psi^*R^q\phi_*\calf\to R^q\phi'_*(\pi^*\calf)$ for a sheaf $\calf$ on $X$. We apply this with $\calf=\calo_X$ and compose with the morphism $\pi^*\calo_X\to\calo_{X'}$ induced by $\pi$.\medskip

\begin{dem}
Let $\eta$ be the generic point of $T$. By the lemma we find a finite map $\eta'\to \eta$ and a proper smooth curve $Y_{\eta'}\to X_{\eta}$ such that the induced map $\Pic^0 Y_{\eta'}\to \Pic^0 X_{\eta}$ factors through multiplication by $p$. By a result of de Jong \cite{dejong}, after replacing $T$ by an alteration $\widetilde T\to T$  and base changing $X$ we may extend $Y_{\eta'}$ to a semistable curve $Y\to X$. (If $Y$ has several components, we do this componentwise.) Now $\Pic^0(X/T)$ and $\Pic^0(Y/T)$ are semi-abelian schemes. By a basic result on semi-abelian schemes $G$ over a normal base (\cite{fc}, I 2.7), the restriction functor $G\to G_\eta$ to the generic point is fully faithful. Thus, since we know that the restriction of $\Pic^0(X/T)\to \Pic^0(Y/T)$ to the generic point factors through multiplication by $p$, the same is true for the map itself. Finally, we deduce the result on $R^1\phi_*\calo_X$ by passing to the normal bundle of the zero section as in the previous corollary.
\end{dem}

\noindent{\em Proof of Theorem \ref{bhatt} for $a=0$ and $U$ proper.\/} Since $\Spec\calo_K$ is affine, by the Serre vanishing theorem it will suffice to prove a relative result: there exists an alteration $\alpha:\, \widetilde U\to U$ such that $R^bf_*\calo_{U}\to R^b(f\circ\alpha)_*\calo_{U'}$ is divisible by $p$ in $\Hom(R^bf_*\calo_{U}, R^b(f\circ\alpha)_*\calo_{\widetilde U})$, where $f$ is the structure map $U\to\Spec\calo_K$.

We use induction on the relative dimension $d$ of $f:\, U\to\Spec\calo_K$. The case of dimension 0 is easy using Kummer theory. By another result of de Jong, after replacing $U$ by an alteration, we find a factorization $U\to T\to \Spec \calo_K$ such that $T$ is integral, $\phi:\, U\to T$ is a projective semi-stable relative curve having a section $s:\, T\to U$, and $f':\, T\to\Spec \calo_K$ is proper surjective of relative dimension $d-1$. Since $\phi$ is of relative dimension 1, the Leray spectral sequence
$$
R^pf'_*(R^q\phi_*\calo_U) \Rightarrow R^{p+q}f_*\calo_U
$$
yields an exact sequence
$$
0\to R^bf'_*(\phi_*\calo_U)\to R^bf_*(\calo_U)\to R^{b-1}f'_*(R^1\phi_*\calo_U)\to 0.
$$
As $\phi$ has connected fibres, we have $\phi_*\calo_U=\calo_T$, so the exact sequence becomes
$$
0\to R^bf'_*(\calo_T)\to R^bf_*(\calo_U)\to R^{b-1}f'_*(R^1\phi_*\calo_U)\to 0,
$$
and the section $s:\, U\to T$ induces a splitting. By induction we find an alteration $\pi':\, T'\to T$ such that $\pi'^*R^bf'_*\calo_T\to R^bg'_*\calo_{T'}$ is divisible by $p$ in the Hom-group, where $g'=f'\circ\pi'$. Denote by $\phi:\, U'\to T'$ the base change curve. By the previous proposition, we find a further alteration $\pi'':\, T''\to T'$ giving rise to a commutative diagram
$$
\begin{CD}
U'' @>>> U' @>>> U \\
@VV{\phi''}V @VV{\phi'}V @VV{\phi}V \\
T'' @>{\pi''}>> T' @>{\pi'}>> T
\end{CD}
$$
such that $\pi''^*R^1\phi'_*\calo_{U'}\to R^1\phi''_*\calo_{U''}$ is also divisible by $p$ in the Hom-group. We conclude by putting these results together using the above split exact sequence.
\enddem

\renewcommand{\thesection} {A}
\renewcommand{\theprop} {A.\arabic{prop}}
\renewcommand{\thedefi} {A.\arabic{defi}}
\renewcommand{\theex} {A.\arabic{ex}}
\renewcommand{\theexs} {A.\arabic{exs}}
\renewcommand{\therema} {A.\arabic{rema}}
\renewcommand{\theremas} {A.\arabic{remas}}
\renewcommand{\thecons} {A.\arabic{cons}}
\renewcommand{\thefacts} {A.\arabic{facts}}
\renewcommand{\theconj} {A.\arabic{conj}}
\renewcommand{\thetheo} {A.\arabic{theo}}
\renewcommand{\thecor} {A.\arabic{cor}}
\renewcommand{\thelem} {A.\arabic{lem}}
\renewcommand{\thesubsection} {A.\arabic{subsection}}

\section{Appendix: Methods from simplicial algebra}

In this appendix we summarize some basics from simplicial algebra needed for the study of cotangent complexes and derived de Rham algebras. For the first three subsections our main reference is Chapter 8 of Weibel's book \cite{weibel}.

\subsection{Simplicial methods}\label{simpmeth}

Denote by $\Delta$ the category whose objects are the finite ordered sets $[n]=\{0<1\dots<n\}$ for each integer $n\geq 0$, and the morphisms are nondecreasing functions.

\begin{defi}\rm
A {\em simplicial} (resp. {\em cosimplicial\/}) {\em object} in a category $\mathcal{C}$ is a contravariant (resp. covariant) functor $X:\,\Delta\to\mathcal{C}$.
\end{defi}

Simplicial (resp. cosimplicial) objects in a category $\mathcal C$ form a category ${\rm Simp}({\mathcal C})$ (resp. ${\rm cosimp}({\mathcal C})$) whose morphisms are morphisms of functors.

Fix an integer $n\geq 1$. For each $0\leq i\leq n$ we define a {\em face map} $\varepsilon_i\colon [n-1]\to [n]$ as the unique nondecreasing map whose image does not contain $i$. In the other direction, we define for each $0\leq i\leq n$ a {\em degeneracy map} $\eta_i\colon [n]\to [n-1]$ as the unique nondecreasing map that is surjective and has exactly two elements mapping to $i$.

\begin{lem}\label{simpob}
Giving a simplicial object $X$ in a category $\mathcal{C}$ is equivalent to giving an object $X_n$ for each $n\geq 0$ together with face operators  $\partial_i=X(\varepsilon_i)\colon X_n\to X_{n-1}$ and degeneracy operators $\sigma_i=X(\eta_i)\colon X_n\to X_{n+1}$ for $0\leq i\leq n$ satisfying the identities
\begin{eqnarray}
\partial_i \partial_j&=& \partial_{j-1}\partial_i \quad i<j\ , \notag\\
\sigma_i \sigma_j &=& \sigma_{j+1}\sigma_i \quad i\leq j\label{simplident}\\
\partial_i\sigma_j &=&\begin{cases} \sigma_{j-1}\partial_i &i<j\\ \id&i=j\text{ or }i=j+1\\  \sigma_j \partial_{i-1}&i>j+1\ .\end{cases}\notag
\end{eqnarray}
\end{lem}

\begin{dem}
See \cite{weibel}, Proposition 8.1.3.
\end{dem}

\begin{ex}\label{constsimp}\rm
If $B$ is an object in the category $\mathcal{C}$, we define the \emph{constant simplicial object} $B_\bullet$ associated with $B$ by setting $B_n:=B$ for all $n$, and declaring all face and degeneracy maps to be identity maps of $B$.
\end{ex}

\begin{ex}\label{deltan}\rm
Fix an integer $n\geq 0$. Setting $\Delta[n]_m:=\Hom_\Delta([m],[n])$ defines a simplicial set $\Delta[n]_\bullet$, i.e. a simplicial object in the category of sets. Here the simplicial structure is induced by contravariance of the Hom-functor. Moreover, $[n]\to \Delta[n]_\bullet$ is a covariant functor from $\Delta$ to the category of simplicial sets.
\end{ex}

We also need the notion of augmented simplicial objects.

\begin{defi}\label{augm}\rm Given an object $B$ and a simplicial object  $X_\bullet$ in a category $\mathcal{C}$, we define an \emph{augmentation} $\varepsilon\colon X_\bullet\to B$ to be a morphism $X_\bullet\to B_\bullet$.
\end{defi}

\begin{lem}\label{augment} Let $X_\bullet$ be a simplicial object.
Defining an augmentation $\epsilon_\bullet\colon X_\bullet\to B$ is equivalent to giving a morphism $\epsilon_0\colon X_0\to B$ satisfying $\epsilon_0 \partial_0=\epsilon_0 \partial_1$.
\end{lem}

\begin{proof}
Given a map $\epsilon_\bullet:\, X_\bullet\to B_\bullet$ of simplicial objects, the degree 0 component $\epsilon_0$ satisfies this identity by definition. Conversely, given $\epsilon_0$ as in the statement, we may choose an arbitrary morphism $\alpha\colon [0]\to [n]$ and set $\epsilon_n:=\epsilon_0\circ X(\alpha)$. This does not depend on the choice of $\alpha$, because for a different choice $\beta\colon [0]\to [n]$ we may find a morphism $\gamma\colon [1]\to [n]$ such that both $\alpha$ and $\beta$ factor through $\gamma$, from which the identity $\epsilon_0 \partial_0=\epsilon_0 \partial_1$ implies that the resulting maps $\epsilon_n$ are the same.  The reader will check that the sequence $\epsilon_n$ indeed defines an augmentation.
\end{proof}

Next we define simplicial homotopies. To do so, we first need an auxiliary construction.

\begin{cons}\rm Let $\mathcal{C}$ be a category in which finite coproducts exist. Assume given a simplicial object $X_\bullet$ in $\mathcal{C}$ and a simplicial object $U_\bullet$ in the category of nonempty finite sets. We define the product $X_\bullet\times U_\bullet$ as a simplicial object in $\mathcal{C}$ with terms given by
\begin{equation*}
(X\times U)_n:=\coprod_{u\in U_n}X_n
\end{equation*}
and the simplicial structure defined as follows: for $\gamma\colon [m]\to [n]$ the morphism $(X\times U)(\gamma)$ maps the component $X_n$ indexed by $u\in U_n$ to the component $X_m$ indexed by $U(\gamma)(u)\in U_m$ via the morphism $X(\gamma)$.
\end{cons}

In particular, it makes sense to speak about the product $X_\bullet\times \Delta[n]_\bullet$ for each $n\geq 0$. Note that by functoriality the two morphisms $\epsilon_0,\epsilon_1:\, [0]\to[1]$ induce morphisms $e_i:\,X_\bullet\times \Delta[0]_\bullet\to X_\bullet\times \Delta[1]_\bullet$ of simplicial objects in $\mathcal C$ for $i=0,1$. Here we may identify $X_\bullet\times \Delta[0]_\bullet$ with $X_\bullet$ since by definition $\Delta[0]_\bullet$ is the constant simplicial object associated with the one-point set.

\begin{defi}\rm Assume $\mathcal C$ has finite coproducts, and consider two morphisms ${f_\bullet,g_\bullet\colon X_\bullet\to Y_\bullet}$ between simplicial objects of $\mathcal{C}$.  A \emph{simplicial homotopy} from $f_\bullet$ to $g_\bullet$ is a morphism ${h_\bullet\colon X_\bullet\times \Delta[1]_\bullet\to Y_\bullet}$ satisfying $f_\bullet=h_\bullet\circ e_0$ and $g_\bullet=h_\bullet\circ e_1$, where $e_0, e_1$ are the maps defined above.

We say that $f_\bullet$ and $g_\bullet$ are {\em homotopic} if they are in the same class of the equivalence relation on maps of simplicial objects generated by simplicial homotopies. Given a morphism $f_\bullet\colon X_\bullet\to Y_\bullet$ of simplicial objects, a {\em homotopy inverse} of $f$ is a morphism $g_\bullet\colon Y_\bullet\to X_\bullet$ such that $f_\bullet\circ g_\bullet$ (resp.\ $g_\bullet\circ f_\bullet$) is homotopic to the identity map of $Y_\bullet$ (resp. $X_\bullet$). If such $f$ and $g$ exist, we say that $X_\bullet$ and $Y_\bullet$ are \emph{homotopy equivalent}.

\end{defi}

\begin{rema}\label{abelcathomotequiv}\rm
If $\mathcal{A}$ is an abelian category, then the existence of a simplicial homotopy between two simplicial maps $f_\bullet$ to $g_\bullet$ in $\mathcal A$ is already an equivalence relation.
See \cite{weibel}, Exercise 8.3.6.
\end{rema}

\subsection{Associated chain complexes}

We now investigate simplicial objects in abelian categories.

\begin{defi}\rm
Given a simplicial object $X_\bullet$ in an abelian category $\mathcal A$, we define its \emph{associated (unnormalized) chain complex} as the complex ${CX}_\bullet$ with ${CX}_n:=X_n$ in degree $n$ and with differential $d_n\colon X_n\to X_{n-1}$ defined by $$d_n:=\sum_{i=0}^n(-1)^i\partial_i.$$

This is indeed a chain complex by the first identity in Lemma \ref{simpob}. The \emph{normalized chain complex} of $X_\bullet$ is the chain complex ${NX}_\bullet$ with $$NX_n:=\bigcap_{0\leq i\leq n-1}\Ker(\partial_i),$$
where $\partial_i\colon X_n\to X_{n-1}$ is the $i$-th face map. The differential $NX_n\to NX_{n-1}$ is defined to be $(-1)^n\partial_n.$ The {\em homotopy groups} of $X_\bullet$ are given  by
$$
\pi_n(X_\bullet):=H_n(NX_\bullet).
$$\end{defi}

\begin{rema}\label{remnc}\rm  By (\cite{weibel}, Theorem 8.3.8), the natural inclusion $NX_\bullet\to CX_\bullet$ is a quasi-isomorphism. Therefore we also have $
\pi_n(X_\bullet)=H_n(CX_\bullet).
$
\end{rema}

The main theorem concerning the normalized chain complex is now the following.

\begin{theo}[Dold--Kan correspondence]\label{DoldKan} Let $\mathcal A$ be an abelian category.
The functor $N$ induces an equivalence of categories between the category of simplicial objects in $\mathcal{A}$ and that of nonnegatively graded homological chain complexes in $\mathcal{A}$.

Under this equivalence simplicial homotopies between simplicial maps correspond to chain homotopies on the associated normalized complexes.
\end{theo}

\begin{dem}
See \cite{weibel}, Theorem 8.4.1.
\end{dem}

The quasi-inverse to the functor $N$ in the Dold--Kan correspondence is given by the {\em Kan transform} $KC_\bullet$ of a nonnegatively graded chain complex $C_\bullet$ in $\mathcal{A}$.  It is the simplicial object whose degree $n$ term is defined by
\begin{equation*}
KC_n:=\bigoplus_{p\leq n}\bigoplus_{\eta:[n]\to [p]\text{ surjective}}C_{p}\
\end{equation*}
and whose maps are defined as follows.
For a morphism $\alpha\colon [m]\to[n]$ of simplices and a surjective morphism $\eta\colon [n]\to [p]$ we may write the composite uniquely in the form $\eta\circ\alpha=\varepsilon'\circ\eta'$ where $\varepsilon'$ is injective and $\eta'$ is surjective. We define the morphism $KC(\alpha)\colon KC_n\to KC_m$ on the direct summand $C_p$ of $KC_n$ by
\begin{equation*}
KC_\alpha\mid_{C_p}:=\begin{cases}
\id_{C_p} &\text{if }\varepsilon'=\id_{[p]}\\
d\colon C_{p}\to C_{p-1}&\text{if }\varepsilon'=\varepsilon_0\\
0&\text{otherwise,}
\end{cases}
\end{equation*}
where $\varepsilon_0\colon [p-1]\to [p]$ denotes the unique injective morphism of simplices whose image avoids $0$.  \smallskip

\subsection{Bisimplicial objects}

We now turn to bisimplicial constructions.

\begin{defi}\rm
A bisimplicial object $X_{\bullet\bullet}$ in a category $\mathcal{C}$ is a simplicial object in the category of simplicial objects in $\mathcal{C}$. \end{defi}

We may regard $X_{\bullet\bullet}$ as a contravariant functor from $\Delta\times\Delta$ to $\mathcal{C}$. We have horizontal (resp.\ vertical) face maps $\partial_i^h\colon X_{pq}\to X_{p-1,q}$ (resp.\ $\partial_i^v\colon X_{pq}\to X_{p,q-1}$) and degeneracy maps $\sigma_i^h\colon X_{pq}\to X_{p+1,q}$ (resp.\ $\sigma_i^v\colon X_{pq}\to X_{p,q+1}$). These satisfy the simplicial identities horizontally and vertically, and horizontal operators commute with each vertical operators.

\begin{defi}\rm The \emph{diagonal} $X_{\bullet}^\Delta$ of a bisimplicial object $X_{\bullet\bullet}$ is the simplicial object obtained by composing the diagonal functor $\Delta\to \Delta\times\Delta$ with the functor $X$.
\end{defi}

Thus $X_n^\Delta=X_{nn}$ and the face (resp.\ degeneracy) operators are given by $\partial_i^\Delta=\partial_i^h\partial_i^v=\partial_i^v\partial_i^h$ (resp. $\sigma_i=\sigma_i^h\sigma_i^v=\sigma_i^v\sigma_i^h$).\smallskip

\begin{cons}\rm
We define the (unnormalized) first quadrant double complex $CX_{\bullet\bullet}$ associated with a bisimplicial object $X_{\bullet\bullet}$ in an abelian category $\mathcal A$ as follows. The horizontal differentials in the double complex are those of the chain complex coming from the horizontal face maps. The vertical differentials are those of the chain complex coming from the vertical face maps, multiplied by a factor $(-1)^p$ for a differential starting from $X_{pq}$.

\end{cons}

\begin{theo}[Eilenberg-Zilber]\label{eilenbergzilber}
Let $X_{\bullet\bullet}$ be a bisimplicial object in an abelian category $\mathcal{A}$. For all $n\geq 0$ there are natural isomorphisms $$\pi_n(X_\bullet^\Delta)\cong H_n \operatorname{Tot}(CX_{\bullet\bullet}),$$
where $\operatorname{Tot}(C_{\bullet\bullet})$ denotes the total complex associated with a double complex $C_{\bullet\bullet}$ (with the direct sum convention).
\end{theo}

\begin{proof}
See \cite{weibel}, Theorem\ 8.5.1.
\end{proof}

\subsection{Simplicial resolutions} Our definition for a simplicial resolution is as follows.

\begin{defi}\label{simpres}\rm
An augmented simplicial object $\epsilon:\,X_\bullet\to B$ in an abelian category is a {\em simplicial resolution} of $B$ if $\epsilon_0:\,X_0\to B$ is surjective and the associated chain complex of $X_\bullet$ is acyclic except in degree 0 where its homology is $B$.
\end{defi}

 As the associated chain complex of the constant simplicial object $B_\bullet$ is acyclic in positive degrees,  the augmentation map in a simplicial resolution $X_\bullet\to B$ induces a quasi-isomorphism $CX_\bullet\stackrel\sim\to CB_\bullet$.

To elucidate the homotopical nature of this definition in the case of abelian groups, we need the following notion.

\begin{defi}\rm
A morphism $X_\bullet\to Y_\bullet$ of simplicial sets is a \emph{trivial (Kan) fibration} if in every commutative solid diagram
$$\xymatrix{
Z_\bullet\ar[r]^{a_\bullet}\ar[d]& X_\bullet\ar[d]^{f_\bullet}\\
W_\bullet\ar[r]_{b_\bullet}\ar@{.>}[ru]& Y_\bullet
}$$
of simplicial sets such that for all $n\geq 0$ the maps $Z_n\to W_n$ are injective a dotted arrow exists making the diagram commutative.
\end{defi}

\begin{rema}\rm
In fact, it is enough to require the right lifting property of the above definition in the special case of the inclusions $\partial\Delta[n]_\bullet\to \Delta[n]_\bullet$, where $\partial\Delta[n]_\bullet$ is the {\em boundary} of the simplicial set $\Delta[n]_\bullet$ of Example \ref{deltan}. See (\cite{HA}, \S 2.2, Proposition 1) or  (\cite{stacks}, Tag 08NK, Lemma 14.30.2).
\end{rema}

\begin{prop}\label{kanhomotopy}
A trivial fibration $X_\bullet\to Y_\bullet$ of simplicial sets has a homotopy inverse.
\end{prop}

\begin{proof}
See (\cite{stacks}, Tag 08NK, Lemma 14.30.8).
\end{proof}

This being said, we have:

\begin{prop}\label{kanabelian}
An augmented simplicial object $\epsilon:\,X_\bullet\to B$ in the category of abelian groups is a simplicial resolution if and only if the underlying morphism of simplicial sets is a trivial fibration.
\end{prop}

\begin{proof}
See (\cite{HA}, \S 2.3, Proposition 2) or  (\cite{stacks}, Tag 08NK, Lemmas 14.31.8 and 14.31.9).
\end{proof}

In particular, a simplicial resolution of abelian groups induces a homotopy equivalence of underlying simplicial sets (but not necessarily of simplicial abelian groups!).

\subsection{Derived functors of non-additive functors}

Simplicial methods may also be used to construct derived functors of not necessarily additive functors between abelian categories, following Dold and Puppe \cite{DP}.

Let $\mathcal{A}$ be an abelian category with enough projectives, and $F:\, {\mathcal A}\to\mathcal B$ a functor to another abelian category $\mathcal B$. For an object $A\in\mathcal A$ consider a projective resolution $P_\bullet\to A$. By the Dold--Kan correspondence (Theorem \ref{DoldKan}) the Kan transform $KP_\bullet\to A$ is a simplicial resolution of $A$ with projective terms. Set
\begin{equation}\label{ld}
L^iF(A):=H_i(NF(KP_\bullet)).
\end{equation}

\begin{lem}
The above definition does not depend on the choice of the projective resolution $P_\bullet$.
\end{lem}

\begin{dem}
Any two projective resolutions $P_\bullet$, $Q_\bullet$ of $A$ are chain homotopy equivalent. Therefore $KP_\bullet$ and $KQ_\bullet$ are simplicially homotopy equivalent by the Dold--Kan correspondence. As simplicial homotopies are preserved by arbitrary functors on simplicial objects, so are $F(KP_\bullet)$ and $F(KQ_\bullet)$. Reading the Dold--Kan correspondence backwards we see that $NF(KP_\bullet)$ and $NF(KQ_\bullet)$ are quasi-isomorphic.
\end{dem}

\begin{defi}\rm
We define the $i$-th left derived functor $L^iF$ of $F$ by means of formula (\ref{ld}) above. Similarly, we define right derived functors $R^iF$ for functors on abelian categories having enough injectives.
\end{defi}

\begin{remas}\label{derfunctrema}\rm ${}$

\noindent 1. In the case of an additive functor $F$ we recover the standard definition of derived functors as additive functors commute with the Kan transform.\smallskip

\noindent 2. More generally, we may define total left derived functors $LF:\, D^-({\mathcal A})\to D^-({\mathcal B})$ for non-additive $F$ and similarly for right derived functors. Instead of a projective resolution we start with a bounded above complex with projective terms representing an object in $D^-({\mathcal A})$, and then apply the functor $NFK$.

\end{remas}

\subsection{Application: derived exterior powers and divided powers} Important examples of non-additive functors are given by the exterior power functors $M\mapsto \wedge^nM$ on the category of modules over a commutative ring $A$. Another example is given by divided powers, as we now recall.

\begin{defi}\label{divpower}\rm
Let $A$ be a commutative ring, $M$ an $A$-module and $B$ an $A$-algebra. A {\em divided power structure} on $B$ by $M$ is given by a sequence of maps $\gamma_n\colon M\to B$  for each $n\geq 0$ satisfying
\begin{enumerate}
\item $\gamma_0(m)=1$
\item $\gamma_s(m)\gamma_t(m)=\binom{s+t}{s}\gamma_{s+t}(m)$
\item $\gamma_n(m+m')=\sum_{s+t=n}\gamma_s(m)\gamma_t(m')$
\item $\gamma_n(\lambda m)=\lambda^n\gamma_n(m)$
\end{enumerate}
for all $m\in M$, $\lambda\in A$, and $s,t,n\geq 0$.
\end{defi}

Note that if $n$ is such that $n!$ is invertible in $B$, the second relation forces $\gamma_n(m)={\gamma_1(m)^n}/{n!}$, whence the term `divided power structure'.

\begin{lem}\label{freepd}
Fix $A$ and $M$. The set-valued functor sending an $A$-algebra $B$ to the set of its divided power structures by $M$ is representable by an $A$-algebra $\Gamma^\bullet_A(M)$.
\end{lem}

\begin{proof}
One constructs $\Gamma^\bullet_A(M)$ by taking the free $A$-algebra $A[\gamma_n(m)]$ on generators $\gamma_n(m)$ for all $n\geq 0$ and $m\in M$, and then taking the quotient by the above four relations.\end{proof}

Observe that $A[\gamma_n(m)]$ has a natural graded algebra structure in which the $\gamma_n(m)$  for fixed $n$ generate the degree $n$ component. As the relations are homogeneous, this induces a grading on $\Gamma^\bullet_A(M)$ whose degree $n$ component we denote by $\Gamma^n_A(M)$. We shall drop the subscript $A$ when clear from the context.

The functors $M\mapsto\Gamma^n(M)$ are also non-additive functors on the category of $A$-modules. Their derived functors are related to those of the exterior product functors by the following identity.

\begin{prop}[Quillen's shift formula]\label{quillen}
Let $E$ be a bounded above complex of $A$-modules. In the associated derived category we have isomorphisms $$L\wedge^n(E[1])\cong (L\Gamma^n (E))[n]$$ for all $n\geq 0$.
\end{prop}

We give a proof for the sake of completeness, based on Quillen's ideas sketched in \cite{quillennotes} and \cite{I1}.

The proof uses an auxiliary construction. Assume given a sequence $E\overset{u}{\to} F\overset{v}{\to} G$ of $A$-module homomorphisms with $v\circ u=0$. To these data we associate a complex of $A$-modules
\begin{equation}
0\to \Gamma^n (E)\to \Gamma^n (F)\to \Gamma^{n-1}(F)\otimes G\to\dots\to \Gamma^1(F)\otimes\wedge^{n-1}G\to\wedge^n G\to 0\label{degreenkoszul}
\end{equation}
for all $n\geq 0$ as follows. The differential $\Gamma^n(E)\to\Gamma^n(F)$ is induced by $u$. For $i\geq 0$ the differentials $d:\,\Gamma^{n-i}(F)\otimes \wedge^iG\to \Gamma^{n-i-1}(F)\otimes \wedge^{i+1} G$ are defined by setting $$d(\gamma_{n-i}(x)\otimes 1):=\gamma_{n-i-1}(x)\otimes v(x),\quad d(1\otimes y):=0$$ for $x\in F$ and $y\in G$ and extending by linearity. We obtain a complex in view of the relations $v\circ u=v\wedge v=0$.

\begin{lem}\label{shortexactkoszul} Assume given a short exact sequence
\begin{equation}\label{sex}
0\to E\to F\to G\to 0
\end{equation}
of flat $A$-modules. The associated complexes (\ref{degreenkoszul}) are exact for all $n\geq 0$.
\end{lem}

\begin{proof} To begin with, the lemma holds in the special cases
\begin{align}
& 0\to A\to A\to 0\to 0,\label{aa1}\\
& 0\to 0\to A\to A\to 0.\label{aa2}
\end{align}
In the first case the complexes (\ref{degreenkoszul}) reduce to the isomorphisms $\Gamma^n(A)\cong \Gamma^n(A)$ and in the second case one has to check that the maps $\Gamma^n(A){\to} \Gamma^{n-1}(A)\otimes A$ induced by $\gamma_n(a)\mapsto \gamma_{n-1}\otimes a$ are isomorphisms. As these are nonzero maps of free $A$-modules of rank 1, the statement follows.

Next one deals with the case where $E$, $F$ and $G$ are finitely generated and free over $A$. In this case the short exact sequence (\ref{sex}) splits, and therefore we may write it as a finite direct sum of short exact sequences of the form (\ref{aa1}) and (\ref{aa2}). Starting from these special cases, one proves the proposition by induction on the sum of the ranks by checking that the complex (\ref{degreenkoszul}) associated with a direct sum of two short exact sequences is a direct sum of tensor products of complexes of type (\ref{degreenkoszul}) associated with the individual short exact sequences. The lemma then follows by the K\"unneth formula for complexes of free modules. Finally, the general case follows by writing a short exact sequence of flat modules as a direct limit of sequences of finitely generated free modules.
\end{proof}

\noindent{\em Proof of Proposition \ref{quillen}.} Replacing $E$ by a quasi-isomorphic complex of free modules, we may assume that $E$ has free terms. Consider the short exact sequence
$$0\to E\to C(E)\to E[1]\to 0$$
of complexes, where $C(E)$ is the cone of the identity map of $E$. By taking Kan transforms we obtain a short exact sequence  $$0\to KE\to KC(E)\to K(E[1])\to 0$$ of simplicial $A$-modules with free terms. Applying Lemma \ref{shortexactkoszul} in each degree and taking associated normalized complexes, we obtain an exact sequence
\begin{align}
0\to N\Gamma^n KE\to N\Gamma^n KC(E)\to N(\Gamma^{n-1}KC(E)\otimes K(E[1]))\to\dots\notag\\
\dots\to N(\Gamma^1KC(E)\otimes\wedge^{n-1}K(E[1]))\to N\wedge^n K(E[1])\to 0\label{normalkandegreenkoszul}
\end{align}
of complexes of $A$-modules. Since $C(E)$ is an acyclic complex with free terms, we may view $KC(E)$ as a free simplicial resolution of the zero module. Thus by definition $N\Gamma^nKC(E)=L^n\Gamma(0)=0$. Applying Lemma \ref{flattensorquasi} with $E_\bullet=\Gamma^pKC(E)$, $F_\bullet=0$, and $L_\bullet=\wedge^{n-p}K(E[1])$ we obtain that all complexes in the middle of \eqref{normalkandegreenkoszul} are acyclic. It remains to note that by definition $L\Gamma^n E= N\Gamma^n KE$ and $L\wedge^n(E[1])=N\wedge^n K(E[1])$.
\enddem

We finish this subsection by constructing of a filtration on higher derived functors of the exterior product attached to short exact sequences of modules.

\begin{cons}\label{lwedgefilt}\rm Given an exact sequence
$$0\to M'\to M\to M''\to 0$$
of flat modules over a ring $A$, define an increasing filtration $I_a\wedge^i(M)$ on $\wedge^i(M)$ by setting
$$
I_a\wedge^i M:=\mathrm{Im}\left(\wedge^{i-a} M'\otimes \wedge^a M\to \wedge^i M\right)\ .
$$
We then have a natural map
$$
\wedge^{i-a} M'\otimes \wedge^a M''\to {\rm gr}^I_a\wedge^i M.
$$
In case the exact sequence splits, this map is an isomorphism and the wedge product decomposes as a direct sum
$$
\bigoplus_a\wedge^{i-a} M'\otimes \wedge^a M'' \cong\wedge^i M.
$$
We construct a derived version of this filtration as follows. Choose a projective resolution $P'_\bullet$ (resp.\ $P''_\bullet$) of $M'$ (resp.\ of $M''$). By the Horseshoe Lemma (\cite{weibel}, Proposition 2.2.8) there is a projective resolution $P_\bullet$ of $M$ fitting in a short exact sequence
$$ 0\to P'_\bullet\to P_\bullet\to P''_\bullet\to 0$$
of complexes. Applying the Kan transform gives a short exact sequence
$$ 0\to KP'_\bullet\to KP_\bullet\to KP''_\bullet\to 0$$
of simplicial $A$-modules. So for each $n$ we have a map $\wedge^{i-a} KP'_n\otimes \wedge^a KP_n\to \wedge^i KP_n$ giving rise to a map
$$\wedge^{i-a} KP'_\bullet\otimes \wedge^a KP_\bullet\to \wedge^i KP_\bullet$$
of simplicial $A$-modules. Passing to the normalized chain complex yields a filtration
$$
I_aL\wedge^i M :=\mathrm{Im}\left(L\wedge^{i-a} M' \otimes^{\bf L} L\wedge^a M \to L\wedge^i M \right)
$$
with analogous splitting properties since by definition $L\wedge^i M$ is represented by the chain complex $N\wedge^i KP_\bullet$ in the derived category of $A$-modules. A similar construction holds for sheaves of modules.
\end{cons}

\subsection{Cohomological descent}

In this subsection and the next, we give a utilitarian summary of the results of cohomological descent we need. The basic reference is \cite{sga4} but the notes of Conrad \cite{conraddescent} and Laszlo \cite{laszlo} are much more readable. There is also a brief summary in \cite{Hodge3}.

Let $X_\bullet$ be a simplicial object in the category of topological spaces, or the category of schemes equipped with  a Grothendieck topology. A {\em simplicial abelian sheaf} on $X_\bullet$ is given by an abelian sheaf $\calf^n$ on each $X_n$ together with morphisms $[\phi]:\,X(\phi)^*\calf^n\to\calf^m$ for each $\phi:\, [n]\to [m]$ in $\Delta$ subject to the compatibility conditions $[\phi]\circ X(\phi)^*[\psi]=[\phi\circ\psi]$ for all composable pairs $\phi, \psi$ of morphisms in $\Delta$. (Recall that $X(\phi)$ is the morphism $X_m\to X_n$ induced by $\phi:\, [n]\to [m]$.)

Given an augmented simplicial space $\varepsilon:\,X_\bullet\to S$, we have a natural pullback functor $\varepsilon^*$ from the category of sheaves on $S$ to simplicial sheaves on $X_\bullet$ induced by termwise pullback via the morphism of simplicial spaces $X_\bullet\to S_\bullet$ corresponding to $\varepsilon$. The functor $\varepsilon^*$ has a right adjoint $\varepsilon_*$ sending a simplicial sheaf $\calf^\bullet$  on $X_\bullet$ to $\ker((\sigma_0-\sigma_1):\, \varepsilon_{0*}\calf^0\to \varepsilon_{1*}\calf^1)$. The functor $\varepsilon^*$ is exact and $\varepsilon_*$ is left exact, giving rise to a total derived functor $R\varepsilon_*$.

\begin{defi}\rm The augmented simplicial space  $\varepsilon:\,X_\bullet\to S$ satisfies cohomological descent if the adjunction map $\id\to R\varepsilon_*\circ\varepsilon^*$ is an isomorphism.\end{defi}

Define the functor $\Gamma(X_\bullet, \cdot)$ by sending a simplicial abelian sheaf $\calf^\bullet$ on $X_\bullet$ to $\Gamma(S, \varepsilon_*\calf)$. The adjunction map induces a morphism
$$
R\Gamma(S, \calf)\to R\Gamma(S, R\varepsilon_*\varepsilon^*\calf)\stackrel\sim\to R(\Gamma(S, \cdot)\circ\varepsilon_*)(\varepsilon^*\calf)=R\Gamma(X_\bullet, \varepsilon^*\calf)
$$
for an abelian sheaf $\calf$ on $S$. If cohomological descent holds, the first map is also an isomorphism, and we obtain an isomorphism
$$
R\Gamma(S, \calf)\stackrel\sim\to R\Gamma(X_\bullet, \varepsilon^*\calf).
$$
Moreover, we have a spectral sequence
$$
E_1^{pq}=H^q(X_p, \varepsilon^*_p\calf)\Rightarrow H^{p+q}(S,\calf).
$$
This construction extends to objects of the bounded below derived category $D^+(S)$. For details, see e.g. \cite{conraddescent}, Theorem 6.11.

\subsection{Hypercoverings} The method of hypercoverings enables one to construct augmented simplicial objects satisfying cohomological descent.

To define hypercoverings, we first need the notion of (co)skeleta. For an integer $n\geq 0$ denote by $\Delta_n$ the full subcategory of $\Delta$ spanned by objects $[m]$ for $m\leq n$. An {\em n-truncated} simplicial object in a category $\mathcal C$ is a contravariant functor $\Delta_n\to {\mathcal C}$. These form a category ${\rm Simp}_n({\mathcal C})$. These notions have obvious augmented and cosimplicial variants.

For each $n\geq 0$ there is a natural functor
$$
{\rm sk}_n:\,{\rm Simp}({\mathcal C})\to {\rm Simp}_n({\mathcal C})
$$
induced by restriction of functors to $\Delta_n$. It is called the {\em skeleton functor}. When $\mathcal C$ admits finite inverse limits, these functors have right adjoints
$$
{\rm cosk}_n:\,{\rm Simp}_n({\mathcal C})\to {\rm Simp}({\mathcal C})
$$
called coskeleton functors. See e.g. \cite{conraddescent}, \S 3 for an exhaustive discussion. Given an object $X_\bullet$ in ${\rm Simp}_n({\mathcal C})$, the degree $p$ term ${\rm cosk}_n{(X_\bullet)}_p$ is given by the finite inverse limit $\limproj X_q$ indexed by maps $[q]\to [p]$ for $q\leq n$.

\begin{defi}\rm
Let $\mathcal C$ be the category of topological spaces, or the category of schemes equipped with  a Grothendieck topology. Assume given a class $P$ of morphisms stable under base change and composition and containing isomorphisms in $\mathcal C$. An augmented simplicial object $X_\bullet\to S$ in $\mathcal C$ is a $P$-hypercovering if for all $n\geq -1$ the adjunction maps
$$
X_\bullet \to {\rm cosk}_n({\rm sk}_n(X_\bullet))
$$
are given in degree $n+1$ by a map $$
X_{n+1} \to {\rm cosk}_n({\rm sk}_n(X_\bullet))_{n+1}
$$
that lies in $P$. (Here $S$ is in degree $-1$ by convention.)
\end{defi}

\begin{theo}\label{descent}  A $P$-hypercovering $X_\bullet\to S$ satisfies cohomological descent in each of the following cases.
\begin{enumerate}
\item $\mathcal C$ is the category of topological spaces or the category of schemes equipped with  a Grothendieck topology and $P$ is the class of surjective covering maps in this topology.
\item $\mathcal C$ is the category of topological spaces and $P$ is the class of proper surjective maps.
\item $\mathcal C$ is the category of schemes equipped with the \'etale topology, and $P$ is the class of proper surjective maps, provided we restrict to torsion sheaves.
\item In the previous situation we may also take for $P$ the class of maps that are composites of proper surjective maps and \'etale coverings.  In the topological situation we may take for $P$ the maps that are composites of proper surjective maps and open coverings.
\end{enumerate}
\end{theo}

For the proof, see the references cited above, more specifically \cite{conraddescent}, Theorems 7.7 and 7.10.

Hypercoverings can also be used to compute sheaf cohomology by a generalization of the \v Cech method.

\begin{theo}\label{limitdescent} Let $S$ be a topological space (resp. a scheme), and ${\mathcal C}_S$ the category of spaces (resp. schemes) over $S$. Assume ${\mathcal C}_S$ is equipped with a Grothendieck topology which in the topological case is the classical one, and let $P$ be the class of surjective covering maps.

The system of $P$-hypercoverings $\varepsilon:\,X_\bullet\to S$ form a filtered inverse system indexed by the homotopy classes of simplicial maps between them. Given an abelian presheaf $\calf$ on ${\mathcal C}_S$ with associated sheaf $\calf^\sharp$, we have canonical isomorphisms
$$
H^i(S, \calf^\sharp)\cong \limind H^i(C(\calf(X_\bullet)))
$$
for all $i>0$, where the direct limit is taken over the dual of the above inverse system.
\end{theo}

For the proof, see \cite{sga4}, Expos\'e V, Theorems 7.3.2 and  7.4.1. The theorem holds more generally for bounded below complexes of presheaves.

\end{document}